\documentclass[12pt]{article} 
\usepackage[utf8]{inputenc}

\usepackage[margin=1in]{geometry}
\usepackage[shortlabels]{enumitem}
\usepackage{amsmath,amscd,amsthm,amsfonts,amssymb}
\usepackage{mathrsfs}
\usepackage[initials]{amsrefs}

\usepackage[english,french]{babel}
\numberwithin{equation}{section} % mettre les num\'eros des sections dans les num\'erotations de formules

\RequirePackage{color}
\definecolor{myred}{rgb}{0.75,0,0}
\definecolor{mygreen}{rgb}{0,0.5,0}
\definecolor{myblue}{rgb}{0,0,0.65}

\usepackage{hyperref}
\hypersetup{citecolor=blue}
\theoremstyle{plain}
\newtheorem{theorem}{Th\'eor\`eme}[section]

\newtheorem{lemma}[theorem]{Lemme}
\newtheorem{rem}[theorem]{Remarque}
\newtheorem*{lemma*}{Lemme}
\newtheorem*{proposition*}{Proposition}
\newtheorem{conjecture}[theorem]{Conjecture}

\newtheorem*{truefact*}{Fact}

\theoremstyle{definition}

\theoremstyle{remark}

\newcommand{\bb}[1]{\expandafter\newcommand\expandafter{\csname #1\endcsname}{{\mathbb {#1}}}} % turns \foo into \mathbb{foo}

\bb C
\bb R
\bb Q
\bb Z
\bb N

\bb F 
 
\newcommand{\e}{\mathrm{e}}

\newcommand{\bfm}{\mathbf m}

\renewcommand{\leq}{\leqslant}
\renewcommand{\geq}{\geqslant}
\renewcommand{\phi}{\varphi}
\renewcommand{\d}{{\rm d}}

\title{R\'epartition conjointe  de trois nombres premiers et applications}
\author{R\'egis de la Bret\`eche }

\begin{document} 
\maketitle  

\bigskip
 {\hfill {\`A la m\'emoire de Nicolas Bergeron}}
\bigskip
 \renewcommand{\abstractname}{Abstract}
\begin{abstract}  
We prove asymptotic results for the singular series associated to the distribution of three primes. Assuming a quantitative version of Hardy and Littlewood's conjecture on prime 3-tuples, we deduce an asymptotic formula related to the joint  distribution of three primes.
This improves recent  results  of  Kuperberg and completes  results by Montgomery and Soundararajan. Following the Montgomery and Soundararajan  approach, we derive conjectural applications to the distribution of  primes in short intervals.
\end{abstract}

\section{Introduction}

%La conjecture des $k$-uplets de nombres premiers par Hardy et Littlewood pr\'edit la r\'epartition des $k$ uplets de nombres premiers s\'epar\'es par des entiers donn\'es. Ainsi si $k=2$, elle conjecture quel est le nombre de couples de nombres premiers jumeaux (dont la diff\'erence vaut $2$). Malgr\'e les avanc\'ees r\'ecentes, elle est encore hors de port\'ee mais permet de pr\'edire des r\'esultats importants sur les nombres premiers. 

%En 2004, sous la conjecture de Hardy et Littlewood, Montgomery et Soundararajan ont \'etabli  une relation asymptotique pour les moments  $$M_k(X,h):=\frac1X\sum_{1\leq n\leq X} \big(\psi(n+h)-\psi(n)-h\big)^k$$ o\`u  $ \psi(x)$ est la fonction sommatoire de la fonction de von Mangoldt $\Lambda.$ Pour $k$ pair, cela fournit un \'equivalent. Nous \'etudions le cas impair et en particulier le cas $k=3$. Nous pr\'esenterons les nouvelles techniques d\'evelopp\'ees pour le cas $k=3$ pour obtenir un \'equivalent et expliquerons les heuristiques dans le cas $k$ impair.

La conjecture des $k$-uplets de nombres premiers par Hardy et Littlewood \cite{HL23} s'\'ecrit$$\qquad\qquad\qquad\qquad\sum_{1\leq n\leq x} \prod_{j=1}^k \Lambda(n+d_j)=(\mathfrak S(\mathcal D)+o(1))x \qquad\qquad\qquad\qquad {(x\to+\infty)}$$
avec $\mathcal D:=\{ d_1,\ldots , d_k\}$ un ensemble de $k$ entiers distincts et $\mathfrak S(\mathcal D)$ la s\'erie singuli\`ere d\'efinie par $$
\mathfrak S(\mathcal D):=\prod_{p}\Big(1-\frac 1p\Big)^{-k} \Big( 1-\frac{\nu_p(\mathcal D)}{p}\Big),
$$
avec $\nu_p(\mathcal D)$ est le nombre de classes de r\'esidu modulo $p$ parmi les \'el\'ements de $\mathcal D$. Ce produit eul\'erien est bien convergent puisque pour les grandes valeurs de nombre premier $p$ nous avons~$\nu_p(\mathcal D)=k$. De plus, $\mathfrak S(\mathcal D)$ est nul s'il existe un nombre premier $p\leq k$ tel que~$\nu_p(\mathcal D)=p$. 
Nous renvoyons \`a \cite{BFT23} pour la pertinence de ce mod\`ele et les implications de cette conjecture.

D'apr\`es \cite{MS04}*{eq. (2) {\color{black} et  Lemme 3}}, la s\'erie singuli\`ere v\'erifie aussi la relation $$
\mathfrak S(\mathcal D)=\sum_{\substack{q_1,\ldots,q_k\\ q_j\geq 1}}\Big( \prod_{j=1}^k\frac{\mu (q_j)}{\phi(q_j)}\Big)
\sum_{\substack{a_1,\ldots,a_k\\ 1\leq a_j\leq q_j \\ (a_j,q_j)=1\\ \sum_j a_j/q_j\in \mathbb Z}}\e\Big( \sum_{j=1}^k\frac{a_j d_j}{q_j}\Big),
$$
avec $\e (t):=\e^{2\pi i t}$.
Nous notons $\Lambda_0 =\Lambda -1$ la fonction arithm\'etique de moyenne asymptotiquement nulle construite \`a partir de la fonction de von Mangoldt $\Lambda$. La conjecture de Hardy et Littlewood s'\'ecrit alors
$$\qquad\qquad\qquad\qquad  \sum_{1\leq n\leq x} \prod_{j=1}^k \Lambda_0(n+d_j)=(\mathfrak S_0(\mathcal D)+o(1))x \qquad\qquad\qquad\qquad {(x\to+\infty)}$$
avec $ \mathfrak S_0(\mathcal D)$  d\'efinie (voir \cite{MS04}*{eq. (9)})  par $$
\mathfrak S_0(\mathcal D)=\sum_{\substack{q_1,\ldots,q_k\\ q_j\geq 2}}\Big( \prod_{j=1}^k\frac{\mu (q_j)}{\phi(q_j)}\Big)
\sum_{\substack{a_1,\ldots,a_k\\ 1\leq a_j\leq q_j \\ (a_j,q_j)=1\\ \sum_j a_j/q_j\in \mathbb Z}}\e\Big( \sum_{j=1}^k\frac{a_j d_j}{q_j}\Big).
$$
Le th\'eor\`eme des nombres premiers fournit $\mathfrak S_0(\{d\})=0$ pour tout $d\geq 1.$

Consid\'erant 
$$R_k(h):=\sum_{\substack{d_1,\ldots,d_k\\ 1\leq d_j\leq h\\ d_j \rm{distincts}}}\mathfrak S_0(\{ d_1,\ldots , d_k\}),$$
Montgomery et Soundararajan \cite{MS04} ont \'etabli   pour tout $k\geq 1$ fix\'e et $h\geq 1$ l'estimation asymptotique
\begin{equation}
\label{est Rkh}    
R_k(h)=\mu_k \big(-h\log h+Ah\big)^{k/2} +O_k\big(h^{k/2-1/(7k)+\varepsilon}\big),
\end{equation}
o\`u $A:= 2-\gamma-\log (2\pi) $ et $\mu_k=(2k')!/2^{k'}k'! $ si $k=2k'$ est pair et $ 0$ si $k$ est impair. Pour cela, ils utilisent l'article important de Montgomery et Vaughan~\cite{MV86}. Rappelons aussi que, d'apr\`es \cite{G76}, la somme relative \`a $\mathfrak S$ et non $\mathfrak S_0$ est \'equivalente \`a $h^k$ pour tout $k\geq 1.$ Cela s'explique notamment dans le fait que $\mathfrak S(\mathcal D)$ est toujours positif ou nul.

Lorsque $k$ est impair, la relation \eqref{est Rkh} ne fournit pas d'\'equivalent asymptotique. Lemke Oliver et Soundararajan \cite{LOS16} et Kuperberg~\cite{K21} conjecturent  lorsque $k$ est impair fix\'e et $h\geq 2$ la relation
\begin{equation}
    \label{conjK}
R_k(h)\asymp h^{(k-1)/2}(\log h)^{(k+1)/2} \end{equation}et Kuperberg~\cite{K21} montre la majoration $ R_3(h)\ll h(\log h)^{5}.$  
De plus, Bloom et Kuperberg \cite{BK23} montrent  pour tout $k\geq 3$ impair fix\'e, $h\geq 2$, la majoration 
$$R_k(h)\ll h^{(k-1)/2}(\log h)^{O_k(1)}. $$

Le but de cet article est de d\'eterminer une \'evaluation asymptotique de $R_3(h)$. Nous notons $\log\log=\log_2.$
Nous obtenons l'estimation suivante pour laquelle nous n'avons pas cherch\'e \`a obtenir dans le terme d'erreur le meilleur exposant possible de $\log_2 h.$ 
\begin{theorem}\label{mainth}
    Lorsque $h\geq 3$, nous avons 
$$R_3(h)= %\textcolor{red}
{\frac 94} h(\log h)^2\Big(1+O\Big( \frac{(\log_2 h)^{12}}{\log h}\Big)\Big).$$
\end{theorem}

R\'ecemment, Leung \cite{L24} a \'etudi\'e des g\'en\'eralisations de $R_k(h)$ o\`u les conditions $d_j\equiv a \mod r$ sont rajout\'ees \`a la somme d\'efinissant $R_k(h)$. Il est clair que nos m\'ethodes fonctionnent dans ce cadre.

D\'eterminer un \'equivalent asymptotique de $R_3(h)$ est  la premi\`ere \'etape cruciale pour estimer asymptotiquement $R_{2k+1}(h)$ pour tout~$k\geq~1.$
L'estimation de $R_3(h)$ du Th\'eor\`eme~\ref{mainth} permet de conjecturer l'ordre asymptotique des moments d'ordre impair. Nous retrouvons la conjecture \eqref{conjK} et nous la pr\'ecisons.

\begin{conjecture}
    \label{conj}
 Soit $k\geq 1$.   Lorsque $h $ tend vers l'infini, nous avons
$$R_{2k+1}(h)\sim (-1)^{k-1}r_{2k+1} h^k(\log h)^{k+1},$$
avec 
$$r_{2k+1}:=%\textcolor{red}
{\frac 34} (2k+1)k \mu_{2k}.$$
\end{conjecture}

Expliquons maintenant l'intuition qui nous permet d'\'enoncer cette conjecture. Dans le cas pair d'apr\`es \cite{MS04}, les cas dominants correspondent aux deux $2k$-uplets $(a_1, \ldots,a_{2k}) $ et $(q_1, \ldots,q_{2k}) $ tels que $q_{i}=q_{i+k}$ et $a_{i+k}=q_{i}-a_i$ pour tout $1\leq i\leq k$ et tous les autres cas obtenus par permutation des indices. Dans \cite{MS04}, cela permet d'obtenir lorsque $h $ tend vers l'infini l'\'equivalent
$$R_{2k} (h)\sim \mu_{2k} R_2(h)^k\sim \mu_{2k} (-h\log h)^k.$$
Pour une estimation de $R_2$ \'etablie dans \cite{MS04}, on pourra voir {\it infra} \eqref{estR2}.
Dans le cas impair, les cas dominants devraient correspondre aux   deux $(2k+1)$-uplets $(a_1, \ldots,a_{2k+1}) $ et $(q_1, \ldots,q_{2k+1}) $ tels que 
$a_{2k-1}/q_{2k-1}+a_{2k}/q_{2k}+a_{2k+1}/q_{2k+1}\in \mathbb{Z}$ et $q_{i}=q_{i+k-1}$ et $a_{i+k-1}=q_{i}-a_i$ pour tout $1\leq i\leq k-1$ et tous les autres cas obtenus par permutation des indices.
 Cela permet d'esp\'erer lorsque $h $ tend vers l'infini l'\'equivalent asymptotique 
$$R_{2k+1} (h)\sim \binom{2k+1}{3}\mu_{2k-2} R_3(h)R_2(h)^{k-1}.$$
Le Th\'eor\`eme \ref{mainth} et l'estimation \eqref{estR2}    permettent de formuler notre conjecture.

Notons enfin qu'il n'y a pas d'obstacle conceptuel pour que la m\'ethode de la preuve du Th\'eor\`eme \ref{mainth} ne fonctionne pas pour estimer $R_{2k+1}(h)$ lorsque $k\geq 2.$ En revanche, le nombre $2^{2k+1}-2k-2$ de variables qui param\`etrent les d\'enominateurs $q_j$ devient vite trop grand pour \^etre facilement appr\'ehend\'e. 

Nous d\'eveloppons maintenant une cons\'equence importante du Th\'eor\`eme \ref{mainth}. %Conditionnellement \`a une forme quantitative de la conjecture de Hardy et Littlewood, n
 Nous estimons lorsque $k=3$ le moment
$$M_k(X,h):=\frac1X\sum_{1\leq n\leq X} \big(\psi(n+h)-\psi(n)-h\big)^k$$
o\`u 
$$\psi(x):=\sum_{1\leq n\leq x}\Lambda(n).$$ 
Conditionnellement \`a une forme quantitative de la conjecture de Hardy et Littlewood, 
Montgomery et Soundararajan \cite{MS04} \'etablissent une estimation asymptotique de $M_k(X,h)$ lorsque $k$ est pair. Lorsque $k$ est impair, ils n'obtiennent qu'une majoration. Nous reprenons les hypoth\`eses de \cite{MS04}.

\begin{theorem}\label{consth}
Soit $E_k(x;{\cal D})$ d\'efini par la relation
\begin{equation}
    \label{hypo}
\sum_{1\leq n\leq x} \prod_{j=1}^k \Lambda(n+d_j)= \mathfrak S(\mathcal D) x+E_k(x;{\cal D} ),\end{equation}
lorsque ${\cal D}=\{ d_1,\ldots,d_k \} $ avec les $d_j$ distincts.
Supposons qu'uniform\'ement lorsque $1\leq k\leq K,$
$0\leq x\leq X$ et ${\cal D}$ un ensemble de $k$ entiers $d_j$ distincts tous $\leq h$
nous ayons $|E_k(x;{\cal D} )|\leq E_K(X,h) $ avec $E_K(X,h)\geq \sqrt{X} $.
Alors,  avec $B:=1-\gamma-\log (2\pi)$ o\`u $\gamma$ est la constante d'Euler, pour tout $\varepsilon>0$, nous avons lorsque $h\leq X^{1/3}$
\begin{equation}\label{estM3(X,h)}
    \begin{split}
     M_3(X,h)= &%\textcolor{red}
{\frac hX\int_0^X \big(  \tfrac 94(\log h)^2-3(\log h-B)(\log x) +(\log x)^2\big)\d x  }
     %h \big( \textcolor{red}{3% \tfrac 94 }(\log h)^2-3(\log h-B)(\log X) +(\log X )^2\big) 
\cr&+O\big(  h(\log h)(\log_2 h)^{12}+ h (h+ \log X)^2  E_3(X,h)/X+h^{1/2+\varepsilon} \log X  \big).
 \end{split}
\end{equation}
\end{theorem}

Le Th\'eor\`eme \ref{consth} fournit un \'equivalent asymptotique de $M_3(X,h)$ lorsque, par exemple, il existe $\varepsilon>0$ tel que 
$(h+ \log X)^2  E_3(X,h)\ll X(\log X)^{2-\varepsilon}$. L'hypoth\`ese \eqref{hypo} avec la majoration $E_3(X,h)\ll_\varepsilon X^{1/2+\varepsilon}$ pour tout $\varepsilon>0$ implique l'hypoth\`ese de Riemann.

Sous l'hypoth\`ese de Riemann, il est \'etabli dans \cite{BF21} une minoration d'une forme pond\'er\'ee de $(-1)^kM_k(X,h) $ mais pour des valeurs de $h$ tr\`es proche de $X$ (par exemple dans l'intervalle $X/(\log X)^{k/2-1-\varepsilon} <h\leq X$). Les deux r\'esultats n'ont donc pas d'intersection de domaine de validit\'e. 

\smallskip
Sous l'hypoth\`ese de Hardy et Littlewood avec $E_{2K} (X,h)\ll X^{1/2+\varepsilon}$ pour tout $\varepsilon>0$, Montgomery et Soundararajan \cite{MS04} \'etablissent l'estimation valable dans le domaine  $\log X\leq h\leq X^{1/2K}$
$$M_{2K}(X,h)=\mu_{2K} \frac{h^K}{X}\int_0^X (\log x/h+B)^{K}\d x
+O\bigg(  (\log X)^K h^K\Big(\frac{h}{\log X}\Big)^{-1/16K}
+h^KX^{-1/2+\varepsilon}\bigg).$$

La Conjecture \ref{conj} et celle de Hardy et Littlewood nous permettent d'obtenir une expression conjecturale de $M_{K}(X,h)$ pour tout $K$ impair. Dans le cas $K=3,$ nous retrouvons une forme affaiblie du Th\'eor\`eme \ref{consth}.
\begin{theorem}
     \label{conjMoment} Soit $K\geq 1$.
Sous la Conjecture \ref{conj} et celle de Hardy et Littlewood \'enonc\'ee au Th\'eor\`eme \ref{consth}, 
 lorsque $h $ et $X$ tendent vers l'infini avec $\log X\leq h\leq X^{1/(2K+1)}$, nous avons
\begin{equation}\label{MKth} \begin{split}
     M_{2K+1} (X,h)=&  (\tfrac 13+o(1))\mu_{2K+2}K h^{K} 
\big(\log  (X/h) \big)^{K-1}% \cr&\quad\times
\big( %\textcolor{red}
 \log (X/h^{3/2}) \big)^2\cr&+  O\big(  h^{2K+1}   E_{2K+1}(X,h)/X\big).\end{split}
\end{equation} \end{theorem}

Nous d\'etaillons maintenant la m\'ethode pour \'etablir le Th\'eor\`eme \ref{mainth} en   mettant en valeur les id\'ees importantes.

Dans toute la suite, $h$ sera un entier $\geq 2$. Nous consid\'erons
\begin{equation}E_h(\alpha):=\sum_{1\leq d\leq h} \e (d\alpha)=\e(\alpha)\frac{\e(h\alpha)-1}{\e(\alpha)-1}\qquad \quad (\alpha\in \mathbb R\smallsetminus \mathbb Z).\label{defEhalpha} \end{equation}
Nous introduisons la somme $V_3(h)$  d\'efinie par
\begin{equation}
    \label{est V3}V_3(h) = 
\sum_{\substack{P(q_1 q_2 q_3)\leq h^3\\ q_j\geq 2}} \prod_{j=1}^3 \frac{\mu(q_j)}{\phi(q_j)} \sum_{\substack{a_1,a_2,a_3\\ 1\leq a_j\leq q_j\\ (a_j,q_j)=1\\ \sum_j a_j/q_j\in \Z}}\prod_{j=1}^3 E_h\Big(\frac{a_j}{q_j}\Big)
\end{equation} qui correspond \`a la somme $R_3(h)$ avec la restriction suppl\'ementaire $P(q_1 q_2 q_3)\leq h^3$ mais sans la condition $d_j$ distincts. Ici, et dans la suite, $P(n)$ d\'esigne le plus grand facteur premier d'un entier $n$ avec la convention $P(1)=1$.
Nous posons $q:=\prod_{p\leq h^3} p$.
Suivant \cite{K21}*{lemmes~2.4 et~2.5}, pour tout $\varepsilon>0$, nous avons 
\begin{equation} \label{estR3h}\begin{split}
R_3(h)   =&V_3(h)-h\Big(\frac{q}{\phi(q)}\Big)^2+3h\big(  \log h- B\big)\frac{q}{\phi(q)}\cr&+h\big( -6 \log h+6B  +4\big)+O\big(h^{1/2+\varepsilon} \big)
. \end{split} \end{equation} 
L'estimation de $R_3(h)$ se r\'eduit donc \`a celle de $V_3(h)$.

Comme dans \cite{K21}, pour les entiers $q_j$ de la somme \eqref{est V3}, nous utilisons la d\'ecomposition  
\begin{equation}
    \label{def qj}q_1=gyz,\quad q_2=gxz,\quad q_3=gxy,\end{equation}
avec $g,x,y,z$ des entiers naturels premiers entre eux. La condition $\sum_j a_j/q_j\in \Z$ est alors \'equivalente \`a \begin{equation}
    \label{congruence avec aj}
a_1x+a_2y+a_3z\equiv 0(\bmod gxyz).\end{equation}                
Nous pouvons donc \'ecrire 
\begin{equation}
    \label{calculV3}
V_3(h)=\sum_{\substack{P(gxyz)\leq h^3\\ gyz, gxz,gxy\geq 2 }}\frac{\mu(g)\mu(gxyz)^2}{\phi(g)\phi(gxyz)^2}V(g,x,y,z;h),
\end{equation}
avec
\begin{equation}
    \label{def Vh}
V(g,x,y,z;h):=
\sum_{\substack{a_1,a_2,a_3\\ 1\leq a_j\leq q_j\\ (a_j,q_j)=1\\ \sum_j a_j/q_j\in \Z}}\prod_{j=1}^3 E_h\Big(\frac{a_j}{q_j}\Big).
\end{equation} 

Comme cela a \'et\'e remarqu\'e par Bloom et Kuperberg \cite{BK23}, la condition 
$\sum_j a_j/q_j\in \Z$ intervient dans de nombreux probl\`emes importants. Nous renvoyons \`a \cite{BK23} pour une bibliographie compl\`ete.
La difficult\'e principale de l'estimation de la somme $V_3(h)$ provient du caract\`ere oscillant du facteur $\mu(g)$. Dans \cite{K21} et \cite{BK23}, ce terme est major\'e en valeur absolue ce qui emp\^eche d'obtenir un majorant meilleur que $h(\log h)^4. $

La premi\`ere \'etape de la d\'emonstration du Th\'eor\`eme  \ref{mainth} consiste \`a utiliser les sym\'etries du probl\`eme. Pour cela, nous introduisons les quantit\'es
\begin{equation}
    \label{defXYZ}
X=2^j,\quad Y=2^k,\quad Z=2^\ell,\end{equation}
et nous scindons alors le domaine de la sommation en des sous-domaines d\'efinis par
\begin{equation}\label{leqXYZ}
X\leq x< 2X, \quad Y\leq y< 2Y, \quad Z\leq z< 2Z. \end{equation}
Ainsi lorsque $(x,y,z)$ est fix\'e, nous avons
\eqref{defXYZ} avec 
$$
j=\Big\lfloor\frac{\log x}{\log 2} \Big\rfloor ,\qquad 
k=\Big\lfloor\frac{\log y}{\log 2} \Big\rfloor ,\qquad 
\ell=\Big\lfloor\frac{\log z}{\log 2} \Big\rfloor .
$$
 Nous nous pla\c cons  dans un ensemble de comptage tel que $x\geq y\geq z$ ce qui implique $X\geq Y\geq Z$.

Nous \'ecrivons aussi
$$Q_1:=gYZ,\quad Q_2:=gXZ,\quad Q_3:=gXY,$$ de sorte que 
$$Q_1\leq q_1< 4 Q_1, \quad Q_2\leq q_2< 4 Q_2, \quad Q_3\leq q_3<4 Q_3.$$
 Au  Lemme~\ref{lemme decoupage}, pour des raisons de  sym\'etries, nous pourrons alors  nous placer dans le cas o\`u 
$x\geq y\geq z$ et $X\geq Y\geq Z$. La partition du domaine en $(x,y,z)$ en des cubes dyadiques est un des points importants de notre m\'ethode.

Pour pouvoir estimer pr\'ecis\'ement $V_3(h)$, il est donc important d'estimer asymptotiquement la somme int\'erieure $V(g,x,y,z;h)$ apparaissant dans le membre de droite de \eqref{est V3}. Si on somme selon les variables $a_1$ et $a_2$, la congruence \eqref{congruence avec aj} devient $a_1x+a_2y \equiv 0(\bmod  z)$. 
Les entiers $a_j$ contribuant principalement \`a la somme v\'erifient $a_j\approx q_j/h\approx Q_j/h.$ Ainsi si on fixe $a_1$
et on consid\`ere la congruence modulo $z$ v\'erifi\'ee par $a_2,$ on obtient une estimation pr\'ecise lorsque $q_2/hz=gx/h>T$ et $gyz=q_1/h>T$ o\`u $T$ est un param\`etre suffisamment grand ($T$ sera pris \`a la fin de la d\'emonstration \'egal \`a $\exp\{  150(\log_2h)^2\}$). Cela sera fait au Lemme~\ref{lemme estS1} en utilisant  une  sommation d'Abel de dimension $2 $ (voir Lemme~\ref{somdabel}).

De m\^eme lorsque $q_1/h\leq 1/T, $ la contribution  attendue  est petite puisque heuristiquement 
$a_1\approx Q_1/h.$ Cela sera montr\'e au Lemme~\ref{lemme estS2}.

Lorsque $a_1$ et $a_2$ ne varient pas dans des intervalles suffisamment grands, nous sommons par rapport aux variables $x$ et $y$ satisfaisant \`a la congruence $a_1x+a_2y \equiv 0 (\bmod \, z)$. Si $x\geq zT^4$ par exemple, cela devrait \^etre suffisant. Ici apparait un des apports les plus importants de l'article. On ne doit pas sommer la fonction constante mais une fonction arithm\'etique proche de  $\mu^2$ la fonction indicatrice des entiers sans facteur carr\'e. Or l'estimation de Hooley~\cite{Hoo75}  qui s'\'ecrit % $$\sum_{\substack{x\leq X\\ x\equiv a\bmod z }}\mu^2(x)$$
\begin{equation}\label{estHooely}
\sum_{\substack{1\leq n\leq X\\ n\equiv a \bmod z}}\mu(n)^2=     \frac {6X}{\pi^2{\color{black}z} }\prod_{p\mid z} (1-1/p^2)^{-1}+O\Big( \Big( \frac XZ\Big)^{1/2}+ Z^{1/2+\varepsilon}\Big)
\end{equation} 
ne fournit une estimation non triviale que  lorsque $Z\leq X^{2/3-\varepsilon}$. %(voir l'estimation de Hooley \cite{Hoo75} du Lemme~\ref{lemme Hooley}). 
Notons que l'estimation~\eqref{estHooely} pr\'ecise le r\'esultat de \cite{P58} et qu'un plus grand domaine d\'efini par $Z\leq X^{25/36-\varepsilon}$ a \'et\'e obtenu dans \cite{MN17}.
Heureusement, des r\'esultats en moyenne obtenus par Le Boudec  \cite{LB18} rendus explicites pour l'occasion (voir Lemme~\ref{lemme Le Boudec}, puis les Lemmes~\ref{lemme M2}, \ref{lemme sumSS2}, \ref{lemme sumW}) permettent de surmonter cette difficult\'e pour obtenir des estimations lorsque $X\gg ZT^4$ avec, l\`a encore, $T$   un param\`etre suffisamment grand.

Les oscillations de $\mu(g)$ permettent de se placer dans le cas o\`u $g$ est petit disons $g\leq T $ (voir Lemme \ref{lemme estS5}). Le reste du domaine \`a couvrir est alors suffisamment petit pour conclure. Nous renvoyons aux Lemmes  \ref{lemme estS6} et \ref{lemme estS7} o\`u sont \'etablis des majorations suffisantes.

La   deuxi\`eme section est consacr\'ee \`a l'\'enonc\'e et \`a la d\'emonstration de r\'esultats techniques n\'ecessaires \`a l'estimation de $V_3(h)$. On pourra l'omettre en premi\`ere lecture.
Ces r\'esultats sont soit obtenus par des m\'ethodes classiques, soit cons\'equences des travaux de Le Boudec sur l'\'equir\'epartition des entiers sans facteur carr\'e dans les progressions arithm\'etiques. 
%Notons aussi le Lemme \ref{lemme s(T)} qui permet d'estimer des sommes nouvelles dans ce contexte.

 \section{Lemmes techniques}
\subsection{Sommation d'Abel en deux variables et majorations d'int\'egra\-les}
Lorsque $f$ est une fonction arithm\'etique en deux variables, notons
$$S(x_1,x_2;f):=\sum_{\substack{1\leq n_1\leq x_1\\ 1\leq  n_2\leq x_2}}f(n_1,n_2).$$ Nous \'enon\c cons une formule de  sommation d'Abel de dimension $2.$
\begin{lemma}\label{somdabel}
    Soient $f$ une fonction arithm\'etique en deux variables et $g$ une fonction $C^1$ sur $[0,+\infty[^2.$ Alors pour tout $x_1,x_2\geq 1,$ on a la formule
 \begin{equation}\label{eqsomdabel}
     \begin{split}
%S(x_1,x_2;fg)&=g(x_1,x_2)S(x_1,x_2;f)-\int_1^{x_1} S(t_1,x_2;f)\partial_1 g(t_1,x_2)\d t_1\\ &\quad -\int_1^{x_2} S(x_1,t_2;f)\partial_2 g(x_1,t_2)\d t_2+\int_1^{x_1}\int_1^{x_2}  S(t_1,t_2;f)\partial_{1,2}^2 g(t_1,t_2)\d t_1\d t_2.
S(x_1,\,   &  x_2 ;fg) =g(x_1,x_2)S(x_1,x_2;f)-\int_0^{x_1} S(t_1,x_2;f)\partial_1 g(t_1,x_2)\d t_1\\ &\quad -\int_0^{x_2} S(x_1,t_2;f)\partial_2 g(x_1,t_2)\d t_2+\int_0^{x_1}\int_0^{x_2}  S(t_1,t_2;f)\partial_{1,2}^2 g(t_1,t_2)\d t_2\d t_1.
     \end{split}\end{equation}
De plus lorsque $M$ est de classe $C^1$ sur $[0,+\infty[^2 $ satisfaisant $M(t_1,0)= M(0,t_2)=0$ pour tout $t_1,t_2\geq 0$, nous avons
     \begin{equation}\label{somdabellisse}
     \begin{split}
 %  (gM)&(x_1,x_2)-\int_1^{x_1} (M \partial_1 g)(t_1,x_2)\d t_1  -\int_1^{x_2}  (M\partial_2 g)(x_1,t_2)\d t_2+\int_1^{x_1}\int_1^{x_2}   (M\partial_{1,2}^2 g)(t_1,t_2)\d t_1\d t_2 \cr&=
% (gM)(1,1)+  \int_1^{x_1} (g \partial_1 M)(t_1,1)\d t_1+\int_1^{x_2} %(g\partial_2 M)(1,t_2)\d t_2  +\int_1^{x_1} \int_1^{x_2}  (g \partial_{1,2}^2 M)(t_1,t_2)\d t_1\d t_2.
 (gM) (x_1,x_2)&-\int_0^{x_1} (M \partial_1 g)(t_1,x_2)\d t_1  -\int_0^{x_2}  (M\partial_2 g)(x_1,t_2)\d t_2\cr&+\int_0^{x_1}\int_0^{x_2}   (M\partial_{1,2}^2 g)(t_1,t_2)\d t_2\d t_1 =\int_0^{x_1} \int_0^{x_2}  (g \partial_{1,2}^2 M)(t_1,t_2)\d t_2\d t_1.
     \end{split}\end{equation}
\end{lemma}

\begin{rem}\label{rem<}
    Ce r\'esultat vaut aussi pour les sommes de la forme 
    $$\sum_{\substack{1\leq n_1< x_1\\ 1\leq  n_2< x_2}}f(n_1,n_2).$$ 
    %En effet, les valeurs des int\'egrales \eqref{eqsomdabel} ne sont pas affect\'ees par ce changement de d\'efinition de $S$. Il est facile d'observer qu'il en est de m\^eme pour $S(x_1,x_2 ;fg)-g(x_1,x_2)S(x_1,x_2;f).$
\end{rem}
\begin{proof}
   En appliquant successivement deux fois une sommation d'Abel, nous obtenons
\begin{align*}
S(x_1,x_2;fg)  &=
\sum_{\substack{1\leq n_1\leq x_1 }}
\Big( \sum_{\substack{1\leq n_2\leq x_2 }}f(n_1,n_2)g(n_1,x_2)-\int_1^{x_2} 
\sum_{\substack{1\leq n_2\leq t_2 }}f(n_1,n_2)\partial_2 g(n_1,t_2)\d t_2\Big)\cr&=
g(x_1,x_2)S(x_1,x_2;f)-
\int_1^{x_1} S(t_1,x_2;f)\partial_1 g(t_1,x_2)\d t_1
\\ &\quad -\int_1^{x_2} S(x_1,t_2;f)\partial_2 g(x_1,t_2)\d t_2+\int_1^{x_1}\int_1^{x_2}  S(t_1,t_2;f)\partial_{1,2}^2 g(t_1,t_2)\d t_2\d t_1.
\end{align*} Nous d\'eduisons la formule \eqref{eqsomdabel} en remarquant que les sommes $S(t_1,t_2;f)$ sont nulles lorsque $t_1\in [0,1[$ ou $t_1\in [0,1[$.
Nous ne d\'etaillons pas la preuve de \eqref{somdabellisse} qui consiste \`a faire deux int\'egrations par parties.
\end{proof}

Nous introduisons la quantit\'e 
\begin{equation}
    \label{defEh+}E_h^+(t)=\frac{1}{ 1/h+||t||}.
    \end{equation}
\begin{lemma}\label{calculmajI} Lorsque $h\geq 2,$
nous avons
\begin{align*}\int_{-1/2}^{1/2} E_h^+(t)\d t &\ll \log h,\qquad
\int_{-1/2}^{1/2} E_h^+(t_2)E_h^+(t_2+t_1)\d t_2 \ll     E_h^+(t_1)   \log h  
\end{align*}
et
\begin{align*}\int_{-1/2}^{1/2}\int_{-1/2}^{1/2}E_h^+(t_1)E_h^+(t_2)E_h^+(t_2+t_1)  \d t_1\d t_2 &\ll h,\cr
  \int_{-1/2}^{1/2}\int_{-1/2}^{1/2}  {|} t_1 {|} E_h^+(t_1)E_h^+(t_2)E_h^+(t_2+t_1)  \d t_1\d t_2&\ll (\log h)^2,\cr  
\int_{-1/2}^{1/2}\int_{-1/2}^{1/2}   {|}t_1t_2 {|}E_h^+(t_1)E_h^+(t_2)E_h^+(t_2+t_1)  \d t_1\d t_2&\ll {\log h}.\end{align*}
\end{lemma}

\begin{proof}
Pour d\'emontrer la deuxi\`eme majoration, nous distinguons le cas $|t_2|\leq \tfrac12|| t_1|| $ du cas $|t_2|> \tfrac12|| t_1|| $ et nous utilisons
$\int_{-1/2}^{1/2} E_h^+(t)\d t  \ll \log h$. Nous n'indiquons pas plus de d\'etails.
\end{proof}

%\begin{rem}\label{remtp}   Nous utiliserons aussi la formule      \begin{align*}  g(1,1)+\int_1^{x_1}\int_1^{x_2}   g(t_1,t_2)\d t_1\d t_2.&=g(x_1,x_2)x_1x_2-\int_1^{x_1}  t_1 x_2 \partial_1 g(t_1,x_2)\d t_1\\ &\quad -\int_1^{x_2}  x_1 t_2 \partial_2 g(x_1,t_2)\d t_2 +\int_1^{x_1}\int_1^{x_2}   t_1 t_2 \partial_{1,2}^2 g(t_1,t_2)\d t_1\d t_2. \end{align*}  En effet, elle sera utilis\'ee lorsque nous reporterons dans la formule \eqref{eqsomdabel} l'approximation de $S(t_1,t_2;f)$ par un terme proportionnel \`a $t_1t_2$. \end{rem}
%\begin{proof}  La formule \eqref{eqsomdabel} se d\'emontre de la m\^eme mani\`ere que la formule de la sommation d'Abel en une variable.\end{proof}

\subsection{Estimations de sommes de fonctions arithm\'etiques}

%Posons $$\sigma(X,q):=\min\Big\{ \frac{X}{q}+(\log 3X)^2,\Big(\frac{X}{q}+1\Big)(\log_2 3X)^2\Big\}.$$
%Nous notons aussi  $$C(q )=\prod_{p\nmid  q}\Big( 1+\frac{1}{p(p-1) }\Big).$$

%\begin{lemma}\label{lemme estS}
%    Soient $a,q,m\in \mathbb N$ tels que $(a,qm)=1$, $(q,m)=1$ et $X\geq 1$. Nous avons
%\begin{equation}\label{majS}S(X;q,a,m) \ll \sigma(X,q).\end{equation}\end{lemma}

%\begin{proof}[D\'emonstration du Lemme~\ref{lemme estS}] La majoration par le second terme du maximum d\'ecoule de la majoration $n/\phi(n)\ll \log_2 (3n)$.
%
%Pour \'etablir la majoration par le premier terme, nous \'ecrivons $$\frac{n^2}{\phi (n)^2}\ll \sum_{k\mid n} \frac{\mu^2(k)2^{\omega(k)}}{k}.$$ La somme \`a majorer est donc $$\ll\sum_{\substack{k\leq X\\ (k,q)=1 }} \frac{\mu^2(k)2^{\omega(k)}}{k} \sum_{\substack{n\equiv a\bmod q\\ k\mid n\\ 1\leq n\leq X}}1 \ll\sum_{\substack{k\leq X\\ (k,q)=1 }} \frac{\mu^2(k)2^{\omega(k)}}{k} \Big( \frac{X}{qk}+1\Big),$$ ce qui fournit ais\'ement la majoration annonc\'ee.  
%\end{proof}

%{\color{blue}  \begin{rem} Lorsque $\Re e (s)>1$, nous avons
%\begin{equation}\label{serieDir}  \sum_{n=1}^\infty \frac{\mu(n)^2n^2}{ \phi(n)^2n^s}  =\prod_{p}\Big( 1+\frac{ 1}{ (1-1/p)^2p^s}\Big)  =\zeta(s)\prod_{p}\Big( 1+\frac{ 2-1/p}{ (1-1/p)^2p^{s+1}}-\frac{ 1}{ (1-1/p)^2p^{2s}}\Big). \end{equation}
%Le troisi\`eme terme dans le facteur eul\'erien est probl\'ematique car il devient grand quand $s\in \mathbb R$ devient petit. Pour d\'epasser cet obstacle, nous utilisons un r\'esultat de Hooley.  \end{rem}}

Nous rappelons aussi le r\'esultat suivant dont nous n'indiquons pas la preuve.

\begin{lemma}\label{lemme sum xi} Soient $\xi$ une fonction arithm\'etique et $M_\xi\geq 0$ tels que 
$$\sum_{n=1}^\infty \frac{|\xi(n)|}{n^{1/2}}\leq M_\xi.$$
    Lorsque $T'>T\geq 1, $ nous avons
$$\sum_{T\leq n<T'} \frac{(1*\xi)(n)}{n}=\log \Big(\frac{T'}{T}\Big)\sum_{n=1}^\infty \frac{ \xi(n) }{n}+O\Big(\frac{M_\xi}{T^{1/2}}\Big) $$
et
$$\sum_{  1\leq n\leq T}  {(1*\xi)(n)} =T\sum_{n=1}^\infty \frac{ \xi(n) }{n}+O\big( {M_\xi T^{1/2}}\big).$$
\end{lemma}

Nous posons
\begin{equation}
    \label{defpsi}
\psi_{\alpha}(m):=\prod_{p\mid m} \Big( 1+\frac1{p^\alpha}\Big),\quad \phi_{\alpha}(m):=\prod_{p\mid m} \Big( 1-\frac{\alpha}{p}\Big)\qquad (\alpha\in\mathbb R,\, m\in {\mathbb N}).\end{equation}
Nous notons 
$${\cal L}(X):=\exp\sqrt{\log X},$$ et rappelons la notation \eqref{notationS*}. 
{\color{black}
Nous notons aussi
\begin{equation}\label{defkappa}
\kappa(m):=\prod_{p\nmid m}\Big( 1+\frac{ 2-1/p}{ (p-1)^2(1+1/p)   }\Big)
\end{equation}}

Nous consid\'erons
\begin{equation}\label{notationS}S(X;q,a,m):=\sum_{\substack{ 1\leq n\leq X\\ n\equiv a\bmod q\\ (n, m)=1}}\frac{\mu(n)^2n^2}{\phi (n)^2},\qquad S_{\mu^2}(X;q,a,m):=\sum_{\substack{1\leq n\leq X\\ n\equiv a\bmod q\\ (n, m)=1}} {\mu(n)^2 } ,\end{equation}
\begin{equation}\label{notationS*}S(X;q):=\sum_{\substack{ 1\leq n\leq X\\ (n, q)=1}}\frac{\mu(n)^2n^2}{\phi (n)^2},\qquad S_{\mu^2}(X;q):=\sum_{\substack{ 1\leq n\leq X\\ (n, q)=1}} {\mu(n)^2 } .\end{equation}

\begin{lemma}\label{lemme est S*}
   Lorsque $ m \in \mathbb N$ et  $X\geq 1$, nous avons
\begin{equation}\label{est S*}
S_{\mu^2}(X;m)=       \frac {6X}{\pi^2 \varphi_{-1} (m)}+O\big(  
\psi_{1/3}(m)^2X^{1/2}{\cal L}(X)^{-1}\big)   
\end{equation}
{\color{black}et \begin{equation}\label{est S*2}
S(X;m)=       \frac {6X\kappa(m)}{\pi^2 \varphi_{-1} (m)}+O\big(  
\psi_{1/3}(m)^2X^{1/2}{\cal L}(X)^{-1}\big)   
\end{equation}}
\end{lemma}

\begin{rem}
    Le facteur suppl\'ementaire ${\cal L}(X)^{-1}$ dans le terme d'erreur est crucial pour notre application. La valeur $\tfrac13$ dans le terme d'erreur peut \^etre ais\'ement remplac\'ee par tout r\'eel $<\tfrac12.$
\end{rem}

\begin{proof} Par souci de compl\'etude, nous indiquons les id\'ees de la preuve qui sont tr\`es classiques. 
Lorsque $m=1$, ce r\'esultat est bien connu et peut \^etre obtenu \`a partir d'une application de la m\'ethode de l'hyperbole compte-tenu de la relation de convolution $\mu^2(n)=\sum_{d^2\mid n } \mu(d)$ et d'une cons\'equence du {\it Th\'eor\`eme des nombres premiers} qui s'\'ecrit  $$
\sum_{n\leq X} \mu(n)\ll X{\cal L}(X)^{-3}.$$ 
Soit  $\lambda$ la fonction de Liouville d\'efinie par $\lambda(n):=(-1)^{ \Omega(n)}$. La relation de convolution 
\begin{equation}    \label{identiteconvolution}\mu^2(n) 1_{(n,m)=1}=\sum_{d\mid (n,m^\infty)} \lambda(d) \mu^2 (n/d), \end{equation}
fournit
$$S_{\mu^2}(X;m)=\sum_{d\mid m^\infty}\lambda(d)S_{\mu^2}(X/d;1).$$ Nous utilisons alors le r\'esultat obtenu lorsque $d\leq \sqrt{X}$ alors que lorsque  $  \sqrt{X}<d\leq X$ la majoration triviale $\leq X/d$ suffit. Nous avons
\begin{equation}    \label{majsumd}\sum_{d\mid  m^\infty } \frac1{d^{1/2}} \ll \psi_{1/2}(m)^2,\qquad \sum_{d\mid  m^\infty } \frac1{d^{1/3}}\ll \psi_{1/3}(m) \psi_{2/3}(m)\ll \psi_{1/3}(m)^2.\end{equation} 
La seconde majoration de \eqref{majsumd} et l'astuce de Rankin fournissent alors le r\'esultat recherch\'e. Nous n'indiquons pas plus de d\'etails.

{\color{black}
La seconde estimation est d\'eduite de la premi\`ere par convolution et les calculs suivants. 
Nous notons $f_1$ la fonction arithm\'etique d\'efinie par
\begin{equation}
    \label{identiteconvolutionavecf1}
\frac{\mu(n)^2n^2}{ \phi(n)^2 }  =(\mu^2*f_1)(n).\end{equation}
Lorsque $\Re e(s)>1$, un simple calcul fournit
 $$\sum_{n=1}^\infty \frac{\mu(n)^2n^2}{ \phi(n)^2n^s}
    =\prod_{p}\Big( 1+\frac{ 1}{ (1-1/p)^2p^s}\Big)
     =\prod_{p}\Big( 1+\frac{ 2-1/p}{ (1-1/p)^2p(p^{s}+1)}\Big)\sum_{n=1}^\infty \frac{\mu(n)^2 }{  n^s}.
$$ Ces calculs permettent facilement d'obtenir l'estimation \eqref{est S*2} \`a partir de \eqref{est S*}.}
\end{proof}
\subsection{Majoration de Le Boudec rendue explicite et ses cons\'equences}

Nous rappelons les notations \eqref{notationS} de $S_{\mu^2}$ et $S.$  Notons \begin{equation}\label{notationE}\begin{split}
 E_{\mu^2}(X;q,a,m)&:=S_{\mu^2}(X;q,a,m)-  \frac {6X}{\pi^2\varphi(q)\varphi_{-1} (mq)} ,\cr
 E(X;q,a,m)&:=S(X;q,a,m)-\frac {6X\kappa(mq)}{\pi^2\varphi(q)\varphi_{-1} (mq)}
%\prod_{p\nmid mq}\Big( 1+\frac{ 2-1/p}{ (p-1)^2(1+1/p)   }\Big)
,\end{split}\end{equation} {\color{black}o\`u $\kappa$ est d\'efinie en \eqref{defkappa}}
et 
\begin{equation}\label{notationM2}{\cal M}_{2,\mu^2}(X,q,m):=\sum_{\substack{1\leq a\leq q\\ (a,q)=1}} E_{\mu^2}(X;q,a,m)^2 ,\quad {\cal M}_{2}(X,q,m):=\sum_{\substack{1\leq a\leq q\\ (a,q)=1}} E (X;q,a,m)^2 .\end{equation}

Notons que lorsque $q$ un nombre premier v\'erifie $X^{5/11+\varepsilon}\leq q\leq X$ et $X/q$ tend vers l'infini, il existe des estimations asymptotiques de 
$E_{\mu^2}(X;q,1,1)$ (cf. \cite{GMRR21}).
Nous rendons explicite la majoration du th\'eor\`eme  1 de  Le Boudec \cite{LB18} valable pour tout entier $q$ qui est une am\'elioration de celles de  Blomer~\cite{Blo08} et de Nunes \cite{MN14}.

\begin{lemma}[\cite{LB18}]\label{lemme Le Boudec}
    Lorsque $1\leq q\leq X$ tel que
    $\mu(q)^2=1$, nous avons uniform\'ement 
    \begin{equation}
    \label{leboudecborne*}
{\cal M}_{2,\mu^2}(X,q,1)
%\ll (\log X)^2  \tau(q)2^{\omega(q)}\Big( \frac{q}{\phi(q)}\Big)^{3} (\log q)^2\Big( X^{1/2}q^{1/2}+ \frac{X }{q^{1/2}}\Big)
\ll
(\log X)^5   
\Big(2^{\omega(q)} X^{1/2}q^{1/2}+ 4^{\omega(q)}\frac{X }{q^{1/2}}\Big) 
\ll
X(\log 2 X)^5  2^{\omega(q)} % \tau(q)2^{\omega(q)}
.\end{equation}
{\color{black} Lorsque $ q> X$, nous avons
$${\cal M}_{2,\mu^2}(X,q,1)\ll X .$$}
\end{lemma} 

 En premi\`ere lecture, la preuve de ce lemme pourra \^etre omise.
 
\begin{proof} {\color{black}Pour \'etablir la seconde majoration, il suffit de remarquer que
$${\cal M}_{2,\mu^2}(X,q,1)\ll X+ \sum_{\substack{1\leq a\leq q\\ (a,q)=1}}\frac{X^2}{\phi(q)^2\phi_{-1}(q)^2} \ll X ,$$
puisque $ \phi(q)\phi_{-1}(q)^2 \geq q>X$.
}

Pour la premi\`ere majoration, nous nous contentons d'indiquer les changements \`a faire par rapport \`a la d\'emonstration de la proposition 1 de \cite{LB18}. Les notations utilis\'ees ici sont celles de \cite{LB18}. 
Les lemmes 2 et 3 de \cite{LB18} ont une version explicite \'etablie dans \cite{LB14}. Comme nous l'a indiqu\'e Shuntaro Yamagishi, il faut rajouter au lemme 2 de \cite{LB14} un facteur $\tau(q)\phi(q)/q $ \`a $E_1(q)$ alors que le lemme 9 de \cite{LB14} est juste sous la forme publi\'ee. D'apr\`es le lemme 2 de \cite{LB14} ainsi corrig\'e,  on peut remplacer  le majorant  $q^\varepsilon M(q,a_1,a_2)$ dans le lemme 2 de \cite{LB18} par $%(q/\phi(q))^3 
M(q,a_1,a_2)+(q/\phi(q))^2(\log 2q)^2\tau(q)$. %Pour cela, il suffit de remarquer que la somme $M(q,a_1,a_2)$ d\'efinie en \cite{LB18} est minor\'ee par un terme \'equivalent \`a~$(\log 2 q)^2.$ 
D'apr\`es le lemme 9 de~\cite{LB14},  on peut remplacer le majorant $q^\varepsilon (F_1F_2+q)$ dans le lemme~3 de~\cite{LB18} par $$2^{\omega(q)}(q/\phi(q))^3(\log 2q)^2 (\tau(q)F_1F_2+q) .$$
La modification de l'\'enonc\'e du lemme 2 de \cite{LB18} n'influe pas sur le r\'esultat.
Il suffit alors de suivre la d\'emonstration de la proposition 1 de \cite{LB18} en choisissant la m\^eme valeur du param\`etre $y$ et en utilisant la majoration
$${\rm card}\{ 1\leq \rho \leq q : \rho^2\equiv a \bmod q\}\ll 2^{\omega(q)}$$ uniforme pour $(a,q)=1$ et la borne $\big(  {q}/{\phi(q)}\big)^{3} (\log 2q)^2\ll (\log 2X)^3.$ Nous ne fournissons pas plus de d\'etails.
\end{proof}

%Addendum Le Boudec
%
% Montrons que $$\sum_{d\mid q} d\sigma_{-1}(d)\leq q(q/\phi(q))^2.$$ Par multiplicativit\'e, il suffit de montrer l'in\'egalit\'e sur une puissance $p^\nu$ d'un nombre premier. Lorsque $k\leq \nu$, nous utilisons l'in\'egalit\'e  $$\sigma_{-1}(p^k)=\frac{1-1/p^{k+1}}{1-1/p}\leq \frac{1-1/p^{\nu+1}}{1-1/p}.$$
 
Nous montrons le m\^eme type de majoration pour tous les ${\cal M}_{2,\mu^2}(X,q,m)$ puis pour
${\cal M}_{2}(X,q,m).$
Nous rappelons la d\'efinition \eqref{defpsi} de $\psi_\alpha$.

\begin{lemma}\label{lemme M2}  
    Lorsque $(q,m)=1$ et %$1\leq q\leq X$ tel que
    $\mu(q)^2=1$, nous avons
    \begin{equation}\label{majM2mu2m}
{\cal M}_{2,\mu^2}(X,q,m)\ll  \psi_{1/2}(m)^4 X(\log 2X)^5  2^{\omega(q)}
%\psi_{1/2}(m)^4X^{\varepsilon} \Big( X+ \frac{X^{5/3}}q\Big) 
    \end{equation}
et
\begin{equation}\label{majM2mu2m2}
{\cal M}_{2}(X,q,m)
 \ll  \psi_{1/2}(m)^4  %X^{\varepsilon} \Big( X+\frac{X^{5/3}}q\Big) .
 X(\log 2X)^5  2^{\omega(q)}.
    \end{equation}
\end{lemma}

\begin{rem}
 %   Nous pourrions avoir une borne de la forme  $$ \ll  \psi_{1/2}(m)^4  (\log X)^5  2^{\omega(q)} \Big( X^{1/2}q^{1/2}+ \frac{X }{q^{1/2}}\Big)$$
%o\`u le $1/2$ dans le $\psi_{1/2}(m)$ devrait \^etre diminu\'e. 
La borne triviale est en $O(X^2/q+q).$ Donc nous   obtenons ici une am\'elio\-ration de la borne triviale d\`es que $q\leq X/T$ avec $T$ plus grand qu'une puissance de $\log X$.
\end{rem}
\begin{proof}
%Nous pouvons m\^eme rendre explicite la d\'ependance en $\varepsilon $ dans la majoration de Le Boudec.
% Lorsque $q\leq X$ tel que $\mu(q)^2=1$, nous obtenons un majorant de l'ordre de   \begin{equation}\label{maj Le Boudec}\begin{split}{\cal M}_{2,\mu^2}(X,q,1)&\ll (\log X)^2  2^{\omega(q)}\Big( \frac{q}{\phi(q)}\Big)^{3} (\log q)^2\Big( X^{1/2}q^{1/2}+ \frac{X }{q^{1/2}}\Big)\cr&\ll(\log X)^5  2^{\omega(q)} \Big( X^{1/2}q^{1/2}+ \frac{X }{q^{1/2}}\Big)\ll   X(\log X)^5  2^{\omega(q)} .\end{split}\end{equation}
 En utilisant l'identit\'e de convolution \eqref{identiteconvolution}, nous obtenons
    \begin{align*}
 {\cal M}_{2,\mu^2} &(X,q,m) =\sum_{\substack{1\leq a\leq q\\ (a,q)=1}}        \Bigg| \sum_{\substack{d\mid m^\infty\\ (d,q)=1}}\lambda(d)\Big( \sum_{\substack{1\leq n\leq X/d\\ dn\equiv a \bmod q}}\mu(n)^2-\prod_{p\nmid q} \Big(1-\frac 1{p^2}\Big)\frac X{dq}\Big)\Bigg|^2
%\cr &\leq \sum_{\substack{d\mid m^\infty\\ (d,q)=1}} {\cal M}_{2,\mu^2}(X/d,q,1)
\cr
&\leq   2\Big(\sum_{\substack{d\mid m^\infty\\ 1\leq d\leq X}}d^{-1/2}\Big)\sum_{\substack{d\mid m^\infty\\ 1\leq d\leq X%{\color{black}/q}
}}d^{ 1/2}{\cal M}_{2,\mu^2}(X/d,q,1) 
\cr&\quad+  2% {\color{black}4}
\Big(  \sum_{\substack{d\mid m^\infty\\ d>X% {\color{black}X/q}
}}\frac { {\color{black}1}}{d } \Big)^2
 \prod_{p\nmid q} \Big(1-\frac 1{p^2}\Big)^2%{\color{black}\Big(}
 \frac {X^2\phi(q)}{q^2}% +{\color{red}\phi(q)\Big)}
 .
    \end{align*}
Avec la majoration \eqref{majsumd}, la m\'ethode de Rankin fournit
$$\Big|\sum_{\substack{d\mid m^\infty\\ d> X}}\frac {  {\color{black}1}}{d }\Big|\leq  \sum_{\substack{d\mid m^\infty }}\frac { 1}{d^{1/2}\sqrt{X} } \ll \frac{\psi_{1/2}(m)^2}{\sqrt{X}}. $$
En utilisant  \eqref{leboudecborne*}, nous obtenons finalement 
$$ {\cal M}_{2,\mu^2}(X,q,m)\ll\psi_{1/2}(m)^4
X(\log 2X)^5  2^{\omega(q)}. 
   $$  

Estimons maintenant    ${\cal M}_{2}(X,q,m)$.
 De \eqref{identiteconvolutionavecf1} puis \eqref{majM2mu2m}, nous d\'eduisons 
   \begin{align*}
{\cal M}_{2} (X,q,m) &=\sum_{\substack{1\leq a\leq q\\ (a,q)=1}}\Bigg|\sum_{\substack{1\leq d\leq X\\ (d,mq)=1}}f_1(d)\Big( 
\sum_{\substack{1\leq n\leq X/d\\ dn\equiv a \bmod q\\ (n,m)=1}}\mu(n)^2-\frac {6X}{\pi^2d\varphi(q)\varphi_{-1} (mq)}\Big)\Bigg|^2
\cr
&\leq 2 \Big(\sum_{\substack{1\leq d\leq X%{\color{red}X/q}
\\ (d,mq)=1}}\frac1{d^{1+\varepsilon}}\Big)\sum_{\substack{1\leq d\leq X\\ (d,mq)=1}}|f_1(d)|^2d^{1+\varepsilon}{\cal M}_{2,\mu^2}(X/d,q,m) 
\cr&\quad 
+ 2 \Big(  \sum_{\substack{d>X%{\color{red}X/q}
\\ (d,mq)=1}}\frac { f_1(d)}{d } \Big)^2
 %{\color{red}\Big(}
 \frac {36X^2}{\pi^4 \varphi(q) \varphi_{-1} (mq)^2}
 %{\color{red}+\phi(q)\Big)}
 .
    \end{align*}
En raisonnant de la mani\`ere que pour ${\cal M}_{2,\mu^2}(X,q,m)$, nous obtenons ais\'ement le r\'esultat annonc\'e. \end{proof}

Nous \'enon\c cons une cons\'equence directe du Lemme~\ref{lemme M2} (voir aussi \cite{LB18}) et du Lemme \ref{lemme est S*}.
%\begin{lemma}\label{lemme sumSS}Lorsque $X,Y\geq {\color{red}q\geq 1}$,    tel que $\mu(q)^2=1$ et $(u,q)=1, $ alors \begin{equation}    \label{estsumSSLB}\begin{split}   \sum_{\substack{1\leq a\leq q\\ (a,q)=1}}  &S_{\mu^2}(X;q,a,1) S_{\mu^2}(Y;q,au,1) -\frac {36XY}{\pi^4\varphi(q) \varphi_{-1} (q)^2}\cr&\ll%\psi_{1/2}(m)^4 (XY)^{1/2 }(\log 2X)^{5/2} (\log 2Y)^{5/2} 2^{\omega(q)} \Big(1+\frac{ X^{1/2 }{\cal L}(Y)^{-1}+ Y^{1/2 }{\cal L}(X)^{-1}}{q }  \Big) .\end{split}\end{equation}\end{lemma}

\begin{lemma}\label{lemme sumSS2}Lorsque $X,Y\geq    1$ et $u,m_1,m_2,q\geq 1$    tels que $\mu(q)^2=1$ et $(um_1m_2,q)=1, $ alors
{\color{black}\begin{equation}
    \label{estsumSSLB2}\begin{split}  
&\sum_{\substack{1\leq a\leq q\\ (a,q)=1}} 
 S (X;q,a,m_1)
 S (Y;q,au,m_2)
 -\frac {36XY\kappa(m_1q)\kappa(m_2q)}{\pi^4\varphi(q) \varphi_{-1} (m_1q)\varphi_{-1} (m_2q)}\cr&\ll \psi_{1/2}(m_1)^2\psi_{1/2}(m_2)^2
(XY)^{1/2 }
(\log 2X)^{5/2} (\log 2Y)^{5/2} 2^{\omega(q)} \Big(1+\frac{ X^{1/2 }{\cal L}(Y)^{-1}+ Y^{1/2 }{\cal L}(X)^{-1}}{q }  \Big) ,
\end{split}\end{equation}
et}
\begin{equation}
    \label{estsumSSLB}\begin{split}  
\sum_{\substack{1\leq a\leq q\\ (a,q)=1}} 
&S_{\mu^2}(X;q,a,1)
 S_{\mu^2}(Y;q,au,1)
 -\frac {36XY}{\pi^4\varphi(q) \varphi_{-1} (q)^2}\cr&\ll 
(XY)^{1/2 }
(\log 2X)^{5/2} (\log 2Y)^{5/2} 2^{\omega(q)} \Big(1+\frac{ X^{1/2 }{\cal L}(Y)^{-1}+ Y^{1/2 }{\cal L}(X)^{-1}}{q }  \Big) .
\end{split}\end{equation}
\end{lemma}
   
%\begin{proof}   Nous pouvons supposer $X\geq Y.$ Le terme \`a estimer s'\'ecrit\begin{equation}\label{terme a estimer} \begin{split}    \sum_{\substack{1\leq a\leq q\\ (a,q)=1}} &  S_{\mu^2}(X;q,a,1) S_{\mu^2}(Y;q,au,1)-\frac {36XY}{\pi^4\varphi(q) \varphi_{-1} (q)^2} %
%\cr&  = \frac {6X}{\pi^2\varphi(q)\varphi_{-1} (q)}\sum_{\substack{1\leq a\leq q\\ (a,q)=1}}  E_{\mu^2}(Y;q,au,1) +\frac {6Y}{\pi^2\varphi(q)\varphi_{-1} (q)}\sum_{\substack{1\leq a\leq q\\ (a,q)=1}}  E_{\mu^2}(X;q,a,1)\cr&\quad + \sum_{\substack{1\leq a\leq q\\ (a,q)=1}}   E_{\mu^2}(X;q,a,1) E_{\mu^2}(Y;q,au,1). \end{split}\end{equation}
%Or d'apr\`es le Lemme~\ref{lemme est S*}, nous avons \begin{align*}\sum_{\substack{1\leq a\leq q\\ (a,q)=1}}  E_{\mu^2}(X;q,a,1)
%&=\sum_{\substack{1\leq n\leq X\\ (n,q)=1}}\mu(n)^2-  \frac {6X}{\pi^2 \varphi_{-1} (q)}\cr &= S_{\mu^2}(X;q)-  \frac {6X}{\pi^2 \varphi_{-1} (q)} \ll \psi_{1/3}(q)^2X^{1/2 }{\cal L}(X)^{-1}. \end{align*} Pour majorer le dernier terme du membre de droite de \eqref{terme a estimer}, nous appliquons la borne de Le Boudec \eqref{majM2mu2m} via l'in\'egalit\'e de Cauchy-Schwarz. 
 % Le terme d'erreur  est donc bien  \begin{align*}&\ll \psi_{1/3}(q)^2\frac{XY^{1/2 }{\cal L}(Y)^{-1}}{q}     + (XY)^{1/2 }(\log 2X)^{5/2} (\log 2Y)^{5/2}  2^{\omega(q)} \cr&\ll  (XY)^{1/2 } (\log 2X)^{5/2} (\log 2Y)^{5/2}  2^{\omega(q)}  \Big(1+\frac{ X^{1/2 }{\cal L}(Y)^{-1} }{q }  \Big).\end{align*}\end{proof}

\begin{proof} {\color{black} Nous \textcolor{black}{n'}\'ecrivons pas les d\'etails de la preuve de la seconde estimation qui ne sera pas utilis\'ee dans cet article mais est \'ecrite \`a des fins de r\'ef\'erences ult\'erieures.}  Nous pouvons supposer $X\geq Y.$
Le terme \`a estimer s'\'ecrit\begin{equation}\label{terme a estimer}
    \begin{split} 
 &    \sum_{\substack{1\leq a\leq q\\ (a,q)=1}}  
 S(X;q,a,m_1)
S(Y;q,au,m_2)-  \frac {36XY\kappa(m_1q)\kappa(m_2q)}{\pi^4\varphi(q) \varphi_{-1} (m_1q)\varphi_{-1} (m_2q)} 
\cr& 
=  
\frac {6X\kappa(m_2q)}{\pi^2\varphi(q)\varphi_{-1} (m_1q)}\sum_{\substack{1\leq a\leq q\\ (a,q)=1}}  E (Y;q,au,m_2)
+\frac {6Y}{\pi^2\varphi(q)\varphi_{-1}(m_2q)}\sum_{\substack{1\leq a\leq q\\ (a,q)=1}}  E (X;q,a,m_1)\cr&\quad +
\sum_{\substack{1\leq a\leq q\\ (a,q)=1}} 
 E (X;q,a,m_1) E (Y;q,au,m_2).
\end{split}
\end{equation}
Or d'apr\`es le Lemme~\ref{lemme est S*}, nous avons
\begin{align*}\sum_{\substack{1\leq a\leq q\\ (a,q)=1}}  E (X;q,a,m_1) 
&=
S_{\mu^2}(X;m_1q)-  \frac {6X\kappa(m_1q)}{\pi^2 \varphi_{-1} (m_1q)}
\ll \psi_{1/3}(m_1q)^2X^{1/2 }{\cal L}(X)^{-1}. \end{align*} Pour majorer le dernier terme du membre de droite de \eqref{terme a estimer}, nous appliquons la borne de Le Boudec \eqref{majM2mu2m2} via l'in\'egalit\'e de Cauchy-Schwarz. 
  Le terme d'erreur  est donc bien
 \begin{align*}&\ll
\psi_{1/3}(q)^2\frac{XY^{1/2 }{\cal L}(Y)^{-1}}{q}    
 + \psi_{1/2}(m_1)^2\psi_{1/2}(m_2)^2(XY)^{1/2 }(\log 2X)^{5/2} (\log 2Y)^{5/2}  2^{\omega(q)} 
\cr&\ll \psi_{1/2}(m_1)^2\psi_{1/2}(m_2)^2 (XY)^{1/2 } (\log 2X)^{5/2} (\log 2Y)^{5/2}  2^{\omega(q)}  \Big(1+\frac{ X^{1/2 }{\cal L}(Y)^{-1} }{q }  \Big).
\end{align*}
\end{proof}

Lorsque $(a_1a_2,g q)=1$, $(g,q)=1$, nous introduisons la somme
\begin{align*}W(X,Y; g,q,a_1,a_2)&:=\sum_{\substack{1\leq a\leq q\\ (a,q)=1}}  \sum_{\substack{1\leq x\leq X\\ ag x\equiv  a_2(\!\bmod q)\\  (x,a_2g)=1}}\frac{\mu(x)^2x^2 }{\phi(x)^2}\sum_{\substack{1\leq y\leq Y\\ ag y\equiv  -a_1(\!\bmod q)\\  (y,a_1gx)=1}}\frac{\mu(y)^2y^2 }{\phi(y)^2}
.\end{align*}
Nous observons 
\begin{align*}W(X,Y; g,q,a_1,a_2)& =  \sum_{\substack{1\leq x\leq X,\, 1\leq y\leq Y\\ a_1x+a_2 y\equiv  0(\!\bmod q)\\  (y,a_1gqx)=1,\,    (x,a_2gq)=1}}\frac{\mu(x)^2x^2 }{\phi(x)^2} \frac{\mu(y)^2y^2 }{\phi(y)^2}
.\end{align*}
Nous aurons aussi besoin d'estimer  la somme
\begin{align*}W^*(X,Y; g,q,a_1,a_2)&:= \sum_{\substack{1\leq x\leq X,\, 1\leq y\leq Y\\ a_1x+a_2 y\equiv  0(\!\bmod q)\\  (y,a_1gqx)=1
,\,     (x,a_2gq)=1\\ (a_1x+a_2 y,g)=1}}\frac{\mu(x)^2x^2 }{\phi(x)^2} \frac{\mu(y)^2y^2 }{\phi(y)^2}
.\end{align*}

Nous introduisons les fonctions arithm\'etiques $w$ et $w^*$ d\'efinies par 
\begin{equation}
\label{defw}    
w( q,a_1,a_2):=\prod_{p\mid  q (a_1,a_2)}\Big( \frac{(1-1/p)^2}{1+1/p^2}\Big) \prod_{\substack{p\mid a_1a_2\\ p\nmid (a_1,a_2)}}\Big( \frac{1-1/p+1/p^2}{1 +1/p^2}\Big) ,\end{equation} et
\begin{equation}
\label{defw*}    
w^*( g,q,a_1,a_2):=w( gq,a_1,a_2)\prod_{p\mid g}\Big(\frac{1-2/p}{1-1/p}\Big)=w( gq,a_1,a_2) \frac{\phi_2(g)}{\phi_1(g)} . 
\end{equation}
Le lemme suivant sera utilis\'e pour d\'emontrer le Lemme \ref{lemme estS3}.

\begin{lemma}\label{lemme sumW}Lorsque $X,Y\geq 1$ et $a_1,a_2,g,q\geq 1$     tels que $\mu(q)^2=1$ et $(a_1a_2,gq)=(g,q)=1, $ nous avons alors 
\begin{align*} %\begin{equation}   \label{estsumWLB}\begin{split}
& W (X, Y; g, q,a_1,a_2) 
-\frac {\zeta(2)}{\zeta(4) }XY\frac{w(gq,a_1,a_2)}{\phi(q)}\cr& \ll \psi_{1/3}(a_1g)^2\psi_{1/3}(a_2g)^2(XY)^{1/2 }
(\log (2XY))^{6}  2^{\omega(q)}
\Big(1+\frac{ X^{1/2 }{\cal L}(Y)^{-1}+ Y^{1/2 }{\cal L}(X)^{-1}}{q } \Big)  ,
\end{align*} 
et 
 \begin{align*} 
 &W^* (X, Y; g, q,a_1,a_2) 
-\frac {\zeta(2)}{\zeta(4) }XY\frac{w^*(g,q,a_1,a_2)}{\phi(q)}\cr& \ll  \psi_{1/3}(a_1g)^2\psi_{1/3}(a_2g)^2(XY)^{1/2 }
(\log (2XY))^{6}   2^{\omega(q)}
\Big(1+\frac{ X^{1/2 }{\cal L}(Y)^{-1}+ Y^{1/2 }{\cal L}(X)^{-1}}{q }  \Big)  .
\end{align*}%\end{split}\end{equation}
\end{lemma}

\begin{proof}La seconde estimation se d\'eduit de la premi\`ere par une simple interversion de M\"obius puisque 
$$
W^* (X, Y; g, q,a_1,a_2) =\sum_{k\mid g}\mu(k) W  (X, Y; g/k, kq,a_1,a_2) 
.$$

Estimons donc \textcolor{black}{$W (X, Y; g, q,a_1,a_2)$}.
Nous avons
\begin{align*}W(X,Y; g,q,a_1,a_2) =\sum_{\substack{1\leq a\leq q\\ (a,q)=1}}  \sum_{\substack{1\leq x\leq X\\  x\equiv  aa_2(\!\bmod q)\\  (x,a_2g)=1}}\frac{\mu(x)^2x^2 }{\phi(x)^2}\sum_{\substack{1\leq y\leq Y\\   y\equiv  -aa_1(\!\bmod q)\\  (y,a_1gx)=1}}\frac{\mu(y)^2y^2 }{\phi(y)^2}.\end{align*}  
{\color{black} 
Une interversion de M\"obius fournit
\begin{align*}W(X,Y; g,q,a_1,a_2) &=\sum_{\substack{k\leq X\\ (k,qa_1a_2g)=1}}\frac{\mu(k)k^4}{\phi(k)^4}\sum_{\substack{1\leq a\leq q\\ (a,q)=1}}  \sum_{\substack{1\leq x\leq X/k\\  x\equiv  aa_2(\!\bmod q)\\  (x,a_2gk)=1}}\frac{\mu(x)^2x^2 }{\phi(x)^2}\sum_{\substack{1\leq y\leq Y/k\\   y\equiv  -aa_1(\!\bmod q)\\  (y,a_1gk  )=1}}\frac{\mu(y)^2y^2 }{\phi(y)^2}\cr&
=\sum_{\substack{k\leq X\\ (k,qa_1a_2g)=1}}\frac{\mu(k)k^4}{\phi(k)^4}\sum_{\substack{1\leq a\leq q\\ (a,q)=1}}  S\Big(\frac{X}{k},q,aa_2,a_2gk\Big) S\Big(\frac{Y}{k},q,-aa_1,a_1gk\Big).\end{align*}  
Nous supposons $X\geq Y$. Lorsque $k\leq Y$, nous pouvons appliquer le Lemme \textcolor{black}{\ref{lemme sumSS2}} avec $u$ satisfaisant \`a la congruence $-a_1\equiv a_2u\bmod q$. %Lorsque $k> Y/q$ et $y\leq Y/k$, il existe un seul $a$ pour lequel  $y\equiv  -aa_1\bmod q $. La somme int\'erieure est donc toujours $$\leq \frac{Y}{k}\Big( \frac{X}{kq}+1\Big).$$ En sommant sur $k$, nous obtenons que la contribution des $k> Y/q$ est $=O(X+Y\log (2X))$ ce qui convient largement.
}{\color{black} 
En prolongeant la s\'erie pour les $k>Y$, nous obtenons un terme principal \'egal \`a 
$$ =\frac {36XY}{\pi^4\varphi(q)  }\sum_{(k,qa_1a_2g)=1}\frac{\mu(k)k^2}{\phi(k)^4}\frac { \kappa(gqa_1k)\kappa(gqa_2k)}{   \varphi_{-1} (gqa_1k)\varphi_{-1} (gqa_2k)} .
$$
En \'ecrivant cette somme sous forme de produit eul\'erien puis en calculant chaque facteur nous obtenons le terme principal annonc\'e.}

\end{proof}

\subsection{Estimations usuelles de certaines doubles sommes}

Nous aurons aussi besoin de sommer la fonction $w^*$ d\'efinie en \eqref{defw*}. Nous rappelons ici la d\'efinition \eqref{defpsi} de $\phi_j$.

\begin{lemma}\label{lemme soma2} Lorsque $ q\geq 1,$ $a_1,A_2\geq 1$  nous avons 
$$\sum_{\substack{1\leq  a_2\leq A_2\\ (a_2,gq)=1}}w^*(g,q,a_1,a_2)
=   \frac{\phi_1(gq)^2 \phi_2(g)}{ \varphi_{-1}(gq)\phi_1(g)}\frac{\zeta(4)}{\zeta(2)^2}A_2
+O\big(\psi_{1/2}(a_1gq) A_2^{1/2}\big).$$ 
\end{lemma}

\begin{proof}  
  Lorsque $(a_1a_2,gq)= 1, $ nous \'ecrivons
$$w^*(g,q,a_1,a_2) =\frac{ \phi_2(g)}{ \phi_1(g)}\prod_{p\mid gq}\Big( \frac{(1-1/p)^2}{1+1/p^2}\Big)w(1,a_1,a_2).$$ 
Lorsque   $\Re e (s)>1$, un rapide calcul fournit 
\begin{align*}\sum_{\substack{ a_2=1\\ ( a_2,gq)=1}}^\infty \frac{w( 1,a_1,a_2)}{ a_2^{s}}
   &=\prod_{\substack{p\nmid gq\\ p\mid a_1}}\Bigg( \frac{1-1/p+1/p^2 }{1 +1/p^2}+ \frac{(1-1/p)^2}{(p^{s}-1) ( 1+1/p^2)}\Bigg)\cr&\quad
    \prod_{\substack{p\nmid gq\\ p\nmid a_1}}\Bigg( 1+\frac{1-1/p+1/p^2 }{(1 +1/p^2)(p^{s}-1)}  \Bigg)
\cr&=\zeta(s)\prod_{\substack{p\nmid gq }}\Big( \frac{1-1/p^{s+1} }{1 +1/p^2} \Big)\prod_{\substack{p\mid gq }}\Big( 1-\frac{1}{p^{s} }  \Big)\cr&\quad
\prod_{\substack{  p\mid a_1}}\Big( \frac{1-1/p+1/p^2  + (1-1/p)^2/(p^{s}-1)}{1+1/p^2 -1/ p^{1+s}}\Big).\end{align*}
Gr\^ace au Lemme \ref{lemme sum xi}, nous en d\'eduisons l'estimation
$$\sum_{\substack{ a_2\leq A_2\\ (a_2, gq)=1}} w( 1,a_1,a_2)
=\frac{\zeta(4)}{\zeta(2)^2\varphi_{-1}(gq)}\prod_{p\mid gq}\Big(1+\frac{1}{p^2}\Big)
 A_2+O\big(\psi_{1/2}(gqa_1)^2A_2^{1/2}\big)
.$$
Cela fournit le r\'esultat requis.
\end{proof}

\begin{lemma}\label{lemme soma1a2} Lorsque $ g,q\geq 1,$ nous avons lorsque $A_1,A_2\geq 1$ 
$$\sum_{\substack{1\leq a_1\leq A_1\\ 1\leq  a_2\leq A_2\\  (a_1a_2,gq)= 1}}w^*(g,q,a_1,a_2)
=  \frac{\phi_1(gq)^3\phi_2(g)}{ \varphi_{-1}(gq)\phi_1(g)}\frac{\zeta(4)}{\zeta(2)^2}A_1A_2
+O\big(2^{\omega(gq)}(A_1+A_2)\log (2A_1A_2)\big).$$ 
\end{lemma}

\begin{proof} La m\'ethode est bien connue. Nous d\'etaillons les calculs afin de d\'eterminer le facteur de $A_1A_2$ que nous obtenons.
  Lorsque $(a_1a_2,gq)= 1, $ nous \'ecrivons
$$w^*(g,q,a_1,a_2) =\frac{ \phi_2(g)}{ \phi_1(g)}\prod_{p\mid gq}\Big( \frac{(1-1/p)^2}{1+1/p^2}\Big)w(1,a_1,a_2).$$
%avec \begin{equation} \label{defw*}     w^*(q,a_1,a_2):=\prod_{\substack{p\mid a_1a_2\\ p\nmid (a_1,a_2)}}\Big( \frac{1-1/p+1/p^2 }{1 +1/p^2}\Big)\prod_{\substack{p\mid  (a_1,a_2)}}\Big( \frac{(1-1/p)^2 }{1+1/p^2}\Big) .\end{equation} Nous convenons que $w^*(q,a_1,a_2)=0$ si $(a_1a_2,q)\neq 1. $
Lorsque $\Re e (s_1)>1$ et $\Re e (s_2)>1$, un rapide calcul fournit 
\begin{align*}\sum_{\substack{a_1,a_2=1\\ (a_1a_2,gq)=1}}^\infty& \frac{w( 1,a_1,a_2)}{a_1^{s_1}a_2^{s_2}}
    \cr&=\prod_{p\nmid gq}\Bigg( 1+\Big(\frac{1}{p^{s_1}-1}+\frac{1}{p^{s_2}-1}\Big)\frac{1-1/p+1/p^2 }{1 +1/p^2}+ \frac{(1-1/p)^2}{(p^{s_1}-1)(p^{s_2}-1)( 1+1/p^2)}\Bigg).\end{align*}
Nous avons donc la relation de convolution
$$1_{(a_1a_2,gq)=1}w(1,a_1,a_2)=\big(1_{(a_1a_2,gq)=1}*w'\big)(a_1,a_2),$$
avec, lorsque $\Re e (s_1)>1$ et $\Re e (s_2)>1$,
$$\sum_{a_1,a_2=1}^\infty \frac{w'(q,a_1,a_2)}{a_1^{s_1}a_2^{s_2}}
    =\prod_{p\nmid gq}\Bigg( 1-\Big(\frac{1}{p^{s_1} }+\frac{1}{p^{s_2}} \Big)\frac{1}{p(1 +1/p^2)}%- \frac{1}{ p^{s_1+s_2} ( p^2+1)}
    \Bigg).$$
Nous avons
$$
\sum_{\substack{1\leq a_1\leq A_1\\ 1\leq a_2\leq A_2}}1_{(a_1a_2,gq)=1}
= {\phi_1(gq)^2} A_1A_2+O\big(2^{\omega(gq)}(A_1+A_2)\big) 
$$
et
$$ \sum_{a_1,a_2=1}^\infty \frac{|w'(gq,a_1,a_2)|}{a_1^{1/\log (2A_1)}a_2^{1/\log (2A_2)}}\Big(\frac1{a_1}+\frac1{a_2}\Big)\ll \log (2A_1A_2)$$
   ce qui fournit alors  
$$\sum_{\substack{1\leq a_1\leq A_1\\ 1\leq a_2\leq A_2\\ (a_1a_2,gq)=1}}w (1,a_1,a_2)
- {\phi_1(gq)^2} A_1A_2
\prod_{p\nmid gq}\Big(   \frac{1-1/p^2}{ 1+1/ p^2 }\Big)
\ll 2^{\omega(gq)}(A_1+A_2)\log (2A_1A_2).$$ L'estimation requise en d\'ecoule.
\end{proof}

Nous introduisons la double somme suivante 
$$N(A_1,A_2; g,x,y,z):=\sum_{\substack{1\leq a_1 \leq A_1\\ (a_1,q_1)=1}} \sum_{\substack{ 1\leq a_2\leq A_2 \\  a_2\equiv -a_1xy^{-1}\bmod z\\ (a_2 ,gx)=(a_1x+a_2y,g )=1}} 1 $$
avec $q_1=gyz$, et rappelons la d\'efinition \eqref{defpsi} de $\phi_j$.

\begin{lemma}\label{lemme estNA1A2} Lorsque $\mu^2(gxyz)=1$ et $A_1, A_2\geq 1$, nous avons uniform\'ement  l'estimation
    \begin{equation}\label{est1NA1A2}\begin{split}
N(A_1,A_2; g,x,y,z) 
&=  { \phi_1(q_1)\phi_2(g) \phi_1(x)} \frac{  A_1A_2}{z} 
\cr&\quad +O\Big(2^{\omega(q_1)}  \frac{  A_2}{z}+2^{\omega(x)}3^{\omega(g)}A_1
+2^{\omega(q_1)}  2^{\omega(x)}3^{\omega(g)}\Big).\end{split}
\end{equation}
\end{lemma}

\begin{proof}
Nous avons %\begin{equation}%\label{est1NA1A2}
    \begin{align*} 
N(A_1,A_2&; g,x,y,z)  =\sum_{\substack{1\leq a_1 \leq A_1\\ (a_1,q_1)=1}}\sum_{\substack{k_2\mid gx\\ k_1\mid  g \\  (k_1,k_2)=1}}\mu(k_1)\mu (k_2)\sum_{\substack{ a_2\leq A_2/ k_2\\  k_2a_2\equiv -a_1xy^{-1}\bmod (k_1z)}} 1
\cr&=\sum_{\substack{1\leq a_1 \leq A_1\\ (a_1,q_1)=1}}\sum_{\substack{k_2\mid gx\\ k_1\mid  g \\  (k_1,k_2)=1}}\mu(k_1)\mu (k_2)\Big\{ \frac{A_2}{k_1k_2z}+O (1  )\Big\}\cr&
= \Big( { \phi_1(q_1)} A_1+O\big(2^{\omega(q_1)}\big)\Big)\Big( { \phi_2(g) \phi_1(x) } \frac{  A_2}{z}+O\big(2^{\omega(x)}3^{\omega(g)}\big)\Big) 
\cr&
=  { \phi_1(q_1)\phi_2(g) \phi_1(x)} \frac{  A_1A_2}{z} 
+O\Big(2^{\omega(q_1)}  \frac{  A_2}{z}+2^{\omega(x)}3^{\omega(g)}A_1
+2^{\omega(q_1)}  2^{\omega(x)}3^{\omega(g)}\Big).
\end{align*} 
%\end{equation}
%En fait, nous utiliserons \begin{equation}\label{est1NA1A2utile}
%   \begin{split}  N(A_1,A_2; g,x,y,z)=  { \phi_1(q_1)\phi_2(g) \phi_1(x)} \frac{  (A_1-1)(A_2-1)}{z}  +O\Big(2^{\omega(q_1)}  \frac{  A_2}{z}+2^{\omega(x)}3^{\omega(g)}A_1 +2^{\omega(q_1)}  2^{\omega(x)}3^{\omega(g)}\Big). \end{split} \end{equation}
\end{proof}

\subsection{Estimations  de certaines sommes friables}

Lorsque $\delta\in \{ (1,1),(1,0),(0,1),(0,0)\}$,
nous introduisons les ensembles
\begin{equation}\label{def ens D}
    \begin{split} 
D(1,1)& :=\{(x,y,z)\in \mathbb N^3\,:\, x\geq y\geq z \geq 1\},\cr
D(1,0)&:=\{(x,y,z)\in \mathbb N^3\,:\, x\geq y>z \geq 1\},\cr
D(0,1)&:=\{(x,y,z)\in \mathbb N^3\,:\, x>y\geq z \geq 1\},\cr
D(0,0)&:=\{(x,y,z)\in \mathbb N^3\,:\, x> y> z\geq 1 \}.   \end{split}
\end{equation}
%o\`u, \`a chaque fois, $X,Y,Z$ satisfont \eqref{defXYZ} et \eqref{leqXYZ}.
Nous estimons 
$$C(h,T):=C\big(h,T;(1,1)\big)+2C\big(h,T;(0,1)\big)+2C\big(h,T;(1,0)\big)+ C\big(h,T;(0,0)\big)$$avec
\begin{equation}
    \label{def C1qT}
C(h,T;\delta):=\sum_{\substack{P(gxyz)\leq h^3 \\ g yz>hT,   g x>hT\\ (x,y,z)\in D(\delta)}}\frac{\mu(g)\phi_2(g) \mu(gxyz)^2}{\phi(g)\phi_1(g)\phi( x  yz)},\end{equation}
o\`u nous rappelons la d\'efinition \eqref{defpsi} de $\phi_j.$

\begin{lemma} \label{LemmeestC1}  Lorsque  $q=\prod_{p\leq h^3} p$ et $%\textcolor{red}
{h\geq }T\geq 2$, nous avons
    $$C(h,T)= \Big(\frac{q}{\phi(q)}\Big)^2
- 3\frac q{  \phi(q)}(\log h)+ %w\textcolor{red}
{\frac32}(\log h)^2+O \big((\log T)^3\log h+(\log h)^3{\cal L}(T)^{-1} \big).$$
\end{lemma}

La preuve de cette estimation est longue mais \'el\'ementaire.  Elle  peut \^etre omise en premi\`ere lecture.

\begin{proof}[D\'emonstration du Lemme \ref{LemmeestC1}.]
Le principe d'inclusion-exclusion permet d'\'eviter les conditions de $h^3$-friabilit\'e des variables \`a sommer. 
\`A l'aide du facteur oscillant $\mu(g)$, nous  nous restreignons d\`es que possible au cas o\`u $g$ est petit. 
Dans tous les calculs qui suivent, nous d\'etaillons le cas o\`u $\delta=(0,0)$. Les autres cas peuvent \^etre  trait\'es de la m\^eme mani\`ere.

Nous avons\begin{align*}
C(h,T;(0,0) ) 
&=\sum_{\substack{P(gxyz)\leq h^3  \\ x>y>z}}\frac{\mu(g)\phi_2(g) \mu(gxyz)^2}{\phi(g)\phi_1(g)\phi( x  yz)}
- \sum_{\substack{P(gxyz)\leq h^3 \\ g yz\leq hT \\ x>y>z}}\frac{\mu(g)\phi_2(g) \mu(gxyz)^2}{\phi(g)\phi_1(g)\phi( x  yz)}\cr&\quad-\sum_{\substack{    gx\leq hT\\ x>y>z}}\frac{\mu(g)\phi_2(g) \mu(gxyz)^2}{\phi(g)\phi_1(g)\phi( x  yz)}+ \sum_{\substack{ g yz\leq hT,   g x\leq hT\\ x>y>z}}\frac{\mu(g)\phi_2(g) \mu(gxyz)^2}{\phi(g)\phi_1(g)\phi( x  yz)}.
    \end{align*} Les autres sommes $C(h,T;\delta )$ se d\'ecomposent de la m\^eme mani\`ere.
Appelons $C_{j}(h,T;\delta)$ la $j$-i\`eme somme et $C_{j}(h,T )$ la somme pond\'er\'ee associ\'ee. 
Par sym\'etrie 
$$C_{1}(h,T )= \sum_{\substack{P(gxyz)\leq h^3  }}  \frac{\mu(g)\phi_2(g) \mu(gxyz)^2}{\phi(g)\phi_1(g)\phi( x  yz)}.$$
Un calcul facile fournit 
\begin{equation}
    \label{estC11*} C_{1}(h,T  ) = \sum_{P(m)\leq h^3} \frac{\mu(m)^2}{\phi(m)}\prod_{p\mid m} \Big(3-\frac{1-2/p}{1-1/p}\Big)=\Big(\frac{q}{\phi(q)}\Big)^2.
    \end{equation}

Nous avons
\begin{align*} C_{2}(h,T;(0,0))&=\sum_{\substack{P(gxyz)\leq h^3 \\ g yz\leq hT \\ x>y>z}}\frac{\mu(g)\phi_2(g) \mu(gxyz)^2}{\phi(g)\phi_1(g)\phi( x  yz)}
\cr&=
\sum_{\substack{P(gxyz)\leq h^3 \\ g yz\leq hT \\    y>z}}\frac{\mu(g)\phi_2(g) \mu(gxyz)^2}{\phi(g)\phi_1(g)\phi( x  yz)}-
\sum_{\substack{P(gxyz)\leq h^3 \\ g yz\leq hT \\  y>z, y\geq x}}\frac{\mu(g)\phi_2(g) \mu(gxyz)^2}{\phi(g)\phi_1(g)\phi( x  yz)}.
\end{align*} 
Notons $C_{21}(h,T;(0,0))$ et $C_{22}(h,T;(0,0))$ ces deux sommes. La somme sur $x$ dans la premi\`ere somme est compl\`ete de sorte que 
%$$C_{21}(h,T;(0,0))=\frac q{\phi(q)} \sum_{\substack{(g,y,z)\\ g YZ\leq hT\\ Y>Z  }} \frac{\mu(g)\phi_2(g) \mu(gyz)^2}{g\phi_1(g)     yz },$$
%et ainsi
\begin{equation}
    \label{C21=c21}
C_{21}(h,T )= \frac {3q}{\phi(q)} c_{21}(h,T ),\end{equation}
avec 
$$c_{21}(h,T ):= \sum_{\substack{ (g,y,z)\\ g yz\leq hT   }} 
\frac{\mu(g)\phi_2(g) \mu(gyz)^2}{g\phi_1(g)     yz }.$$
%\textcolor{red}{Ici, nous avons omis la condition de friabilit\'e portant sur $y$ et $z$ puisque $yz\leq hT\leq h^2.$}
En sommant sur les $m=gyz$ et en utilisant la formule
\begin{equation}
    \label{eq formula}\mu(m)^2\sum_{\substack{(g,y,z)\\ gyz=m}}\frac{\mu(g)\phi_2(g) }{ \phi_1(g) }=\mu(m)^2\prod_{p\mid m} \Big(2-\frac{1-2/p}{1-1/p}\Big)=
    { \mu(m)^2m\over \phi(m)},\end{equation} nous obtenons
$$c_{21}(h,T ) = \sum_{m\leq hT} \frac{\mu(m)^2}{ \phi(m)}=\log (h)+O(\log T),$$ 
puis\begin{equation}
    \label{estC21*}
C_{21}(h,T )= \frac {3q}{\phi(q)} \log (h)+O\big((\log h)\log T\big),\end{equation}

La somme  $C_{22}(h,T; (0,0))$ sera estim\'ee avec $C_{4}(h,T; (0,0))$.

Estimons $C_{3}(h,T)$. % en suivant les m\^emes \'etapes que pour $C_{22}(h,T)$.
 Nous commen\c cons par   restreindre le comptage au domaine $ x>h$, $g\leq T$, puis $y\leq h$.
%, puis estimons les sommes en $x,y,z$ gr\^ace \`a une triple application du Lemme \ref{lemme sum xi}.  
Il vient 
\begin{align*}
C_{3}\big(h,T; & (0,0)\big) %=\sum_{\substack{(g,x,y,z)\\    gx\leq hT\\ x>y>z}}\frac{\mu(g)\phi_2(g) \mu(gxyz)^2}{\phi(g)\phi_1(g)\phi( x  yz)}\cr&
=\sum_{\substack{(g,x,y,z)\\     h<x\leq hT/g\\ h\geq y>z\\ g\leq T}}\frac{\mu(g)\phi_2(g) \mu(gxyz)^2}{\phi(g)\phi_1(g)\phi( x  yz)}+O \big( (\log h)(\log T)^3+(\log h)^2{\cal L}(T)^{-1} \big).
 \end{align*}

%\cr&= C(\log 2)^3\sum_{\substack{ h<  X\leq hT/g\\  g\leq T \\ X>Y>Z  }} \frac{\mu(g)\phi_2(g)  }{g\phi_1(g)\phi_{-2}(g)}+O \big( (\log h)(\log T)^3+(\log h)^2{\cal L}(T)^{-1} \big) \end{align*}

En utilisant
$$\sum_{\substack{z<y\\ (z,gxy)=1}}
\frac{ \mu(z)^2}{ \phi(z)}=\phi_1(gxy)\log y+O\big(\psi_{1/2}(gxy)\big) $$ cons\'equence imm\'ediate du Lemme \ref{lemme sum xi},
puis
$$\sum_{\substack{y<h\\ (y,gx)=1}}
\frac{ \mu(y)^2\log y}{ y}=\frac{(\log h)^2}{2\zeta(2)\phi_1(gx)}+O\big(\psi_{1/2}(gx )\log h\big),$$
nous obtenons \begin{align*}
C_{3}\big(h,T; & (0,0)\big)  =\frac{(\log h)^2}{2\zeta(2)}\!\!\!\sum_{\substack{(g,x )\\     h<x\leq hT/g\\ g\leq T\\ }}\frac{\mu(g)\phi_2(g) \mu(gx)^2}{gx\phi_1(g)\phi_{-1}(gx)}+O \big( (\log h)(\log T)^3+(\log h)^2{\cal L}(T)^{-1} \big).
 \end{align*}
Nous omettons alors les conditions $g\leq T$ et $h<x$ 
puisque
\begin{equation}
    \label{somgmu}
\sum_{g\leq G}\frac{\mu(g)\phi_2(g) \mu(gx)^2}{\phi(g)\phi_1(g)      }\ll \psi_{1/2}(x){\cal L}(G)^{-1}.\end{equation}
Puis nous rassemblons les termes en fonction de $m=gx$. En utilisant \begin{equation}
    \label{eq formula2}\mu^2(m)\sum_{\substack{(g,x)\\ gx=m}}\frac{\mu(g)\phi_2(g) }{ \phi_1(g) }={\mu^2(m)\over \phi(m)} ,\end{equation} nous obtenons apr\`es calculs
\begin{align*}
C_{3}\big(h,T;   (0,0)\big)  &=\frac{(\log h)^2}{2\zeta(2)}\sum_{\substack{m\leq hT }}\frac{  \mu(m)^2}{ m^2\phi_1(m)\phi_{-1}(m)}+O \big( (\log h)(\log T)^3+(\log h)^2{\cal L}(T)^{-1} \big)\cr
&=\frac{1}{2}(\log h)^2+O \big( (\log h)(\log T)^3+(\log h)^2{\cal L}(T)^{-1} \big).
 \end{align*}

 %Pour $X=2^j$ fix\'e, le nombre de couples $(k,\ell)$ est \'egal \`a $$=\frac{j^2}{2}+O(j)=\frac{(\log h)^2}{2(\log 2)^2} +O\big((\log h)(\log T)\big).$$ Nous en d\'eduisons \begin{align*}C_{3}(h,T;   (0,0))  &=\frac {C\log 2}{2 }(\log h)^2s(\e^{\vartheta_h}T)+O \big( (\log h)(\log T)^3+(\log h)^2{\cal L}(T)^{-1} \big)\cr &=\frac {1}{2 }(\log h)^2+O \big( (\log h)(\log T)^3+(\log h)^2{\cal L}(T)^{-1} \big),\end{align*}o\`u la derni\`ere \'egalit\'e d\'ecoule du Lemme \ref{lemme s(T)}.
Comme les sommes des cas  $\delta\neq (0,0)$ v\'erifient la m\^eme estimation,   nous obtenons 
 \begin{equation}\label{estC13-2*}
    \begin{split}
 C_{3} (q,T)&=3(\log h)^2
+O \big( (\log h)(\log T)^3+(\log h)^2{\cal L}(T)^{-1} \big).
\end{split}
\end{equation}

Nous estimons maintenant $C_{422}(h,T; \delta):= C_{4}(h,T; \delta)+ C_{22}(h,T; \delta)$ pour  le cas $\delta=(0,0).$
%De m\^eme, sans indiquer   les d\'etails, nous obtenons 
Nous avons
\begin{align*}  
C_{422}\big(h,T;(0,0)\big)&=\sum_{\substack{(g,x,y,z)\\ gyz\leq hT,   g x\leq hT\\ y>z}}\frac{\mu(g)\phi_2(g) \mu(gxyz)^2}{\phi(g)\phi_1(g)\phi( x  yz)}
\cr&=
\sum_{\substack{(g,x,y,z)\\ h<  yz\leq hT/g \\   x\leq hT/g ,\, g\leq T\\  y>z}}\frac{\mu(g)\phi_2(g) \mu(gxyz)^2}{\phi(g)\phi_1(g)\phi( x  yz)}+
\sum_{\substack{(g,x,y,z)\\ h<x\leq hT/g\\   yz\leq h ,\, g\leq T \\ y>z}}\frac{\mu(g)\phi_2(g) \mu(gxyz)^2}{\phi(g)\phi_1(g)\phi( x  yz)}\cr&\quad + O \big((\log h)(\log T)^3 +(\log h)^2{\cal L}(T)^{-1}\big)
%\cr &= \textcolor{red}{\frac38}{(\log h)^2} +O \big((\log h)(\log T)^3+(\log h)^2{\cal L}(T)^{-1} \big)
.
\end{align*}

Apr\`es sommation sur $x$ puis sur $(g,y,z)$, la premi\`ere somme est approch\'ee de mani\`ere satisfaisante par $\tfrac12(\log h)^2$. Apr\`es sommation sur $(y,z)$ puis sur $(g,x)$, la seconde somme est approch\'ee de mani\`ere satisfaisante par $\tfrac14(\log h)^2$. Nous ne fournissons pas les d\'etails qui sont semblables \`a ceux de l'estimation de $C_3$.

Il en d\'ecoule donc 
\begin{equation}\label{estC14*}C_{422}(h,T)=
\frac{%\textcolor{red}
{9}}2{(\log h)^2} +O \big((\log h)(\log T)^3+(\log h)^2{\cal L}(T)^{-1} \big).\end{equation}  
En rassemblant les estimations \eqref{estC11*}, \eqref{estC21*},  \eqref{estC13-2*}, \eqref{estC14*},  nous obtenons le Lemme \ref{LemmeestC1}.
%$$C(h,T)= \Big(\frac{q}{\phi(q)}\Big)^2- \frac {3q}{\phi(q)}(\log h)+  3(\log h)^2+O \big((\log h)(\log T)^3+(\log h)^3{\cal L}(T)^{-1} \big).$$
\end{proof}

\section{Premi\`ere m\'ethode de comptage}
\subsection{Pr\'eliminaires}

Nous rappelons la d\'efinition \eqref{est V3} de $V_3(h)$.
Nous utilisons la d\'ecomposition  \eqref{def qj} et la condition \eqref{congruence avec aj}.  
 Gr\^ace \`a des raisons de sym\'etrie, nous pouvons ordonner les variables $x,y,z$. 
Nous consid\'erons
\begin{equation}
    \label{calculV3delta}
V_3(h;\delta):=\sum_{\substack{P(gxyz)\leq h^3\\ (x,y,z)\in D(\delta) \\  gyz, gxz,gxy\geq 2}}\frac{\mu(g)\mu(gxyz)^2}{\phi(g)\phi(gxyz)^2}V(g,x,y,z;h),
\end{equation}
lorsque $\delta\in \{ (1,1),(1,0),(0,1),(0,0)\}$ o\`u les $D(\delta)$ ont \'et\'e introduits en \eqref{def ens D} et $V$ a \'et\'e d\'efini en~\eqref{def Vh}.

Le lemme suivant permet de toujours se placer dans le cas $
X\geq Y\geq Z \geq 1$. 
\begin{lemma}\label{lemme decoupage} Lorsque $h\geq 2$, nous avons 
$$V_3(h)=
V_3\big(h;(1,1)\big)+2V_3\big(h;(1,0)\big)+2V_3\big(h;(0,1)\big)+ V_3\big(h;(0,0)\big).$$ 
\end{lemma}

\begin{proof} 
 Nous avons
   \begin{align*}
   {\mathbb N}^3=&
    \{ (x,y,z)\in {\mathbb N}^3\,: \, x>y\geq z\}\cup \{ (x,y,z)\in {\mathbb N}^3\,: \, x>z>y\}\cr&\cup \{ (x,y,z)\in {\mathbb N}^3\,: \, y\geq x\geq z\}\cup \{ (x,y,z)\in {\mathbb N}^3\,: \, y\geq z>x\}\cr&
    \cup \{ (x,y,z)\in {\mathbb N}^3\,: \, z\geq x> y\}\cup \{(x,y,z)\in {\mathbb N}^3\,: \, z>y\geq x\} .
  \end{align*} 
  En exploitant la sym\'etrie et en renommant les variables, nous obtenons le r\'esultat annonc\'e.
\end{proof}

Compte-tenu des notations \eqref{defEh+} et \eqref{defEhalpha}, nous avons   \begin{align}
\label{majprodetderivee} & E_h(\alpha)\ll   E_h^+(\alpha),\qquad E'_h(\alpha)\ll   hE_h^+(\alpha).\end{align}
Pour estimer $V_3(h;\delta),$ nous scindons la somme en plusieurs
parties $S_j$ d\'efinies par des in\'egalit\'es.
Nous choisirons la notation $S_j^+$ lorsque l'on remplacera 
$\mu(g) $ par $\mu^2(g) $  et les $E_h$ par des $E_h^+$ de sorte que $|S_j|\ll  S_j^+.$

De plus, nous introduisons un param\`etre $T$ qui, \`a la fin, v\'erifiera $T=\exp\{150 (\log_2 h)^2\}$.

\subsection{Premier cas :  $gyz=q_1>hT$
et $gx=q_2/z>hT.$}

Soit $S_1(h,T;\delta)$ la contribution \`a la somme $V_3(h;\delta)$ du cas o\`u $gyz>hT$ et $gx>hT.$ Nous rappelons la d\'efini\-tion~\eqref{def C1qT} et l'estimation du Lemme \ref{LemmeestC1} de $C(h,T;\delta)$. Nous montrons dans cette sous-section le r\'esultat suivant.

\begin{lemma}\label{lemme estS1} Soit $\delta\in \{ (1,1),(1,0),(0,1),(0,0)\}$. Lorsque $h\geq 2$ et $T\geq 1$, nous avons   l'estimation
    \begin{equation}\label{estS1} 
S_1(h,T;\delta)
= h C(h,T;\delta) 
+O\Big(  \frac{h}{T} (\log h)^{14}\Big).
\end{equation}
\end{lemma}

\begin{proof}
Pour estimer $V(g,x,y,z;h)$, nous utilisons une sommation d'Abel par rapport aux  deux variables $a_1$ et $a_2$ (voir Lemme~\ref{somdabel}). Nous prenons les $a_j$ tels que $|a_j|\leq q_j/2$.  Le terme principal issu de l'utilisation de l'estimation~\eqref{est1NA1A2} du Lemme~\ref{lemme estNA1A2} est trait\'e \`a l'aide de \eqref{somdabellisse}. Les termes d'erreur sont contr\^ol\'es gr\^ace aux calculs d\'ecrits au Lemme~\ref{calculmajI} et aux majorations \eqref{majprodetderivee}.
Nous obtenons
\begin{equation}\label{est4}
    \begin{split}  V(g,x,y,z;h)& = 
\sum_{\substack{1\leq |a_1|\leq q_1/2\\ (a_1,q_1)=1}} \sum_{\substack{ 1\leq |a_2|\leq q_2/2\\  a_2\equiv -a_1xy^{-1}\bmod z\\ (a_2 ,gx)=(a_1x+a_2y,g )=1}}E_h\Big(\frac{a_1}{q_1}\Big)E_h\Big(\frac{a_2}{q_2}\Big)E_h\Big(-\frac{a_1}{q_1}-\frac{a_2}{q_2}\Big)
\cr&= \frac{ { \phi_1(q_1)\phi_2(g) \phi_1(x)}}{z}I(q_1,q_2;h)
  \cr&\quad +O\Big((\log h)^2\Big\{2^{\omega(x)}3^{\omega(g)} q_1h^2+ 2^{\omega(q_1)}h^2\frac{q_2}{z}
  +h^32^{\omega(q_1)}  2^{\omega(x)}3^{\omega(g)}\Big\}\Big),
\end{split}\end{equation}
avec
\begin{align*}I(q_1,q_2;h)&:= \sum_{(\varepsilon_1,\varepsilon_2)\in \{\pm 1\}^2}\int_{0}^{q_1/2}\int_{0}^{q_2/2}E_h\Big(\frac{\varepsilon_1\alpha_1}{q_1}\Big)E_h\Big(\frac{\varepsilon_2\alpha_2}{q_2}\Big)E_h\Big(-\frac{\varepsilon_1\alpha_1}{q_1}-\frac{\varepsilon_2\alpha_2}{q_2}\Big)\d \alpha_2\d \alpha_1\cr&= \int_{-q_1/2}^{q_1/2}\int_{-q_2/2}^{q_2/2}E_h\Big(\frac{\alpha_1}{q_1}\Big)E_h\Big(\frac{\alpha_2}{q_2}\Big)E_h\Big(-\frac{\alpha_1}{q_1}-\frac{\alpha_2}{q_2}\Big)\d \alpha_2\d \alpha_1  
  =hq_1q_2 .\end{align*} 
En effet, comme le membre de droite de~\eqref{eqsomdabel}, le terme d'erreur issu de l'application du Lemme~\ref{somdabel}  comprend quatre termes. D\'etaillons le terme correspondant au dernier terme de~\eqref{eqsomdabel}. D'apr\`es les majorations du Lemme~\ref{calculmajI}, il est 
\begin{align*}
 &\ll   \int_{-q_1/2}^{q_1/2}\int_{-q_2/2}^{q_2/2}  \frac{ \big(2^{\omega(q_1)}  |t_2|/{z}+2^{\omega(x)}3^{\omega(g)}|t_1|
+2^{\omega(q_1)}  2^{\omega(x)}3^{\omega(g)}\big)  h^2\d t_2\d t_1}{(|t_1/q_1|+1/h)(|t_2/q_2|+1/h)(||t_1/q_1+t_2/q_2||+1/h)}\cr&\ll
q_1q_2
\int_{-1/2}^{1/2}\int_{-1/2}^{1/2}  \frac{ \big(2^{\omega(q_1)}   { gx|t_2|} +2^{\omega(x)}3^{\omega(g)}q_1|t_1|
+2^{\omega(q_1)}  2^{\omega(x)}3^{\omega(g)}\big)  h^2\d t_2\d t_1}{(|t_1|+1/h)(|t_2|+1/h)(||t_1+t_2||+1/h)}\cr&
\ll 
(\log h)^2\Big\{2^{\omega(x)}3^{\omega(g)} q_1h^2+ 2^{\omega(q_1)}h^2gx
  +h^32^{\omega(q_1)}  2^{\omega(x)}3^{\omega(g)}\Big\}\cr&%\textcolor{red}
  {\ll h(\log h)^2 2^{\omega(xyz)}6^{\omega(g)}\frac{g^2xyz }{T}}.
\end{align*} 
%\begin{align*}    \int_{-q_1/2}^{q_1/2}&\int_{-q_2/2}^{q_2/2}     \big(2^{\omega(q_1)}   {  t_2}/{Z}+2^{\omega(x)}3^{\omega(g)}t_1 +2^{\omega(q_1)}  2^{\omega(x)}3^{\omega(g)}\big)  \cr&\qquad\qquad h^2 E_h^+( t_1/q_1)E_h^+( t_2/q_2 )E_h^+( t_1/q_1+t_2/q_2)\d t_1\d t_2\cr&\ll q_1q_2 \int_{-1/2}^{1/2}\int_{-1/2}^{1/2}  \big(2^{\omega(q_1)}   {  q_2t_2}/{Z}+2^{\omega(x)}3^{\omega(g)}Q_1t_1 +2^{\omega(q_1)}  2^{\omega(x)}3^{\omega(g)}\big)  h^2\d t_1\d t_2 E_h^+( t_1 )E_h^+( t_2)E_h^+( t_1 +t_2 )\d t_1\d t_2\cr& \ll  (\log h)^2\Big\{2^{\omega(x)}3^{\omega(g)} Q_1h^2+ 2^{\omega(q_1)}h^2\frac{Q_2}{Z} +h^32^{\omega(q_1)}  2^{\omega(x)}3^{\omega(g)}\Big\}\end{align*} 
Nous ne fournissons pas les d\'etails pour les trois autres termes.

Lorsque $gyz>hT\geq 2,$  $gx>hT\geq 2$,    nous avons toujours 
$gyz, gxz,gxy\geq 2 .$ 
Avec les restrictions $gyz>hT,$  $gx>hT$, la contribution du terme principal est donc
$$\sum_{\substack{P(gxyz)\leq h^3 \\ gyz>hT,   gx>hT\\ (x,y,z)\in D(\delta)}}\frac{\mu(g)\mu(gxyz)^2}{\phi(g)\phi(gxyz)^2}{ \phi_1(q_1)\phi_2(g) \phi_1(x)}\frac{hq_1q_2}z  = h C(h,T;\delta)  ,$$
avec $C(h,T;\delta)$ d\'efini en \eqref{def C1qT}.
La contribution du  terme  d'erreur dans la somme $S_1(\delta)$  
est 
\begin{align*}
    &\ll  \frac{h}{T} (\log h)^2
 \sum_{\substack{P(gxyz)\leq h^3 \\ gx>hT\\ gyz>hT}}
 2^{\omega(xyz)}6^{\omega(g)}  \frac{ \mu(gxyz)^2g^2xyz}{\phi(g)\phi(gxyz)^2} \ll  \frac{h}{T } (\log h)^{14}.
\end{align*}
Cela ach\`eve la preuve du Lemme~\ref{lemme estS1}.
\end{proof} 

\subsection{Deuxi\`eme cas :  $gyz=q_1\leq h/T$
et $gx=q_2/z>h.$}

Soit $S_2^+(h,T;\delta)$ la somme des termes pris en module dans le cas   $gyz=q_1\leq h/T$
et $gx=q_2/z>hT$. Nous montrons dans cette sous-section le r\'esultat suivant.

\begin{lemma}\label{lemme estS2} Soit $\delta\in \{ (1,1),(1,0),(0,1),(0,0)\}$. Lorsque $h\geq 2$ et $T\geq 1$, nous avons   l'estimation
    \begin{equation}\label{estS2} 
S_2^+(h,T;\delta)\ll  \frac h{T}(\log h)^6 .
\end{equation}
\end{lemma}

\begin{proof}
  Nous adaptons l'argument pr\'ec\'edent au cas o\`u $Q_1$ est petit
  en majorant
 $$ V^+(g,x,y,z;h)  = 
\sum_{\substack{1\leq |a_1|\leq q_1/2 }} \sum_{\substack{ 1\leq |a_2|\leq q_2/2\\  a_2\equiv -a_1xy^{-1}\bmod z }}E_h^+\Big(\frac{a_1}{q_1}\Big)E_h^+\Big(\frac{a_2}{q_2}\Big)E_h^+\Big(-\frac{a_1}{q_1}-\frac{a_2}{q_2}\Big).$$
Une sommation d'Abel et la majoration 
$$\sum_{\substack{ 1\leq  \varepsilon_2a_2 \leq A_2\\  a_2\equiv -a_1xy^{-1}\bmod z }}1\ll\frac{A_2}z+1$$
lorsque $\varepsilon_2\in \{ \pm1\}$
fournissent  
\begin{equation}\label{est51}
    \begin{split} V^+(g,x,y,z;h)  &\ll  \sum_{\substack{1\leq |a_1|\leq q_1/2\\ (a_1,q_1)=1}}\Big\{ \frac{  1}{  z}I_2^+(a_1,q_1,q_2;h)
  + R(a_1) 
  \Big\},
\end{split}\end{equation}
avec
\begin{align*}
    I_2^+(a_1,q_1,q_2;h)&:= \sum_{ \varepsilon_2 \in \{\pm 1\} }\int_{0}^{q_2/2}E_h^+\Big(\frac{ a_1}{q_1}\Big)E_h^+\Big(\frac{\varepsilon_2\alpha_2}{q_2}\Big)E_h^+\Big(-\frac{ a_1}{q_1}-\frac{\varepsilon_2\alpha_2}{q_2}\Big)\d \alpha_2 
%\cr&\ll q_2\Big|E_h^+\Big(\frac{a_1}{q_1}\Big)\Big|^2
\end{align*}
et $R (a_1)$ le terme d'erreur v\'erifiant \begin{align*}
  R (a_1)  \ll& \frac{1 }{q_2}E_h^+\Big(\frac{a_1}{q_1}\Big)\int_{-q_2/2}^{q_2/2}\Big| 
    {E^+_h}'
    \Big(\frac{\alpha_2}{q_2}\Big)E_h^+\Big(-\frac{a_1}{q_1}-\frac{\alpha_2}{q_2}\Big)-E_h^+\Big(\frac{\alpha_2}{q_2}\Big) {E^+_h}'\Big(-\frac{a_1}{q_1}-\frac{\alpha_2}{q_2}\Big)\Big| \d \alpha_2
    \cr&+ E_h^+\Big(\frac{a_1}{q_1}\Big)\Big( E_h^+\Big(\frac{a_1}{q_1}\Big)+E_h^+\Big(\frac12+\frac{a_1}{q_1}\Big)\Big) .
\end{align*}

Nous utilisons \eqref{majprodetderivee} et la deuxi\`eme majoration du Lemme~\ref{calculmajI} de sorte  que le premier terme du majorant $R_1(a_1)$ v\'erifie
\begin{align*}
  R_1(a_1)  \ll&      \frac{h}{||a_1/q_1||}   \int_{-1/2}^{1/2} \frac{\d \alpha_2}{(||\alpha_2||+1/h)(||\alpha_2+a_1/q_1||+1/h)}\ll      \frac{h(\log h)}{||a_1/q_1||^2}  .
\end{align*}
%En distinguant le cas $|\alpha_2|\leq 1/2q_1$ du cas $|\alpha_2|> 1/2q_1$ et en utilisant  $\int_{-1/2}^{1/2} \d \alpha_2/(||\alpha_2||+1/h)\ll \log h$
Nous obtenons   
\begin{align*}
    R(a_1)\ll&    h(\log h)\Big( \frac{1}{||a_1/q_1||^2} 
    + \frac{q_1}{||a_1/q_1||}\Big)
    \ll  h(\log h)    \frac{q_1}{||a_1/q_1||}   
\end{align*}  
et
$$\sum_{\substack{1\leq |a_1|\leq q_1/2\\ (a_1,q_1)=1}}
    R (a_1)\ll    h{q_1^2}(\log h)(\log q_1) . 
    $$ 
    Nous avons aussi  $$
 I_2^+(a_1,q_1,q_2;h) \ll  \frac{ q_2  (\log h)}{||a_1/q_1||^2}. $$
 %Nous avons  \begin{align*} I_2(a_1,q_1,q_2;h) &=  \int_{-q_2/2}^{q_2/2}E_h\Big(\frac{ a_1}{q_1}\Big)E_h\Big(\frac{ \alpha_2}{q_2}\Big)E_h\Big(-\frac{ a_1}{q_1}-\frac{ \alpha_2}{q_2}\Big)\d \alpha_2+O\Big(\Big|E_h\Big(\frac{ a_1}{q_1}\Big)\Big|^2\frac{h^2}{q_2} \Big) \cr &=  \Big|E_h\Big(\frac{ a_1}{q_1}\Big)\Big|^2\Big\{ q_2+O\Big( \frac{h^2}{q_2} \Big)\Big\}.\end{align*}

En utilisant $h<gx$, $q_1\leq h/T$, nous obtenons $$ V^+_h(g,x,y,z)\ll \frac{q_1^2q_2(\log h)}{z}+ hq_1^2(\log h)(\log q_1)\ll \frac{q_1^2q_2}{z}{(\log h)^2} 
\ll \frac{hq_1q_2(\log h)^2}{zT}.$$ Il en d\'ecoule
\begin{align*}\sum_{\substack{P(gxyz)\leq h^3 \\ gyz\leq h/T,   gx>hT\\ (x,y,z)\in D(\delta)}}\frac{ \mu(gxyz)^2 V^+(g,x,y,z;h) }{\phi(g)\phi(gxyz)^2 }%&\ll \sum_{\substack{P(gxyz)\leq h^3 \\ q_1\leq h/T,   g x>hT\\ (x,y,z)\in D(\delta)}}\frac{ \mu(gxyz)^2q_1^2q_2}{\phi(g)\phi(gxyz)^2z}   2^{\omega(x)}3^{\omega(g)} {(\log h)^2} 
%\cr &\ll \frac hT\sum_{\substack{P(gxyz)\leq h^3 \\ q_1\leq h/T }}\frac{ \mu(gxyz)^2(g^2xyz)q_1/h}{\phi(g)^3\phi(xyz)^2} 2^{\omega(x)}3^{\omega(g)} {(\log h)^2}  \cr&
&\ll \frac{h(\log h)^2}T\sum_{\substack{P(gxyz)\leq h^3 }}\frac{ \mu(gxyz)^2g^2xyz }{\phi(g)^3\phi(xyz)^2}      \ll  \frac h{T}(\log h)^6.\end{align*}
%Cette contribution est donc n\'egligeable d\`es que $T\geq (\log h)^5.$
\end{proof}

\section{Deuxi\`eme m\'ethode de comptage}

\subsection{R\'eduction du domaine pour les valeurs de $a_1$ et $a_2.$} 
 Il reste donc \`a consid\'erer le cas $h/T<gyz \leq hT$
et $gx >hT $ et le cas $gx\leq hT. $ 
Soit ${\mathcal E}$ un ensemble pour lesquels une des quatre in\'egalit\'es suivantes est v\'erifi\'ee 
$$|a_1|>TQ_1/h, \quad |a_2|>TQ_2/h, \quad |a_1|\leq Q_1/hT, \quad |a_2|\leq Q_2/hT.$$
Nous notons $\Sigma$ la somme des modules des termes somm\'es sur un tel ensemble.

\begin{lemma}\label{lemme Sigma}
    Nous avons
    $$\Sigma\ll \frac{h (\log h)^{8} }{  T } .$$
\end{lemma}

\begin{proof}
  Nous \'ecrirons $x\sim X$ pour indiquer la relation $x\in [X,2X[$. Nous nous restreindrons au cas $x\sim X$, $y\sim Y$, $||a_1/q_1||\sim A_1/q_1$, $||a_2/q_2||\sim A_2/q_2$,  avant de sommer en $X$, $Y$ et $A_1$, $A_2$ o\`u $A_1$, $A_2$ sont des puissances de $2$.

%{\color{blue}  Nous \'ecrivons  $$\frac{a_2}{q_2}=\frac{a}{z}-\frac{b}{gx},\quad \frac{a_1}{q_1}=-\frac{a}{z}+\frac{c}{gy} $$  avec
%$$1\leq a\leq z,\quad  (a,z)=1,\qquad 1\leq b\leq gx,\quad  (b,gx)=1,\qquad \quad 1\leq c\leq gy,\quad  (c,gy)=1. $$
%Dans ce cas, $a_3=-q_3(a_1/q_1+a_2/q_2)\equiv by-cx(\bmod gyz).$ Notre m\'ethode repose sur l'observation suivante. Lorsque $a_2$ est fix\'e, les valeurs $b$ et $x$  v\'erifient $$bz\equiv -a_2(\bmod gx),\qquad ag x\equiv  a_2(\bmod z).$$ Donc $b$ est enti\`erement fix\'e et $x$ v\'erifie une congruence modulo $z$. Lorsque $a_1$ est fix\'e, les valeurs $c$ et $y$   v\'erifient $$ cx\equiv  a_1(\bmod gy),\qquad ag y\equiv  -a_1(\bmod z).$$ Donc $c$ est enti\`erement fix\'e et $y$ v\'erifie  une congruence modulo $z$.} 

Nous devons majorer la somme suivante prise sur $\mathcal E$
\begin{equation}\label{expressionT}
    \begin{split}  
\Sigma (A_1,A_2,X,Y)&=   
 \sum_{\substack{g,z \mid q \\ X\geq Y\geq Z%\\ A_2 \leq Q_2/2 , A_1 \leq Q_1/2 
 }}\frac{  \mu(gz )^2}{\phi(g)^3\phi( z)^2 } \sum_{\substack{| a_2 |\sim A_2 \\  | a_1 |\sim A_1\\ (a_2a_1,gz)=1 }}\sum_{\substack{ Y\leq y<2Y\\ (y,a_1gz)=1}}\frac{\mu(y)^2 }{\phi(y)^2}
 \sum_{\substack{ X\leq x<2X\\ P(x)\leq h^3\\ a_1x\equiv - a_2 y \bmod z\\ (x,a_2gy)=1\\ (a_1x+a_2 y,g)=1}}\frac{\mu(x)^2 }{\phi(x)^2}\cr&\quad\qquad\qquad\times 
\Big| E_h\Big( -\frac{a_2}{gzx}-\frac{a_1}{gzy}  \Big)  E_h\Big( \frac{a_2}{gzx} \Big)E_h\Big( \frac{a_1}{gzy} \Big)\Big|.
\end{split}
\end{equation}

Pour l'estimer, nous introduisons la quantit\'e
\begin{equation}\label{expression2T}
    \begin{split}  
\sigma^+(g,z,a_1,a_2) &:=    \sum_{\substack{ Y\leq y<2Y\\ (y,a_1gz)=1}}\frac{\mu(y)^2 y^2}{\phi(y)^2}
 \sum_{\substack{ X\leq x<2X\\ P(x)\leq h^3\\ a_1x\equiv - a_2 y\bmod z \\ (x,a_2gy)=1\\ (a_1x+a_2 y,g)=1}}\frac{\mu(x)^2 x^2}{\phi(x)^2}|F(x,y,gz,a_1,a_2)|,
\end{split}
\end{equation}
avec
\begin{equation}
    \label{defF}
F(x,y,z, a_1,a_2):=\frac{1}{x^2y^2}E_h\Big( -\frac{a_2}{xz}-\frac{a_1}{yz}  \Big) E_h\Big( \frac{a_2}{xz} \Big)E_h\Big( \frac{a_1}{yz} \Big).\end{equation}
%Nous avons 
%\begin{equation}\label{expressionT2}    \begin{split}  
%\Sigma (A_1,A_2,X,Y) & =  \sum_{\substack{g,z \mid q \\ X,Y\geq Z
%\\ A_2 \leq Q_2/2 , A_1 \leq Q_1/2 
%}}\frac{  \mu(gz )^2}{\phi(g)^3\phi( z)^2 } \sum_{\substack{| a_2 |\sim A_2 \\  | a_1 |\sim A_1\\ (a_2a_1,gz)=1 }} \sigma^+(g,z,a_1,a_2).
%\end{split} \end{equation}

%{\color{red}paragraphe \`a d\'eplacer}
Gr\^ace \`a \eqref{majprodetderivee}, nous avons les majorations suivantes
\begin{equation}\label{majpartialF}
    \begin{split} 
    F(x,y,gz,a_1,a_2)&\ll \frac{h^3}{X^2Y^2(1+{h||a_1/q_1||})(1+{h||a_2/q_2||})(1+{h ||a_1/q_1+ a_2/q_2||)}  }%\cr&\qquad\ll \frac{h^3}{X^2Y^2(1+{hA_1}/{Q_1})(1+{hA_2}/{Q_2})(1+{hA_1}/{Q_1}+{hA_2}/{Q_2}) }
    ,\cr
    \partial_1F(x,y,gz,a_1,a_2)&
    \ll \frac{h^3}{X^3Y^2(1+{h||a_1/q_1||}) (1+{h ||a_1/q_1+ a_2/q_2||)}  }
    %\ll \frac{h^3}{X^3Y^2(1+{hA_1}/{Q_1})(1+{hA_1}/{Q_1}+{hA_2}/{Q_2}) }
    %\Big\{ 1+ \frac{hA_2}{Q_2} \Big\}
    ,\cr
    \partial_2F(x,y,gz,a_1,a_2)&\ll \frac{h^3}{X^2Y^3 (1+{h||a_2/q_2||})(1+{h ||a_1/q_1+ a_2/q_2||)}  }%\ll \frac{h^3}{X^2Y^3 (1+{hA_2}/{Q_2})(1+{hA_1}/{Q_1}+{hA_2}/{Q_2}) }
    %\Big\{ 1+\frac{hA_1}{Q_1} \Big\}
    ,\cr
    \partial_1\partial_2F(x,y,gz,a_1,a_2)&\ll \frac{h^3}{X^3Y^3 (1+{h  ||a_1/q_1+ a_2/q_2||)}  }
    %\ll \frac{h^3}{X^3Y^3(1+{hA_1}/{Q_1}+{hA_2}/{Q_2}) }.
    %\Big\{ 1+\frac{hA_1}{Q_1}+\frac{hA_2}{Q_2}+\frac{h^2A_1A_2}{Q_1Q_2}\Big\}. 
    \end{split}
\end{equation} 

Nous pouvons ais\'ement \'ecarter les cas o\`u $A_1> TQ_1/h$ ou $A_2> TQ_2/h$. Prenons comme exemple le cas o\`u $A_2/Q_2\notin [ A_1/4Q_1,4A_1/ Q_1] $ de sorte que $
1+h ||a_1/q_1+ a_2/q_2||\asymp 1+{hA_1}/{Q_1}+{hA_2}/{Q_2} $.
Alors gr\^ace \`a la majoration de $F$, nous avons 
\begin{align*} \sigma^+(g,z,a_1,a_2)&\ll \frac{h^3}{X^2Y^2(1+{hA_1}/{Q_1})(1+{hA_2}/{Q_2})(1+{hA_1}/{Q_1}+{hA_2}/{Q_2}) }\cr&\qquad\qquad\times
 \sum_{\substack{ X\leq x<2X\\ P(x)\leq h^3\\ (x,a_2g )=1}}\frac{\mu(x)^2 x^2}{\phi(x)^2} \sum_{\substack{ Y\leq y<2Y\\ a_2 y\equiv -a_1x  \bmod z\\ (y,a_1gzx)=1}}\frac{\mu(y)^2 y^2}{\phi(y)^2}
 \cr&\ll \frac{h^3(\log h)^{2}}{X Y Z(1+{hA_1}/{Q_1})(1+{hA_2}/{Q_2})(1+{hA_1}/{Q_1}+{hA_2}/{Q_2}) }
 \sum_{\substack{ X\leq x<2X\\ P(x)\leq h^3 }}\frac{\mu(x)^2  }{x} ,
   \end{align*}
   o\`u pour passer \`a la derni\`ere ligne nous avons utilis\'e la majoration 
$$\frac{  x^2}{\phi(x)^2}\frac{  y^2}{\phi(y)^2}=\frac{  (xy)^2}{\phi(xy)^2} 
\ll   (\log h)^2.$$
    
En sommant sur $a_1\sim A_1$ et $a_2\sim A_2,$ et en utilisant $1+{hA_1}/{Q_1}+{hA_2}/{Q_2}>T,$ nous obtenons  
\begin{align*} \sum_{\substack{| a_1 |\sim A_1 \\  | a_2 |\sim A_2\\ (a_1a_2,gz)=1 }} &\sigma^+(g,z,a_1,a_2)\cr&\ll \frac{h^3(\log h)^{2}A_1A_2}{X Y Z(1+{hA_1}/{Q_1})(1+{hA_2}/{Q_2})(1+{hA_1}/{Q_1}+{hA_2}/{Q_2}) }
\sum_{\substack{ X\leq x<2X\\ P(x)\leq h^3 }}\frac{\mu(x)^2  }{x}
%&\ll \frac{h^3(\log h)^{2}A_1A_2}{X Y zT(1+{hA_1}/{Q_1})(1+{hA_2}/{Q_2})  }\sum_{\substack{ X\leq x<2X\\ P(x)\leq h^3 }}\frac{\mu(x)^2  }{x}
%\ll \frac{h (\log h)^{2}Q_1Q_2}{X Y z T }\ll \frac{h (\log h)^{4}g^2z}{  T }
.\end{align*}
Puis nous sommons  par rapport \`a $A_1,A_2$, $g,z$ $X,Y,$ nous obtenons une 
 contribution
 $$\ll \sum_{X,Y,g,z}\frac{h (\log h)^{4}}{ T }
 \frac{\mu(gz )^2g^2z}{\phi(g)^3\phi( z)^2 }\ll \frac{h (\log h)^{8} }{  T } .$$ 
L'int\'er\^et d'avoir gard\'e la somme $$  \sum_{\substack{ X\leq x<2X\\ P(x)\leq h^3 }}\frac{\mu(x)^2  }{x}$$ en $x$ provient du fait que le nombre  de valeurs possibles  de $X$  n'est pas forc\'ement un $O(\log h)$ dans le cas o\`u $gx>hT $ alors que 
$$\sum_X    \sum_{\substack{ X\leq x<2X\\ P(x)\leq h^3 }}\frac{\mu(x)^2  }{x}\leq \sum_{\substack{ P(x)\leq h^3 }}\frac{\mu(x)^2  }{x}\ll \log h.$$

Majorons maintenant  le cas o\`u $A_2/Q_2\in [ A_1/4Q_1,4A_1/ Q_1] $, $A_1/Q_1\gg T/h$ et $A_2/Q_2\gg T/h$. Dans un premier temps,  nous traitons le cas $Q_2\leq h$ en particulier le nombre de $X$ sur lesquels sommer est $O(\log h)$.
Nous majorons
\begin{equation}\label{expressionSigma'}
    \begin{split}  
\Sigma' (A_1,A_3,X,Z)&=   
 \sum_{\substack{g,y \mid q \\ X,Y\geq Z%\\ A_2 \leq Q_2/2 , A_1 \leq Q_1/2 
 }}\frac{  \mu(gy )^2}{\phi(g)^3\phi( y)^2 } \sum_{\substack{| a_3 |\sim A_3 \\  | a_1 |\sim A_1\\ (a_3a_1,gy)=1 }}\sum_{\substack{ Z\leq z< 2Z\\ (z,a_1gy)=1}}\frac{\mu(z)^2 }{\phi(z)^2}
 \sum_{\substack{ X\leq x<2X\\ a_1x\equiv - a_3z(\!\bmod y)\\ (x,a_3gz)=1\\ (a_1x+a_3z,g)=1}}\frac{\mu(x)^2 }{\phi(x)^2}\cr&\quad\qquad\qquad\times 
  E_h^+\Big(\frac{a_1}{gyz}+\frac{a_3}{gxy}  \Big)  E_h^+\Big( \frac{a_1}{gyz} \Big)E_h^+\Big( \frac{a_3}{gxy} \Big) .
\end{split}
\end{equation}
Nous notons $\sigma'(g,y,a_1,a_3)$ la double somme int\'erieure. Nous imposons $|a_3|\sim A_3.$
Puisque $Y\leq X$, nous avons alors 
\begin{align*}\sigma'(g,y,a_1,a_3)&\ll \frac{h^3}{X^2Z^2(1+{hA_1}/{Q_1})^2 (1+{hA_3}/{Q_3}) }
  \sum_{\substack{ Z\leq z< 2Z\\ (z,a_1gy)=1}}\frac{\mu(z)^2z^2 }{\phi(z)^2}\!\!\!
 \sum_{\substack{ X\leq x<2X\\ a_1x\equiv - a_3z(\!\bmod y)\\ (x,a_3gz)=1\\ (a_1x+a_3z,g)=1}}\frac{\mu(x)^2x^2 }{\phi(x)^2}
 \cr&\ll \frac{h^3(\log h)^{2}}{X Z Y(1+{hA_1}/{Q_1})^2(1+{hA_3}/{Q_3}) }.
   \end{align*}
La suite de la preuve est semblable \`a l'obtention de la  majoration pr\'ec\'edente. Ici comme $X\leq Q_2\leq h$, le nombre  de valeurs possibles  de $X$ est un $O(\log h).$ La contribution totale de ce cas-l\`a est encore $\ll h(\log h)^8/T.$

 Nous traitons le cas $Q_2>h$. 
Nous r\'epartissons dans  des intervalles dyadiques l'ensemble des valeurs possibles de $||a_1/q_1+ a_2/q_2||$ de sorte que $||a_1/q_1+ a_2/q_2||\sim A_3/Q_3$ avec $Q_3=gXY.$
Nous avons alors
\begin{align*}\sigma^+(g,z,a_1,a_2)&\ll \frac{h^3}{X^2Y^2(1+{hA_1}/{Q_1})^2 (1+{hA_3}/{Q_3}) }
  \sum_{\substack{ X\leq x<2X\\ P(x) \leq h^3\\ (x,a_2gz)=1}}\frac{\mu(x)^2 x^2}{\phi(x)^2}
 \!\!\!\!\!\! 
 \sum_{\substack{ Y\leq y<2Y\\ a_2 y\equiv -a_1x  \bmod z\\ (y,a_1gx)=1}}\frac{\mu(y)^2 y^2}{\phi(y)^2}
 \cr&\ll \frac{h^3(\log h)^{2}}{X Y Z(1+{hA_1}/{Q_1})^2(1+{hA_3}/{Q_3}) }
  \sum_{\substack{ X\leq x<2X \\ P(x) \leq h^3}}\frac{\mu(x)^2 x }{\phi(x)^2}.
   \end{align*}
 L'intervalle de variation de $a_2$ lorsque $a_1$ est fix\'e est d'une taille $\ll A_3Q_2/Q_3 $ de sorte que la contribution de la somme $\sigma^+(g,z,a_1,a_2)$ sur $a_1$ et $a_2$  soit 
\begin{align*}
   & \ll \frac{h^3(\log h)^{2}A_1(A_3Q_2/Q_3+1)}{X Y Z(1+{hA_1}/{Q_1})^2(1+{hA_3}/{Q_3}) }\sum_{\substack{ X\leq x<2X }}\frac{\mu(x)^2 x }{\phi(x)^2}
 \cr&\ll \frac{h (\log h)^{2}Q_1Q_2}{X Y ZT }\Big(1+\frac{h}{Q_2}\Big)
 \sum_{\substack{ X\leq x<2X \\ P(x) \leq h^3}}\frac{\mu(x)^2 x }{\phi(x)^2}\ll \frac{h (\log h)^{4}g^2z}{  T }\sum_{\substack{ X\leq x<2X\\ P(x) \leq h^3 }}\frac{\mu(x)^2 x }{\phi(x)^2}
 \end{align*}
ce qui permet de conclure d'apr\`es les calculs pr\'ec\'edents.

Ensuite de la m\^eme mani\`ere, nous   pouvons \'egalement \'ecarter les cas o\`u $A_1\leq  Q_1/hT$ ou $A_2\leq  Q_2/hT$. 
Prenons le cas $A_1\leq  Q_1/hT$. Les m\^emes majorations fournissent
une contribution globale de 
\begin{align*}
    &\ll \sum_{\substack{P(gz) \leq h^3  \\ X,Y\geq Z\\ A_2 \leq Q_2/2 , A_1 \leq Q_1/Th }}\frac{  \mu(gz )^2}{\phi(g)^3\phi( z)^2 } \frac{A_1Q_2 h^2(\log h)^2}{XYZ}\sum_{\substack{ X\leq x<2X\\ P(x) \leq h^3 }}\frac{\mu(x)^2 x }{\phi(x)^2}
\cr&\ll \sum_{\substack{P(gz) \leq h^3 \\ X,Y\geq Z    }}\frac{  \mu(gz )^2}{\phi(g)^3\phi( z)^2 } \frac{Q_1Q_2 h (\log h)^3}{TXYZ}
\sum_{\substack{ X\leq x<2X\\ P(x) \leq h^3  }}\frac{\mu(x)^2 x }{\phi(x)^2}\ll \frac{  h (\log h)^7}{T}.
\end{align*} 
\end{proof}
\goodbreak

\subsection{Troisi\`eme cas :  $q_1=gyz > hT$,  $gx\leq hT,$ $x> zT^4$.}

Soit $S_3(h,T;\delta)$ la contribution du cas o\`u  $gyz > hT$,  $gx\leq hT,$ $x> zT^4$, $(x,y,z)\in D(\delta)$. Nous montrons dans cette sous-section le r\'esultat suivant. 

\begin{lemma}\label{lemme estS3} Soit $\delta\in \{ (1,1),(1,0),(0,1),(0,0)\}$. Lorsque $\tfrac 12 h\geq  T\geq 1$, nous avons   l'estimation
    \begin{equation}\label{estS3} 
S_3(h,T;\delta)
= %\textcolor{red}
{\frac{1}{8}} h(\log h)^2   
  +O \Big(h(\log T)^2\log h+ \frac{h(\log h)^{10}}{{\cal L}(T)}+\frac{h (\log h)^{12}}{T }\Big).
\end{equation}
\end{lemma}

\begin{proof}    Dans le d\'ebut de la d\'emonstration, la condition $q_1> hT$ n'est pas n\'ecessaire.   \`A la section pr\'ec\'edente, nous avons  montr\'e que nous pouvions nous restreindre au cas
$$ Q_1/Th <|a_1|\leq  Q_1T/h , \qquad Q_2/Th <|a_2|\leq  Q_2T/h .  $$
L'in\'egalit\'e 
$q_1\geq Q_1\geq h/T$ d\'ecoule de ces in\'egalit\'es. 

Dans la suite de la d\'emonstration, nous nous restreignons au cas $\delta=(1,1)$, les autres cas pouvant \^etre obtenus de la m\^eme mani\`ere.
 Nous consid\'erons la valeur   de la contribution du cas o\`u   $hT/ g\geq  x\geq y\geq z ,$ $ x\geq   zT^4$, $Q_1/Th <|a_1|\leq  Q_1T/h ,$ $ Q_2/Th <|a_2|\leq  Q_2T/h$. La condition $hT/g\geq  x$ implique que la condition de $h^3$-friabilit\'e sur $x$ et $y$ est automatiquement v\'erifi\'ee. %\textcolor{red}
 {Nous rajouterons la condition $y>hT/hz$ en cours de d\'emonstration. }

Pour  estimer cette somme, nous introduisons la quantit\'e
\begin{equation}\label{defs}
    \begin{split}  
\sigma(g,z,a_1,a_2) &:= \sum_{\substack{X\leq x< 2X,\, Y\leq y<2Y \\ a_1x+a_2 y\equiv  0 \bmod z\\  (y,a_1gzx)=1
,\,     (x,a_2gz)=1\\ (a_1x+a_2 y,g)=1\\ x\geq y}}\frac{\mu(x)^2x^2 }{\phi(x)^2} \frac{\mu(y)^2y^2 }{\phi(y)^2} 
 F(x,y,gz,a_1,a_2),
\end{split}
\end{equation}
avec $F$ d\'efinie en 
    \eqref{defF}. 
 %\textcolor{red}
 {Lorsque $hT/g\in [X,2X[$, on remplace l'intervalle de sommation $[X,2X[$ en $x$ par $[X,hT/g]$.
 Lorsque $zT^4\in [X,2X[$, on remplace l'intervalle de sommation $[X,2X[$ en $x$ par $[zT^4,2X[$. Lorsque $Y=Z,$ on remplace l'intervalle de sommation $[Y,2Y[$ en $y$ par $[z,2Y[$. Lorsque $hT/gz\in [Y,2Y[$, on remplace l'intervalle de sommation   $[Y,2Y[$ en $y$ par $[hT/gz,2Y]$. Ces diff\'erentes r\'eductions ne pr\'esentent pas de difficult\'e suppl\'ementaire. Enfin, nous expliquerons en fin de d\'emonstration comment tenir compte de la condition $x\geq y$ qui intervient que lorsque $X=Y.$}

Nous aurons aussi besoin de  la somme
\begin{align*}W_{X,Y}^*(X',Y'; g,z,a_1,a_2)&:= \sum_{\substack{X\leq x<  X',\, Y\leq y<  Y'\\ a_1x+a_2 y\equiv  0 \bmod z\\  (y,a_1gzx)=1
,\,     (x,a_2gz)=1\\ (a_1x+a_2 y,g)=1\\ x\geq y}}\frac{\mu(x)^2x^2 }{\phi(x)^2} \frac{\mu(y)^2y^2 }{\phi(y)^2}
.\end{align*}
 
 La condition $x\geq y$ est impliqu\'ee par $X>Y$. Elle n'est donc utile que si $X=Y$. Nous commen\c cons par nous placer dans le cas $X>Y$.
 
Du Lemme~\ref{lemme sumW}, nous d\'eduisons lorsque $X'\leq 2X$ et $Y'\leq 2Y$, $Y\leq X$
l'estimation \begin{equation}
    \label{estsumWXYLB*}\begin{split}
 W_{X,Y}^* &(X', Y'; g,z,a_1,a_2) 
=\frac {\zeta(2)}{\zeta(4) }(X'-X)(Y'-Y)\frac{w^*(g,z,a_1,a_2)}{\phi(z)}\cr&\quad+O\Big(2^{\omega(g)}\psi_{1/3}(a_1g)^2\psi_{1/3}(a_2g)^2(XY)^{1/2 }
(\log (XY))^{6}   2^{\omega(z)}
\Big(1+\frac{ X^{1/2 }{\cal L}(X)^{-1} }{z }  \Big)\Big) .\end{split}\end{equation}

Nous indiquons rapidement comment modifier l'estimation de $W_{X,Y}^*(X', Y'; g,z,a_1,a_2) $ lorsque $X=Y$.
Nous consid\'erons
\begin{align*}
 I_{r}&=[X+rX/T, \min\{X+(r+1)X/T,X'\}[,\,\, I'_{r}=[X+rX/T,
 \min\{X+(r+1)X/T,Y'\}[,\cr & J_{r}=[X ,\min\{X+rX/T,Y'\}[. \end{align*}
Lorsque $X=Y$, nous recouvrons le domaine 
$${\mathcal D}:=\{ (x,y):\, X\leq x<  X',\, Y\leq y<  Y', x\geq y\}$$ par les $I_r\times J_r$ et par les $I_r\times I'_r.$
Il est clair que nous pouvons supposer $X'\geq Y'$ quitte \`a~remplacer $Y':=\min \{ Y',X'\}.$

La contribution de l'intervalle $I_r\times J_r$ dans la somme 
$W_{X,Y}^*(X', Y'; g,z,a_1,a_2) $ a pour terme principal
$$\frac {\zeta(2)}{\zeta(4) }{\rm Vol}(I_r\times J_r)\frac{w^*(g,z,a_1,a_2)}{\phi(z)}.$$
Le terme d'erreur est le m\^eme que celui de \eqref{estsumWXYLB*}
avec $X=Y$.
La contribution de tous les intervalles $I_r\times I'_r$ dans la somme 
$W_{X,Y}^*(X', Y'; g,z,a_1,a_2) $ est major\'ee par 
$$\sum_{0\leq r\leq T} \frac{X}{T}\Big( \frac{X}{zT}+1\Big)\ll \frac{X^2}{zT}. $$

Nous en d\'eduisons que lorsque $X=Y$
\begin{equation}
    \label{estsumWXYLB*X=Y}\begin{split}
 W&_{X,Y}^*  (X', Y'; g,z,a_1,a_2) 
=\frac {\zeta(2)}{\zeta(4) }
{\rm Vol}({\mathcal D})
 \frac{w^*(g,z,a_1,a_2)}{\phi(z)}+O\Big(\frac{X^2}{zT}\Big)\cr&\quad+O\Big(T2^{\omega(g)}\psi_{1/3}(a_1g)^2\psi_{1/3}(a_2g)^2(XY)^{1/2 }
(\log (XY))^{6}   2^{\omega(z)}
\Big(1+\frac{ X^{1/2 }{\cal L}(X)^{-1} }{z }  \Big)\Big) .\end{split}\end{equation}

Une double sommation d'Abel fournit la formule
\begin{align*} 
\sigma(g,z,&a_1, a_2)  =F(2X,2Y)W_{X,Y}^*(2X,2Y)
 -\int_X^{2X} W_{X,Y}^*(t_1,2Y)\partial_1 F(t_1, 2Y)\d t_1\\ &\quad -\int_Y^{2Y} W_{X,Y}^*(%\textcolor{red}
 {2}X,t_2)\partial_2 F(2X,t_2 )\d t_2 +\int_X^{2X}\int_Y^{2Y} W_{X,Y}^*(t_1,t_2)\partial_{1,2}^2 F(t_1,t_2)\d t_2\d t_1,
   \end{align*}
   o\`u, pour simplifier la formule, nous n'avons pas indiqu\'e la d\'ependance en $( g,z,a_1,a_2)$ de $F(x,y,gz,a_1,a_2)$ et de $W_{X,Y}^*(t_1,t_2; g,z,a_1,a_2).$ Nous avons ici utilis\'e la remarque \ref{rem<} afin de sommer sur les $x$ tels que $X\leq x<2X$ et non $X<x\leq 2X$.

Nous pouvons donc appliquer \eqref{estsumWXYLB*} en utilisant les majorations \eqref{majpartialF}.
 La contribution du terme d'erreur du Lemme~\ref{lemme sumW} lorsque $X>Y$ dans l'estimation de $\sigma(g,z,a_1,a_2) $
est 
\begin{align*}& \ll
 \frac{2^{\omega(g)}\psi_{1/3}(a_1g)^2\psi_{1/3}(a_2g)^22^{\omega(z)}
(\log h)^{6}    h^3(1+ { X^{1/2 } {\cal L}(X)^{-1}}/{z }  ) }{X^{3/2 }Y^{3/2 }%(1+{hA_1}/{Q_1}+{hA_2}/{Q_2})
(1+{hA_1}/{Q_1} )(1+ {hA_2}/{Q_2}) }
\cr& \ll
 \frac{2^{\omega(g)}\psi_{1/3}(g)^42^{\omega(z)}
(\log h)^{6}  h Q_1Q_2(1+ { X^{1/2 } {\cal L}(X)^{-1}}/{z }  ) }{X^{3/2 }Y^{3/2 }%(1+{hA_1}/{Q_1}+{hA_2}/{Q_2})  
}\frac{h^2}{Q_1Q_2}\frac{\psi_{1/3}(a_1)^2\psi_{1/3}(a_2)^2}{(1+{hA_1}/{Q_1} )(1+ {hA_2}/{Q_2})}
%\cr& \ll \frac{2^{\omega(g)}\psi_{1/2}(g)^44^{\omega(z)} (\log h)^{6} h Q_1Q_2(1+ { X^{1/2 } {\cal L}(X)^{-1}}/{z }  ) }{X^{3/2 }Y^{3/2 }   }\frac{h^2}{Q_1Q_2}\frac{\psi_{1/2}(a_1)^2\psi_{1/2}(a_2)^2}{(1+{hA_1}/{Q_1} )(1+ {hA_2}/{Q_2})}
. 
\end{align*}
Or 
\begin{equation}
    \label{majsommeA1A2}
    \begin{split}
\sum_{\substack{A_1,A_2}}\sum_{\substack{a_1\sim A_1\\ a_2\sim A_2}}
\frac{\psi_{1/3}(a_1)^2\psi_{1/3}(a_2)^2}{(1+{hA_1}/{Q_1} )(1+ {hA_2}/{Q_2})}
&\ll
\sum_{\substack{A_1,A_2}} 
\frac{A_1A_2 }{(1+{hA_1}/{Q_1} )(1+ {hA_2}/{Q_2})}\cr&
\ll \frac{Q_1Q_2}{h^2}(\log Q_1)(\log Q_2),
\end{split}
\end{equation}
donc en sommant sur $a_1,a_2,A_1,A_2,$ puis $g,z,X,Y$, nous obtenons un terme d'erreur global
\begin{align*}&   \ll(\log h)^{8} \!\! \!\! \!\!\!\!\sum_{\substack{P(gz)\leq h^3  \\  hT\geq X\geq Y\geq Z , X\geq   ZT^4 }}\!\!\!\!\!\!\!\!\!\!
\frac{2^{\omega(g)} \mu(gz )^2g^2\psi_{1/3}( g)^42^{\omega(z)}z^2}{\phi(g)^3\phi( z)^2 }\Big(1+\frac{  X^{1/2 } {\cal L}(X)^{-1}}{z }  \Big)\frac{hT }{X^{1/2}Y^{1/2}} \cr&\ll h(\log h)^{8}\sum_{\substack{P(gz)\leq h^3   }}
\frac{2^{\omega(g)} \mu(gz )^2g^2\psi_{1/3}( g)^42^{\omega(z)}z }{\phi(g)^3\phi( z)^2  } \Big(\frac1T+\frac{1 }{{\cal L}(T)  z^{1/2}}  \Big) 
\cr&\ll\frac{h (\log h)^{12}}{T }+ \frac{h(\log h)^{10}}{{\cal L}(T)}.
\end{align*} De la mani\`ere, la contribution du terme d'erreur lorsque $X=Y$ est aussi englob\'ee par ce majorant.
 D'apr\`es l'identit\'e issue de \eqref{somdabellisse}
 \begin{align*}
     \int_{X}^{2X}\int_{Y}^{2Y}&F(x,y)\d x\d y =F(2X,2Y)XY
 -\int_{X}^{2X}  (t_1-X)Y\partial_1 F(t_1,2Y)\d t_1\\ &  -\int_{Y}^{2Y} X(t_2-Y)\partial_2 F(2X,t_2)\d t_2+\int_{X}^{2X}\int_{Y}^{2Y} (t_1-X) (t_2-Y) \partial_{1,2}^2 F(t_1,t_2)\d t_2\d t_1,
 \end{align*} lorsque $X>Y$
le terme principal vaut
\begin{equation}\label{defsigmaM}
    \begin{split} 
&\sigma_M
 (g,z,a_1,a_2) :=\frac {\zeta(2)}{\zeta(4) }  \frac{w^*(g,z,a_1,a_2)}{\phi(z)}\int_{X}^{2X}\int_{Y}^{2Y}F(x,y,gz,a_1,a_2)\d y\d x
\cr&%\!\!\!\!
=
\frac {\zeta(2)}{\zeta(4) } \frac{w^*(g,z,a_1,a_2)}{\phi(z)XY}\int_{1}^{2}\!\!\!\!\int_{1}^{2}
E_h\Big( -\frac{a_2}{gzu_1X}-\frac{a_1}{gzu_2Y}  \Big) E_h\Big( \frac{a_2 }{gzu_1X} \Big)E_h\Big( \frac{a_1}{gzu_2Y} \Big)\frac{\d u_1\d u_2}{u_1^2u_2^2}
%\cr&=\frac {\zeta(2)}{\zeta(4) } \frac{w^*(g,z,a_1,a_2)}{\phi(z)XY}\int_{1/2}^{1}\int_{1/2}^{1}E_h\Big( -\frac{a_2u_2}{gzX}-\frac{a_1u_1}{gzY}  \Big) E_h\Big( \frac{a_2u_2}{gzX} \Big)E_h\Big( \frac{a_1u_1}{gzY} \Big)\d u_1\d u_2
\end{split}
\end{equation}
avec $w^*$ d\'efini en \eqref{defw}.

Lorsque $X=Y$
le terme principal vaut
\begin{equation}\label{defsigmaMX=Y}
    \begin{split} 
&\sigma_M
 (g,z,a_1,a_2) :=\frac {\zeta(2)}{\zeta(4) }  \frac{w^*(g,z,a_1,a_2)}{\phi(z)}\int_{X}^{2X}\int_{Y}^{2Y}F(x,y,gz,a_1,a_2)1_{x\geq y}\d y\d x
\cr&=
\frac {\zeta(2)}{\zeta(4) } \frac{w^*(g,z,a_1,a_2)}{\phi(z)XY}\int_{1}^{2}\!\!\!\int_{1}^{2}
E_h\Big( -\frac{a_2}{gzu_1X}-\frac{a_1}{gzu_2Y}  \Big) E_h\Big( \frac{a_2 }{gzu_1X} \Big)E_h\Big( \frac{a_1}{gzu_2Y} \Big)1_{u_1\geq u_2}\frac{\d u_1\d u_2}{u_1^2u_2^2}
%\cr&=\frac {\zeta(2)}{\zeta(4) } \frac{w^*(g,z,a_1,a_2)}{\phi(z)XY}\int_{1/2}^{1}\int_{1/2}^{1}E_h\Big( -\frac{a_2u_2}{gzX}-\frac{a_1u_1}{gzY}  \Big) E_h\Big( \frac{a_2u_2}{gzX} \Big)E_h\Big( \frac{a_1u_1}{gzY} \Big)\d u_1\d u_2
\end{split}
\end{equation}

Nous sommons ensuite par rapport \`a $a_1$ et $a_2$ en appliquant le Lemme~\ref{lemme soma1a2} et en faisant une sommation d'Abel double.  Dor\'enavant, nous nous pla\c cons dans le cas $q_1> hT$.  
Le terme principal obtenu,  apr\`es sommation en $(a_1,a_2)$ tel que 
$Q_1/hT< \varepsilon_1a_1\leq Q_1T/h$, $ 
 Q_2/hT< \varepsilon_2a_2\leq Q_2T/h$ et $(a_1a_2,gz)=1$
est 
\begin{equation}\label{tp1}
    \begin{split}
      = &
\frac {1}{\zeta(2) } \frac{\phi_1(gz)^3{ \phi_2(g)}}{ \varphi_{-1}(gz){ \phi_1(g)} \phi(z)XY}
 \sum_{(\varepsilon_1,\varepsilon_2)\in \{\pm 1\}^2} \int_{ Q_1/hT }^{Q_1T/h} 
\int_{ Q_2/hT }^{Q_2T/h} 
%\int_{1/2}^{1}\int_{1/2}^{1}
\cr&\qquad 
\int_{1}^{2}\!\!\!\int_{1}^{2}
E_h\Big( -\frac{\alpha_2}{gzu_1X}-\frac{\alpha_1}{gzu_2Y}  \Big) E_h\Big( \frac{\alpha_2 }{gzu_1X} \Big)E_h\Big( \frac{\alpha_1}{gzu_2Y} \Big)\frac{\d u_1\d u_2}{u_1^2u_2^2}\d \alpha_2\d \alpha_1.
%E_h\Big( -\frac{\varepsilon_2\alpha_2u_2}{Q_2}-\frac{\varepsilon_1\alpha_1u_1}{Q_1}  \Big) E_h\Big( \frac{\varepsilon_2\alpha_2u_2}{Q_2} \Big)E_h\Big( \frac{\varepsilon_1\alpha_1u_1}{Q_1} \Big)\d u_1\d u_2\d \alpha_2\d \alpha_1.
\end{split}
\end{equation}
%avec $Q_1=gzY  $ et $Q_2=gzX.$

Le terme d'erreur obtenu est scind\'e en quatre parties. Celle correspondant au dernier terme de \eqref{eqsomdabel} est,   apr\`es sommation sur $A_1$ et $A_2$, major\'ee par
\begin{align*}
   & \ll  \frac{2^{\omega(gz)}}{ \phi(z)XY}\frac{h^2}{ Q_1Q_2} 
\int_{-Q_1/2}^{Q_1/2}\int_{-Q_2/2}^{Q_2/2}\int_{1}^{2}\!\!\!\int_{1}^{2}(\textcolor{red}{|}\alpha_1\textcolor{red}{|}+\textcolor{red}{|}\alpha_2\textcolor{red}{|})(\log h)
\cr&\qquad\qquad\qquad \quad E_h^+\Big( -\frac{ \alpha_2 }{gzu_1X}-\frac{ \alpha_1}{gzu_2Y}  \Big) E_h^+\Big( \frac{ \alpha_2}{gzu_1X} \Big)E_h^+\Big( \frac{ \alpha_1 }{gzu_2Y} \Big)\d u_1\d u_2\d \alpha_2\d \alpha_1
\cr  &\ll   \frac{2^{\omega(gz)}}{  \phi(z)XY} {h^2(Q_1+Q_2)}(\log h)^3
\ll  \frac{2^{\omega(gz)}}{  \phi(z)XY} {h Q_1 Q_2 }\frac{(\log h)^3}{T}
\ll \frac{2^{\omega(gz)}(gz)^2h}{  \phi(z) }  \frac{(\log h)^3}{T}
\end{align*}
o\`u nous avons fait le changement de variables $t_1=\alpha_1/gzu_2Y$
et $t_2=\alpha_2 /gzu_1X$ et utilis\'e la quatri\`eme majoration du Lemme \ref{calculmajI}.

La partie du terme d'erreur correspondant au deuxi\`eme terme de \eqref{eqsomdabel} pour $A_1$ et $(\varepsilon_1,\varepsilon_2)$ fix\'es est 
\begin{align*}
    \ll  \frac{2^{\omega(gz)}}{  \phi(z)XY}\frac{h }{  Q_1 }&
\int_{1}^{2}\!\!\!\int_{1}^{2}\int_{A_1}^{2A_1} ( \alpha_1 +A_2)(\log h)
\cr&E_h^+\Big( -\frac{\varepsilon_2A_2 }{gzu_1X}-\frac{\varepsilon_1\alpha_1 }{gzu_2Y}  \Big) E_h^+\Big( \frac{\varepsilon_2A_2 }{gzu_1X} \Big)E_h^+\Big( \frac{\varepsilon_1\alpha_1 }{gzu_2Y} \Big)\d \alpha_1\d u_1\d u_2 
\cr  \ll  \frac{2^{\omega(gz)}}{  \phi(z)XY}{h }&E_h^+\Big( \frac{ A_2 }{gzX} \Big)(\log h)
\int_{1/2}^{1}\int_{1/2}^{1}\int_{A_1/Q_1}^{2A_1/Q_1} (Q_1 t_1 +A_2)
\cr&E_h^+\Big( -\frac{\varepsilon_2A_2}{gzu_1X}- {\varepsilon_1t_1u_1}  \Big) E_h^+ ( { t_1u_1}) \d t_1\d u_1\d u_2.
\end{align*} 
  Apr\`es sommation sur $A_1$ et $\varepsilon_1$, nous obtenons un majorant en 
\begin{align*}
    \ll   \frac{2^{\omega(gz)}}{  \phi(z)XY}{h }&E_h^+\Big( \frac{ A_2 }{Q_2} \Big)(\log h)
 \int_{1}^{2}\int_{-1/2}^{1/2} (Q_1 {|}t_1 {|}+A_2)
 E_h^+\Big( \frac{\varepsilon_2A_2 }{gzu_1X}+ { t_1 }  \Big) E_h^+ (t_1 ) \d t_1\d u_1 .
\end{align*} 
o\`u pour supprimer l'int\'egrale en $u_1$ nous avons fait un changement  homoth\'etique de variables.
En distinguant les cas $|t_1|\leq 1/h$ et $|t_1|>1/h$, nous obtenons 
\begin{align*} hE_h^+\Big( \frac{ A_2 }{Q_2} \Big)&(\log h)\int_{1}^{2} 
\int_{-1/2}^{1/2}  (Q_1 {|}t_1{|}+A_2)
 E_h^+\Big( \frac{ A_2}{gzu_1X} + { t_1 }  \Big) E_h^+ (t_1 ) \d t_1\d u_1
 \cr&\ll hE_h^+\Big( \frac{ A_2 }{Q_2} \Big)(\log h)\Big\{ \frac{Q_1}{h }E_h^+\Big(  \frac{ A_2 }{Q_2}\Big)+Q_1\log h+A_2(\log h) E_h^+\Big(  \frac{ A_2 }{Q_2}\Big)\Big\}.
\end{align*}
Apr\`es sommation sur $A_2$, nous obtenons un terme d'erreur en 
$$ \ll  \frac{2^{\omega(gz)}}{  \phi(z)XY}h^2(\log h)^2(Q_1+Q_2)
\ll  \frac{2^{\omega(gz)}(gz)^2}{  \phi(z) }h \frac{(\log h)^2}{T}.
$$
Nous ne d\'etaillons pas la majoration de la contribution des autres termes.

Restreignons-nous au cas $X>Y$. Gr\^ace \`a une  interversion des sommations et \`a un changement de variables, le terme principal \eqref{tp1} devient
\begin{align*} =&
 \frac{\phi_1(gz)^2 \phi_2(g)g^2z }{  \varphi_{-1}(gz)  \zeta(2)} \cr& \sum_{(\varepsilon_1,\varepsilon_2)\in \{\pm 1\}^2}
\int_{1}^{2}\!\!\!\int_{1}^{2}\int_{ Z/u_1zhT }^{ ZT/u_1zh} \int_{ Z/u_2zhT }^{ZT/u_2zh} \! 
E_h( - \varepsilon_2\alpha_2 -\varepsilon_1\alpha_1   ) E_h( \varepsilon_2\alpha_2  )E_h ( %\textcolor{red}
{\varepsilon_1}\alpha_1)\d \alpha_2\d \alpha_1\frac{\d u_1\d u_2}{u_1u_2} .\end{align*}
Lorsque $q_1>hT$ et $q_2>hT$, nous \'etendons alors les int\'egrales en $\alpha_j$ \`a l'intervalle $[0,1/2]$ gr\^ace au Lemme \ref{calculmajI} et utilisons la formule issue de la m\'ethode du cercle
\begin{align*}
  \sum_{(\varepsilon_1,\varepsilon_2)\in \{\pm 1\}^2} & \int_{0}^{1/2}\int_{0}^{1/2} 
E_h( - \varepsilon_2\alpha_2 -\varepsilon_1\alpha_1   ) E_h( \varepsilon_2\alpha_2  )E_h ( \varepsilon_1\alpha_1)\d \alpha_2\d \alpha_1
\cr&=\int_{-1/2}^{1/2}\int_{-1/2}^{1/2} 
E_h( - \alpha_2 -\alpha_1   ) E_h( \alpha_2  )E_h ( \alpha_1)\d \alpha_2\d \alpha_1
=h,\end{align*}
o\`u $E_h$ est d\'efini en \eqref{defEhalpha}.
Le terme principal \eqref{tp1} lorsque $X>Y$ est donc 
\begin{align*}&=\frac {h}{\zeta(2) } \frac{\phi_1(gz)^2 \phi_2(g)g^2z }{  \varphi_{-1}(gz)    }
 \Big\{\int_{X}^{2X}\int_{Y}^{2Y}\frac{\d u_1\d u_2}{u_1u_2}+O\Big( \frac{\log h}{  T}\Big)\Big\}.\end{align*}
%{\color{blue}En fait, ici le terme d'erreur en $O(1/T)$ est $\ll {1/(1+Q_1/h)}+{1/(1+Q_2/h)}$ dans le cas g\'en\'eral. }
Lorsque $X=Y$, il devient
\begin{align*}&=\frac {h}{\zeta(2) } \frac{\phi_1(gz)^2 \phi_2(g)g^2z }{  \varphi_{-1}(gz)    }
 \Big\{\int_{X}^{2X}\int_{X}^{2X}1_{u_1\geq u_2}\frac{\d u_1\d u_2}{u_1u_2}+O\Big( \frac{\log h}{  T}\Big)\Big\}.\end{align*}

Nous sommons en $X$ et $Y$ tels que $X\geq Y$ nous obtenons un terme
\begin{align*}&=\frac {h}{\zeta(2) } \frac{\phi_1(gz)^2 \phi_2(g)g^2z }{  \varphi_{-1}(gz)    }
 \Big\{{\mathcal V}(g,z)+O\Big( \frac{(\log h)^3}{  T}\Big)\Big\}.\end{align*}
avec
$${\mathcal V}(g,z):=\int_{\substack{u_1\geq zT^4, u_2\geq z\\ gzu_2>hT\\ gu_1\leq hT}}1_{u_1\geq u_2}\frac{\d u_1\d u_2}{u_1u_2}.$$
Lorsque $gz\leq h/T^3$, on a
$${\mathcal V}(g,z)=\tfrac 12 \big( \log (\min\{ z,h/gzT^3\}\big)^2+\log \big(\min\{ T^4, T^3gz^2/h\}\big)   \log (h/gzT^3)1_{gz^2\geq h/T^3}.$$ Lorsque $gz> h/T^3$, cette int\'egrale est nulle.

Il nous reste \`a estimer
$${\mathcal V}:=\frac {1}{\zeta(2) }\sum_{\substack{g\leq T  }} \frac{\mu (g )  \phi_2(g)  }{  g \phi_1(g)     }
\sum_{z\leq h/gT^3}\frac{ \mu^2(gz)   }{   z \varphi_{-1}(gz)    }{\mathcal V}(g,z)$$
Nous consid\'erons ${\mathcal V}_1$ (respectivement ${\mathcal V}_2$) la contribution du premier ${\mathcal V}_1(g,z)$ (respectivement du second ${\mathcal V}_2(g,z)$) terme de ${\mathcal V}(g,z).$

La contribution des $(g,z)$ tels que $h/T^3\leq gz^2\leq hT^2/g$ dans ${\mathcal V}_2$ est clairement major\'ee par $\ll (\log T)^3(\log h)$.  
En n\'egligeant la contribution des $g>T$ et en utilisant~\eqref{eq formula2}, nous obtenons donc
$${\mathcal V}_2 =\frac {4}{\zeta(2) }\!\!\sum_{\sqrt{hT^2}<m\leq h/T^3}\!\! \!\!\!\!\!\!\frac{ \mu(m)^2(\log T)\log (h/mT^3)}{ m^2\phi_1(m)\phi_{-1}(m)}+O\Big((\log T)^3(\log h)+ \frac{\log h}{{\mathcal L}(T)}\Big)\ll (\log T)^3(\log h).$$

Nous scindons l'estimation de ${\mathcal V}_1$ en fonction de la taille de $z$.
Lorsque $z\leq \sqrt{h}/T^2,$  nous avons ${\mathcal V}_1(g,z)=(\log z)^2/2$ qui ne d\'epend pas de la valeur de $g$
et la condition de taille sur $g$ est l'in\'egalit\'e $g\leq T.$
Gr\^ace \`a \eqref{somgmu}, nous obtenons
que cette contribution   est 
$$\ll (\log h)^2/{\mathcal L}(T).$$

Lorsque $\sqrt{h}/T^2<z\leq \sqrt{h/gT^3} $ et $g\leq T$, nous pouvons approcher le terme \`a sommer  $${\mathcal V}_1(g,z)=\tfrac 12(\log z)^2=\tfrac 18(\log h)^2+O\big((\log T)\log h\big).$$ Le terme d'erreur a une contribution n\'egligeable major\'ee par $O((\log T)^3\log h).$
En appliquant le Lemme \ref{lemme sum xi} et en utilisant la relation
$$\sum_{z=1}^\infty \frac{ \mu^2(gz)   }{   z^s \varphi_{-1}(gz)    }=\zeta(s)\prod_{p}\Big(1+\frac{1}{p^s(p+1)} \Big)
\prod_{p\mid g }\Big(1+\frac{1}{p  }+\frac{1}{p^s } \Big)^{-1},$$ nous obtenons
$$\frac {1}{\zeta(2) }\sum_{\sqrt{h}/T^2<z\leq \sqrt{h/gT^3} } \frac{ \mu^2(gz)   }{   z \varphi_{-1}(gz)    }
=-\frac{C\log (g/T)}{2\phi_{-2}(g)}+O\Big(\frac{\psi_{1/2}(g)}{h^{1/3}}\Big),$$
avec
$$C:=\prod_{p} \Big\{\Big(1+\frac 2p\Big)\Big(1-\frac 1p\Big)^2\Big\}. $$ 

De la formule
$$\sum_{\substack{g=1  }}^\infty \frac{-C  \phi_2(g)\mu (g )\log g  }{  g \phi_1(g) \phi_{-2}(g)    }=1, $$
nous d\'eduisons que la contribution \`a ${\mathcal V}_1$ des $\sqrt{h}/T^2<z\leq \sqrt{h/gT^3} $  est
$$=\frac 1{16}(\log h)^2+O\Big((\log T)^3\log h+(\log h)^2/{\mathcal L}(T)\Big).$$

La contribution des $z$ tels que   $\sqrt{h/gT^3} <z\leq \sqrt{h/ T^2}/g $ est trait\'ee de la m\^eme mani\`ere. Nous pouvons approcher le terme \`a sommer  $${\mathcal V}_1(g,z)=\tfrac 12(\log h/gzT)^2=\tfrac 1{8}(\log h)^2+O\big((\log T)\log h\big).$$ Le terme d'erreur a encore une contribution n\'egligeable major\'ee par $O((\log T)^3\log h).$
En appliquant le Lemme \ref{lemme sum xi}, nous obtenons
$$\frac {1}{\zeta(2) }\sum_{\sqrt{h/gT^3} <z\leq \sqrt{h/ T^2}/g  } \frac{ \mu^2(gz)   }{   z \varphi_{-1}(gz)    }
=-\frac{C\log (g/T)}{2\phi_{-2}(g)}+O\Big(\frac{\psi_{1/2}(g)}{h^{1/3}}\Big).$$
Nous en d\'eduisons que la contribution \`a ${\mathcal V}_1$ des $\sqrt{h/gT^3} <z\leq \sqrt{h/ T^2}/g $ est encore
$$=\frac 1{16}(\log h)^2+O\Big((\log T)^3\log h+(\log h)^2/{\mathcal L}(T)\Big).$$

 Reste le cas
$\sqrt{h/ T^2}/g <z\leq h/gT^3$.
En n\'egligeant la contribution des $g>T$ et en utilisant \eqref{eq formula2}, nous obtenons une contribution
$$  =\frac {1}{2\zeta(2) }\sum_{\sqrt{h/ T^2}  <m\leq h/ T^3} \frac{ \mu(m)^2}{ m^2\phi_1(m)\phi_{-1}(m)} (\log (h/mT^3))^2+O\Big( \frac{(\log h)^3}{{\mathcal L}(T)}\Big)\ll  \frac{(\log h)^3}{{\mathcal L}(T)}.$$

 Nous obtenons une contribution globale
\begin{align*} S_3 (h,T;(1,1))   &=\frac{1}{8} h(\log h)^2+O \Big(h(\log T)^3\log h+ \frac{h(\log h)^{10}}{{\cal L}(T)}+\frac{h (\log h)^{12}}{T }\Big).\end{align*}
   \end{proof} 
   
\goodbreak

\subsection{Quatri\`eme cas :  $g\leq T$,  $gx\leq hT,$ $z\leq y\leq x\leq  zT^4$.}
Majorons $S_4^+$ la   somme des termes pris en module dans le cas %valeur   de la contribution du cas o\`u $A_2\leq Tq_2/h,$ $A_1\leq Tq_1/h$ et 
$g\leq T$,  $gx\leq hT,$ $z\leq y\leq x\leq  zT^4$,
$Q_1/Th <A_1\leq  Q_1T/h ,$ $  Q_2/Th <A_2\leq  Q_2T/h$.

 \begin{lemma}\label{lemme estS4} Soit $\delta\in \{ (1,1),(1,0),(0,1),(0,0)\}$. Lorsque $h\geq 2$ et $T\geq 1$, nous avons   l'estimation
    \begin{equation}\label{estS4} 
S_4^+(h,T;\delta)\ll  h(\log h)(\log_2 h)^2(\log T)^5  .
\end{equation}
\end{lemma}

\begin{proof} Nous avons $g\leq T$,   $  X\leq hT/ g,$ $   Z\leq Y\leq X\leq 2ZT^4 $,
$Q_1/Th <A_1\leq  Q_1T/h ,$ $  Q_2/Th <A_2\leq  Q_2T/h$.
Nous utilisons la majoration 
\begin{align*}
F(x,y,gz,a_1,a_2) %&\ll \frac{h^3}{X^2Y^2(1+{h||a_1/q_1||})(1+{h||a_2/q_2||})}
 &\ll \frac{ h^3}{X^2Y^2(1+{hA_1/Q_1})(1+{hA_2/Q_2})}
%(1+{h||a_1/q_1+a_2/q_2||})
.\end{align*}
%Regardons par exemple le cas o\`u $||a_1/q_1+a_2/q_2||\asymp A_1/Q_1+A_2/Q_2$.{\color{red} Est ce n\'ecessaire ? }
%Puisqu'alors $F(x,y,gz,a_1,a_2)\ll F^+(X,Y,gz,A_1,A_2)$ avec  $$F^+(X,Y,gz,A_1,A_2)\ll \frac{h^3}{X^2Y^2(1+{hA_1/Q_1})(1+{hA_2/Q_2})
%(1+{h(A_1/Q_1+A_2/Q_2)})}. $$ 
%La sommation sur  $y\sim Y$ et $x\sim X$ qui satisfait la congruence $a_1x\equiv - a_2 y  \bmod z $   fournit  d'apr\`es le Lemme~\ref{lemme estS}  une contribution $\ll (\log_2 h)^2XY/z.$
%Gr\^ace au Lemme~\ref{lemme estS},
Puisque $1/\phi_1(n)\ll  \log_2(3n),$
nous obtenons la majoration 
\begin{align*}  
\sigma^+(g,z,a_1,a_2) &\ll    \sum_{\substack{ Y\leq y<2Y\\ (y,a_1gz)=1}}\frac{\mu(y)^2  }{\phi_1(y)^2}
 \sum_{\substack{ X\leq x<2X\\ a_1x\equiv - a_2 y \bmod z\\ (x,a_2gy)=1\\ (a_1x+a_2 y,g)=1}}\frac{\mu(x)^2  }{\phi_1(x)^2}\frac{ h^3}{X^2Y^2(1+{hA_1/Q_1})(1+{hA_2/Q_2})}
 \cr&\ll \frac{(\log_2 h)^2 h^3}{X Y Z(1+{hA_1/Q_1})(1+{hA_2/Q_2})},
\end{align*} 
De plus,   nous avons 
$$\sum_{\substack{A_1,A_2\\Q_1/Th <A_1\leq  Q_1T/h \\ Q_2/Th <A_2\leq  Q_2T/h}}\frac{A_1A_2h^3}{ (1+{hA_1/Q_1})(1+{hA_2/Q_2})}
\ll  hQ_1Q_2(\log T)^2 ,$$
o\`u les $A_j$ sont des puissances de $2$.
Nous obtenons alors   une contribution 
\begin{align*}
S^+_{4}(h,T;\delta)&\ll h(\log_2 h)^2(\log T)^2\sum_{\substack{g\leq T \\     Z<X,Y\leq 2ZT^4\\ z\leq hT }}
\frac{ \mu(gz )^2 z Q_1Q_2 }{\phi(g)^3\phi( z)^2 XYZ^2}
\cr&\ll h(\log_2 h)^2(\log T)^4\sum_{\substack{g\leq T\\      z\leq  hT }}
\frac{ \mu(gz )^2  g^2z}{\phi(g)^3\phi( z)^2  }\ll h(\log h)(\log_2 h)^2(\log T)^5.
\end{align*}
\end{proof}

\subsection{Cinqui\`eme cas : $g>T$.}
Nous montrons que la contribution des $g>T$ est n\'egligeable gr\^ace aux oscillations de la fonction de M\"obius. 

Nous fixons $X,Y,Z$ des puissances de $2$. 
Nous consid\'erons que $(x,y,z) $ appartient \`a un ensemble $E$ inclus dans $D(\delta)$ et $g\in I_G(x,y,z)$ avec $I_G(X,Y,Z)$ un intervalle inclus dans $ ]G,2G]$ avec $T<G\leq hT$. Nous notons $S_5(X,Y,Z)$ la contribution. Nous notons $S_5(h) $
la somme des $| S_5(X,Y,Z)|$ lorsque $(X,Y,Z)$ v\'erifie  $YZ\ll hT$.
\begin{lemma}\label{lemme estS5}  Soit $\delta\in \{ (1,1),(1,0),(0,1),(0,0)\}$.  Lorsque $h\geq  T\geq 2$, nous avons   l'estimation
    \begin{equation}\label{estS5} 
S_5(h,T;\delta)\ll  \frac{h (\log h)^4}{ {\cal L} (T)}  + \frac{h (\log h)^{8} }{  T } .
\end{equation}
\end{lemma}

\begin{proof} 
 Quitte \`a multiplier par un facteur $\log h$, nous pouvons supposer $G\leq g <2G$ avec $T\leq G\leq hT.$ Sans changer les termes d'erreur, gr\^ace au Lemme \ref{lemme Sigma}, nous pouvons remplacer les conditions 
$  gYZ/Th <A_1\leq  gYZT/h , \, gXZ/Th <A_2\leq  gXZT/h   $ 
en
$  GYZ/Th <A_1\leq  2GYZT/h , \, GXZ/Th <A_2\leq  2GXZT/h .  $ Nous faisons une sommation d'Abel par rapport \`a $g$ en utilisant la majoration
\begin{equation}
    \label{majMm}
M_m(G):=\sum_{\substack{g< G\\ (g,m)=1 }}\frac{\mu(g)\phi_2(g) \mu(g)^2g^3}{\phi(g)^3}
\ll  \psi_{1/2}(m)G{\cal L} (G)^{-2}.\end{equation} 
Cette majoration s'obtient de mani\`ere classique par la m\'ethode d'analyse complexe.
Nous avons la relation\begin{align*}
& x^2y^2\frac{\partial F}{\partial g} (x,y,gz,a_1,a_2) \cr&= -  \frac{a_2}{g^2xz}
\Big\{ 
E_h\Big( -\frac{a_2}{gxz}-\frac{a_1}{gyz}  \Big) E_h'\Big( \frac{a_2}{gxz} \Big)E_h\Big( \frac{a_1}{gyz} \Big)-E'_h\Big( -\frac{a_2}{gxz}-\frac{a_1}{gyz}  \Big) E_h\Big( \frac{a_2}{gxz} \Big)E_h\Big( \frac{a_1}{gyz} \Big)\Big\}\cr&\quad -\frac{a_1}{g^2yz}
\Big\{
E_h\Big( -\frac{a_2}{gxz}-\frac{a_1}{gyz}  \Big) E_h\Big( \frac{a_2}{gxz} \Big)E'_h\Big( \frac{a_1}{gyz} \Big)-E'_h\Big( -\frac{a_2}{gxz}-\frac{a_1}{gyz}  \Big) E_h\Big( \frac{a_2}{gxz} \Big)E_h\Big( \frac{a_1}{gyz} \Big)\Big\} 
  \cr&\!\!\!\!=
  -  \frac{a_2}{g^2xz} 
E_h\Big( -\frac{a_2}{gxz}-\frac{a_1}{gyz}  \Big) E_h'\Big( \frac{a_2}{gxz} \Big)E_h\Big( \frac{a_1}{gyz} \Big) 
-\frac{a_1}{g^2yz} 
E_h\Big( -\frac{a_2}{gxz}-\frac{a_1}{gyz}  \Big) E_h\Big( \frac{a_2}{gxz} \Big)E'_h\Big( \frac{a_1}{gyz} \Big)\cr&  +\Big(\frac{a_2}{g^2xz}+\frac{a_1}{g^2yz}\Big)
 E'_h\Big( -\frac{a_2}{gxz}-\frac{a_1}{gyz}  \Big) E_h\Big( \frac{a_2}{gxz} \Big)E_h\Big( \frac{a_1}{gyz} \Big) 
\end{align*} 
qui fournit, \`a l'aide de \eqref{majprodetderivee}, la majoration 
\begin{equation}\label{calculpartialg}\begin{split}
   \frac{\partial}{\partial g} \Big( \frac{ F  (x,y,gz,a_1,a_2)}{g^3} \Big)
   &\ll \frac{h^3  }{G^4(XY)^2}\Big( \frac{1}{(1+{h/g||a_1/yz||})(1+{h/g ||a_1/yz+ a_2/xz||)} }\cr&\qquad+\frac{1}{(1+{h/g||a_2/xz||})(1+{h/g ||a_1/yz+ a_2/xz||)} }\cr&\qquad+\frac{1}{(1+{h/g||a_2/xz||})(1+{h/g ||a_1/yz||)} } \Big).
\end{split}  
\end{equation}
Nous avons $| S_{5}(X,Y,Z)|\leq \sum_{T\leq G\leq hT}  S_{5}(G,X,Y,Z)$
avec 
\begin{align*}
    S_{5}(G,&X,Y,Z) 
 = \sum_{\substack{   z \sim Z}} \frac{  \mu(  z)^2 }{ \phi(z)^2}
    \!\!\!\sum_{\substack{| a_2 |\sim A_2 \\  | a_1 |\sim A_1\\ (a_2a_1, z)=1 }}
     \sum_{\substack{ X\leq x<2X\\ (x,y,z)\in E\\  (x,a_2)=1 }}\frac{\mu(x)^2x^2 }{\phi(x)^2}\cr&\quad
\sum_{\substack{ Y\leq y<2Y\\ a_2y\equiv - a_1x \bmod z\\(y,a_1 zx)=1}}\frac{\mu(y)^2 y^2}{\phi(y)^2}\Bigg|
\sum_{\substack{g\in I_G(x,y,z)\cap [G,2G[ \\ (g,a_2a_1xyz(a_1x+a_2 y) )=1}}\!\!\!\!\!\!\frac{\mu(g)\phi_2(g) \mu(g)^2g^3}{\phi(g)^3 }
\frac{ F  (x,y,gz,a_1,a_2)}{g^3} \Bigg|
.
\end{align*} Nous \'ecrivons $I_G(x,y,z)\cap [G,2G[=[G', G''[$ (Le cas d'intervalle ferm\'e ou ouvert se traite de la m\^eme mani\`ere).
Posant $m=a_2a_1xyz(a_1x+a_2 y)$, la somme int\'erieure vaut 
\begin{align*}
\frac{M_m(G'')  F  (x,y,G''z,a_1,a_2)}{G''^3}&-
\frac{M_m(G')  F  (x,y,G'z,a_1,a_2)}{G'^3}\cr&\quad -\int_{G'}^{G''}M_m(g) \frac{\partial}{\partial g}\Big( \frac{ F  (x,y,gz,a_1,a_2)}{g^3} \Big) \d g.  \end{align*}
%La stricte in\'egalit\'e $g<G$ dans la d\'efinition \eqref{majMm} de $M_m$ permet de tenir compte de la relation $G\leq g<2G$ (et non $G<g\leq 2G$).

D'apr\`es \eqref{majMm} et \eqref{majprodetderivee}, les deux premiers termes sont major\'es par
$$\ll \frac{\psi_{1/2}(m)  h^3{\cal L} (G)^{-2}}{ X^2Y^2G^2(1+{hA_1/GYZ})(1+{hA_2/GXZ})} 
\ll \frac{\psi_{1/2}(m)  hZ^2 {\cal L} (G)^{-2}}{ XY A_1 A_2 }.$$
Apr\`es sommation sur   $y\sim Y$ satisfaisant $a_2y\equiv - a_1x \bmod z$, puis sur  $a_1\sim A_1$, $a_2\sim A_2$, leur contribution est alors facilement major\'ee par
$$\ll      h  {\cal L} (G)^{-2}\sum_{\substack{   z \sim Z}} \frac{\psi_{1/2}(z)  \mu(  z)^2 Z}{ \phi(z)^2} 
\sum_{\substack{ P(x)\leq h^3\\    (x,a_2)=1 }}\frac{\psi_{1/2}(x)\mu(x)^2x  }{\phi(x)^2}
\ll h  (\log h){\cal L} (G)^{-2}.$$
Le nombre de $(A_1,A_2)$ est major\'e par $O((\log T)^2)$ de sorte que la contribution \`a la somme des $|S_{5}(G,X,Y,Z)|$ de ces termes est 
$$\ll h (\log h)^3 {\cal L} (G)^{-2}(\log T)^2. $$
 
Un majorant de l'int\'egrale issu de la majoration de la d\'eriv\'ee par \eqref{calculpartialg} peut \^etre scind\'e en trois parties. La troisi\`eme correspond au troisi\`eme terme du majorant de \eqref{calculpartialg} et fournit un majorant de la forme 
$$=O\Big(\frac{\psi_{1/2}(m)  h^3{\cal L} (G)^{-1}}{X^2Y^2 G^2(1+{hA_1/GYZ})(1+{hA_2/GXZ})}\Big)$$
comme pr\'ec\'edemment. La premi\`ere et la deuxi\`eme n\'ecessitent de faire un changement d'indice de sommations en prenant $a_3$ d\'efini par $a_3z=-a_1x-a_2y$ et de sommer sous la condition suppl\'ementaire $|a_3|\sim A_3.$ Le reste des calculs est semblable au cas pr\'ec\'edent.
%Nous sommons ensuite sur $y\sim Y$ et $x\sim X$ qui satisfont \`a la congruence $a_1x\equiv - a_2 y  \bmod z $, puis sur $a_1\sim A_1 $ et $a_2\sim A_2 $.
Nous obtenons ainsi une contribution \`a la somme des $|S_{5}(G,X,Y,Z)|$
$$   %h (\log T)^3\sum_{z\sim Z} \frac{ \mu(z)^2     z\psi_{1/2}( z) }{ \phi(z)^2Z}{\cal L} (G)^{-2}
\ll  h (\log h)^3(\log T)^3 {\cal L} (T)^{-2} $$ o\`u les puissances suppl\'ementaires de  $\log h$ correspondent \`a la fixation de $A_1,A_2 $ mais aussi  \'eventuel\-lement de $A_3.$
Il reste \`a multiplier par $O(\log h)$ le majorant du nombre de valeurs possibles de valeurs de $G$ pour obtenir la majoration du lemme.
\end{proof}

\subsection{Sixi\`eme cas :  $h/T<gyz\leq hT$, $gx\leq hT$.}

Le sixi\`eme cas correspond au cas $h/T<gyz\leq hT$, $gx\leq hT$, $(x,y,z)\in D(\delta)$.

\begin{lemma}\label{lemme estS6} Soit $\delta\in \{ (1,1),(1,0),(0,1),(0,0)\}$. Lorsque $h\geq T\geq 2$, nous avons   l'estimation
    \begin{equation}\label{estS6} 
S_6(h;\delta)\ll \frac{h (\log h)^4}{ {\cal L} (T) }+ h (\log h)(\log_2 h)^2(\log T)^5 +\frac{h (\log h)^{8} }{  T } .
\end{equation}
\end{lemma}

\begin{proof}
Nous   modifions le domaine de sommation.
Tout d'abord, gr\^ace au cinqui\`eme cas, nous nous restreignons au cas o\`u $g\leq T$, puis gr\^ace au quatri\`eme cas nous excluons le cas  $x\leq zT^4 $. Nous pouvons donc nous placer dans le cas 
$$h/T<gyz\leq hT , \quad gx\leq hT,\quad x>zT^4,\quad (x,y,z)\in D(\delta).$$
Enfin, d'apr\`es le Lemme \ref{lemme Sigma}, nous pouvons nous placer dans le cas o\`u
$$Q_1/hT<|a_1|\leq TQ_1/h, \quad Q_2/hT<|a_2|\leq TQ_2/h.$$

 La somme $\sigma 
(g,z,a_1,a_2)$  en $x$ et $y$ est approch\'ee de mani\`ere satisfaisante par la quantit\'e $\sigma_M
(g,z,a_1,a_2)$ d\'efinie en \eqref{defsigmaM}.
Nous avons aussi $gxz>h/T.$ Il est alors possible de la m\^eme mani\`ere qu'au quatri\`eme cas de majorer la contribution du sous-cas 
$h/T<gxz\leq 2hT.$ Nous pouvons donc nous placer dans le cas $gxz>2hT$ et $Q_2=gXZ\gg hT.$ Nous avons ainsi $x>2y$ et donc $X>Y$. La relation $X>Y$ implique bien $x>y$.

Nous consid\'erons donc $X,Y,Z$ tels que 
$$ X/2\geq Y\geq Z,\quad h/4T<gYZ\leq ht, \quad gX\leq hT,\quad 2X\geq ZT^4, 2gXZ>hT. $$
\`A partir du Lemme \ref{lemme soma2}, une sommation d'Abel permet d'estimer \begin{equation}
    \label{suma2sigmaM}\begin{split}
&
 \sum_{\substack{Q_2/hT<|a_2|\leq Q_2T/h\\ (a_2,gz)=1}}\sigma_M
(g,z,a_1,a_2)
\cr&= \frac {\zeta(2)}{\zeta(4){\phi(z)} } \sum_{\substack{Q_2/hT<|a_2|\leq Q_2T/h\\ (a_2,gz)=1}}  w^*(g,z,a_1,a_2)\int_{X}^{2X}\int_{Y}^{2Y}F(x,y,gz,a_1,a_2)\d y\d x
. \end{split}\end{equation} 
Notons que l'intervalle d'int\'egration $[X,2X[$ en $x$ est en fait l'intervalle \begin{equation}
    \label{intervalle de x} [\max\{ X,2Y,zT^4, 2hT/gz\},\min\{ 2X, hT/g\}[\end{equation}  tandis que l'intervalle d'int\'egration $[Y,2Y[$ en $y$ est en fait $ [\max\{ Y,h/gzT\},\min\{ 2Y, hT/gz \}[$.

Le terme principal obtenu correspond au terme principal du Lemme \ref{lemme soma2} est 
\begin{align*}
    &  = \frac{  \phi_2(g)\phi_1(gz)}{ \zeta(2)\varphi_{-1}(gz)z } \Big(\int_{Q_2/hT}^{Q_2T/h}+\int_{-Q_2T/h}^{-Q_2/hT}\Big)\cr&\qquad\qquad \int_{Y}^{2Y}\!\!\!\int_{X}^{2X}
E_h\Big( -\frac{\alpha_2 }{gzv_1 }-\frac{a_1 }{gzv_2}  \Big) E_h\Big( \frac{\alpha_2}{gzv_1} \Big)E_h\Big( \frac{a_1}{gzv_2} \Big)\frac{\d v_1\d v_2}{v_1^2v_2^2}\d \alpha_2.
\end{align*}

Afin de compl\'eter l'int\'egrale en $\alpha_2$, nous utilisons le r\'esultat suivant. Posant
$$R(\alpha_1,T):=E_h^+ ( \alpha_1)^2/T+ E_h^+ ( \alpha_1)   E_h^+ ( T/h+\alpha_1) ,$$
lorsque $1/hT<|\alpha_1|\leq T/h$, nous expliquons comment montrer 
\begin{equation}
    \label{maj completion integrale}
\Big(\int_0^{{1/hT}}+\int_{T/h}^{1/2}\Big)|E_h (  {\alpha_1 } )E_h ( {\alpha_2 } )E_h ( -\alpha_1-{\alpha_2 } )|\d\alpha_2 \ll R(\alpha_1,T).\end{equation}

Lorsque $0\leq \alpha_2 \leq 1/2hT$, nous avons alors
$|\alpha_1+\alpha_2 |\geq |\alpha_1|-|\alpha_2 |\geq \tfrac 12|\alpha_1|$.
Nous
utilisons alors $E_h ( -\alpha_1-{\alpha_2 } )\ll |E_h^+ ( \alpha_1)|$
et $E_h ({\alpha_2 } )\ll h$ ce qui fournit 
$$\int_0^{{1/2hT}} |E_h (  {\alpha_1 } )E_h ( {\alpha_2 } )E_h ( -\alpha_1-{\alpha_2 } )|\d\alpha_2 \ll E_h^+(\alpha_1)^2 /T.$$
Lorsque $1/2hT< \alpha_2 \leq 1/ hT$ et $\alpha_2/2<-\alpha_1<2\alpha_2$, nous utilisons $E_h ( -\alpha_1-{\alpha_2 } )\ll h $ et $E_h ({\alpha_2 } )\ll E_h^+ ( -\alpha_1)$ ce qui suffit encore
$$\int_{1/2hT}^{{1/hT}} |E_h (  {\alpha_1 } )E_h ( {\alpha_2 } )E_h ( -\alpha_1-{\alpha_2 } )|\d\alpha_2 \ll |E_h^+ (\alpha_1)|^2 /T.$$
La contribution des $\alpha_2\in\, ]\max\{{T/h}, -2\alpha_1\},{1/2}]$ est facilement major\'ee puisqu'alors nous avons $ E_h (  {\alpha_1 } ) E_h ( -\alpha_1-{\alpha_2 } )\ll  E_h^+ ( \alpha_1)^2$ et $E_h (  {\alpha_2 } )\ll h/T$.
Reste le cas $ \alpha_2\in\, ] {T/h}, 2T/h]$ et $-\alpha_1\in \,] {T/2h},  T/h]$. Nous avons $$|E_h (  {\alpha_1 } )E_h ( {\alpha_2 } )|\ll  E_h^+ (  {\alpha_1 } ) h/T .$$
et $|E_h ( -\alpha_1- \alpha_2   )|\ll E_h^+ (T/h+\alpha_1).$ Nous avons donc bien \eqref{maj completion integrale}.
\bigskip

En utilisant \eqref{maj completion integrale} pour compl\'eter l'int\'egrale, nous obtenons une estimation de la somme~\eqref{suma2sigmaM} avec un terme principal \'egal \`a 
\begin{equation}
    \begin{split}
        \label{somzEh2}
    & = 
 \frac{  g\phi_1(gz)\phi_2(g)}{\zeta(2) \varphi_{-1}(gz)     }
   \int_{-1/2}^{1/2}\int_{X}^{2X}\!\!\!\int_{Y}^{2Y}
E_h\Big( - {\alpha_2 } -\frac{a_1 }{gzv_2}  \Big) E_h ( {\alpha_2 } )E_h\Big( \frac{a_1 }{gzv_2} \Big)\frac{\d v_2}{v_2^2}\frac{\d v_1}{v_1}\d \alpha_2%\cr&\qquad\qquad\qquad\qquad\qquad+O\Big(\frac{E_h^+(A_1/Q_1)^2}{T}\Big)\Big\}
\cr&= 
\frac{g\phi_1(gz)  \phi_2(g)}{\zeta(2) \varphi_{-1}(gz)    } \int_{Y}^{2Y} \Big|E_h\Big( \frac{a_1 }{gzv_2} \Big)\Big|^2\frac{\d v_2}{v_2^2}\int_{X}^{2X}\frac{\d v_1}{v_1}.
%\Big\{1+O\Big(\frac{1}{T}\Big)\Big\}. 
\end{split}
\end{equation}
La contribution du terme d'erreur est clairement $\ll h(\log h)^4/T^{1/2}.$ Nous n'indiquons pas les d\'etails.
Lorsque nous sommons en $X$ l'int\'egrale $\int_{X}^{2X}\frac{\d v_1}{v_1}$ d\'ependant de $X$  nous obtenons un facteur 
\begin{align*}
\int_{\max\{ 2Y, zT^4, 2hT/gz\}}^{hT/g}\frac{\d v_1}{v_1}&=\log \big( \min\{h/Tgz,   z /2,hT/2gY \}\big)
\cr&=\log \big( \min\{h/Y,Y \}\big)+O\big(\log T\big)
\cr&=\log (h/Y)+O\big(\log T\big). 
 \end{align*}
since $Y^2\geq h/4gT.$ Au prix d'un terme d'erreur acceptable, nous rempla\c cons donc cette int\'egrale par $\log (h/Y).$

Estimons  maintenant la somme  de l'approximation du terme   \eqref{somzEh2} lorsque $z\sim Z $ et $(z,a_1g)=1$ avec un coefficient $\mu(g)\mu^2(gz)/\phi(g)^3\phi(z)^2$    autrement dit
\begin{equation}
    \label{somzint}
\frac{(\log h/Y)   \mu(g) \phi_2(g)}{\zeta(2) \varphi_{-1}(g ) \phi(g)^2}
\sum_{\substack{z\sim Z\\ (z,a_1g)=1}} \frac{ \mu^2( z)  }{ \varphi_{-1}( z) z^2\phi_1(z)} \int_{Y}^{2Y} \Big|E_h\Big( \frac{a_1 }{gzv_2 } \Big)\Big|^2\frac{\d v_2}{v_2^2}.
\end{equation}

Nous pouvons supposer $Z>T.$
Lorsque $\Re e (s)>1$, nous avons 
$$\sum_{\substack{z=1\\ (z,a_1g)=1}}^\infty\frac{\mu^2(z)}{\phi_{-1}(z)\phi_1(z)z^s}=\zeta(s)\prod_{p\mid a_1g}\Big(1-\frac{1}{p^s}\Big)
\prod_{p\nmid a_1g}\Big(1+\frac{1}{p^{2+s}(1-1/p^2)}
-\frac{1}{p^{2s}(1-1/p^2)}\Big)  .
$$ Soit $\tilde\phi$ la fonction arithm\'etique fortement multiplicative d\'efinie par 
$$ \tilde\phi(p)=1+1/p-1/p^2.$$
Gr\^ace au Lemme \ref{lemme sum xi}, nous obtenons
\begin{align*}\sum_{\substack{1\leq z\leq Z\\ (z,a_1g)=1}}\frac{\mu^2(z)}{\phi_{-1}(z)\phi_1(z)}&=\phi_1(a_1g) \prod_{p\nmid a_1g}\Big(1-\frac{1}{p^2(1+1/p)}\Big)Z+O\big( \psi_{1/2}(a_1g)^2Z^{1/2}\big)\cr&
=C_2\phi_1(a_1g)\frac{\phi_{-1}(a_1g)}{\tilde\phi(a_1g)}  Z+O\big( \psi_{1/2}(a_1g)^2Z^{1/2}\big),\end{align*}
avec
$$C_2:=\prod_{p }\Big(1-\frac{1}{p^2(1+1/p)}\Big).   $$
Apr\`es sommation d'Abel, nous obtenons  pour l'estimation de la  somme apparaissant dans~\eqref{somzint}
un terme principal \'egal \`a 
\begin{equation}\label{somme a1}
    \begin{split} 
    &
=C_2\phi_1(a_1g) \frac{\phi_{-1}(a_1g)}{\tilde\phi(a_1g) }  \int_{Z}^{2Z} \int_{Y}^{2Y} \Big|E_h\Big( \frac{a_1 }{gv_2t} \Big)\Big|^2\frac{\d t\d v_2}{t^2v_2^2}.
%\cr&=C_2\phi_1(a_1g) \frac{\phi_{-1}(a_1g)}{\tilde\phi(a_1g)YZ}   \int_{1/2}^{1} \int_{1/2}^{1} \Big|E_h\Big(  \frac{a_1u_1u_3}{gZY} \Big)\Big|^2\d u_1\d u_3
\cr&=C_2\phi_1(a_1g) \frac{g\phi_{-1}(a_1g)}{\tilde\phi(a_1g) }   \int_{1/2gZY}^{1/gZY} \int_{1/2}^{1}  |E_h (   {a_1u_1u_3} )|^2\d u_1\d u_3.
 \end{split}
\end{equation}

Nous rappelons que la somme sur les $a_1$ porte sur les $a_1$ tels que  $Q_1/hT<|a_1|\leq TQ_1/h$ ce qui implique $1\leq |a_1|\leq T^2.$ Cette somme peut \^etre vue une somme courte (lorsque $Q_1=h/T$, nous avons $a_1=1$) ce qui emp\^eche d'estimer asymptotiquement cette somme.
De plus, nous pouvons \'elargir de le domaine de sommation  des $a_1$
au cas  $|a_1|< Q_1/hT$.

Nous devons donc estimer
$$C_2\sum_{g\leq T}\frac{\mu(g)  \phi_{2}(g)}{\phi(g) \tilde\phi( g)}  \sum_{( Y,Z)}\log (h/Y)
 \int_{1/2gZY}^{1/gZY} \int_{1/2}^{1} \sum_{\substack{ 1\leq |a_1|\leq gYZT/h\\ (a_1,g)=1}}\!\!\!\!\!\phi_1(a_1 ) \frac{ \phi_{-1}(a_1 )}{\tilde\phi(a_1 ) }   |E_h (   {a_1u_1u_3} )|^2\d u_1\d u_3,$$ 
o\`u $( Y,Z)$ v\'erifie 
\begin{equation}
    \label{ineg1XYZ}
h/4T<gYZ\leq hT , \quad    Z\geq Y.\end{equation}
Montrons que la double int\'egrale ci-dessus est 
$O(h).$ En effet, la somme int\'erieure est $\ll h/u_1u_3. $ Lorsque $u_3\leq 1/a_1u_1h$, la majoration triviale $|E_h (   {a_1u_1u_3} )|^2\leq h^2$ fournit le r\'esultat. Lorsque $u_3> 1/a_1u_1h$, la majoration   $|E_h (   {a_1u_1u_3} )|^2\leq 1/(a_1u_1u^3)^3$ fournit aussi cette majoration apr\`es sommation sur $a_1>1/u_1u_3h$. 
 
Nous pouvons donc changer les in\'egalit\'es \eqref{ineg1XYZ} en 
\begin{equation}
    \label{ineg2XYZ}
h/4T<gYZ\leq hT ,    \quad  h^{1/2}<Y \leq h,\end{equation}
quitte \`a n\'egliger un terme en $O(h(\log h)(\log T)^2).$
Ensuite nous observons que nous pouvons sommer $a_1$ jusqu'\`a $T^2$ au lieu de $gYZT/h$ 
quitte \`a n\'egliger un terme en $O(h (\log h)^2(\log T)/T).$ En effet, la somme sur les  $  gYZT/h<a_1\leq T^2$ v\'erifie $|E_h (   {a_1u_1u_3} )|\ll h/T.$

Enfin, nous sommons par rapport \`a la variable $Z$, puis par rapport \`a $Y$, pour obtenir 
un terme principal 
\begin{align*}
    &=C_2\sum_{g\leq T}\frac{\mu(g)  \phi_{2}(g)}{\phi(g) \tilde\phi( g)}  \sum_{\substack{ h^{1/2}<Y \leq h}}\!\!\!\!\!\log (h/Y)
 \int_{1/2hT}^{T/h} \int_{1/2}^{1} \sum_{\substack{ 1\leq |a_1|\leq T^2\\ (a_1,g)=1}}\phi_1(a_1 ) \frac{ \phi_{-1}(a_1 )}{\tilde\phi(a_1 ) }   |E_h (   {a_1u_1u_3} )|^2\d u_1\d u_3\cr
 &=\frac{C_2(\log h)^2}{8\log 2}   \sum_{g\leq T}\frac{\mu(g)  \phi_{2}(g)}{\phi(g) \tilde\phi( g)}   
 \int_{1/2hT}^{T/h} \int_{1/2}^{1} \sum_{\substack{ 1\leq |a_1|\leq T^2\\ (a_1,g)=1}}\phi_1(a_1 ) \frac{ \phi_{-1}(a_1 )}{\tilde\phi(a_1 ) }   |E_h (   {a_1u_1u_3} )|^2\d u_1\d u_3\cr&\qquad+ O\big(h(\log h)(\log T)\big).
 \end{align*}
Les modifications des domaines de comptage et ses sommations que nous venons de faire permettent d'intervertir les sommations en $a_1$ et en $g$ pour utiliser 
$$\sum_{\substack{ 1\leq g\leq T \\ (g,a_1 )=1}}\frac{\mu(g)  \phi_{2}(g)}{\phi(g) \tilde\phi( g)}=-\sum_{\substack{   g>T \\ (g,a_1)=1}}\frac{\mu(g)  \phi_{2}(g)}{\phi(g) \tilde\phi( g)} \ll \psi_{1/2}(a_1)^2{\cal L}(T)^{-2},$$
ce qui fournit une contribution 
$$\ll h (\log h)^2 {\cal L}(T)^{-1}.$$

En rassemblant ces r\'esultats, nous obtenons bien la majoration \eqref{estS6}.
\end{proof}

\subsection{Septi\`eme cas :  $h/T<gyz \leq hT$
et $gx>hT.$}

Soit $S_7$ la contribution du cas   $h/T<gyz \leq hT$
et $gx>hT.$ 
Nous montrons dans cette sous-section le r\'esultat suivant.   

\begin{lemma}\label{lemme estS7} Soit $\delta\in \{ (1,1),(1,0),(0,1),(0,0)\}$. Lorsque $T\geq 1$, nous avons   l'estimation
    \begin{equation}\label{estS7} 
S_7(h,T;\delta)\ll \frac{h (\log h)^4}{ {\cal L} (T)}+ h (\log h) (\log T)^2+\frac{h(\log h)^8}{T} .
\end{equation}
\end{lemma}

\begin{proof}
Nous nous proc\'edons comme dans le sixi\`eme cas sauf qu'ici la sommation longue correspond \`a la variable $a_2$ puisque $Q_2/Z>hT$. Nous nous pla\c cons dans le cas $\delta=(1,1)$, les autres \'etant semblables. Nous n'indiquons que les \'etapes sans d\'etailler tous les calculs. 

Quitte \`a n\'egliger un terme d'erreur de taille acceptable,  
nous pouvons nous placer dans le cas o\`u
$$g\leq T, \quad h/T<gYZ\leq hT , \quad gX> hT,\quad    Y\geq Z,\qquad Q_1/hT<A_1\leq TQ_1/h, \quad $$ 
l'encadrement $Q_2/hT<A_2\leq TQ_2/h$ \'etant inutile et l'in\'egalit\'e $X\geq Y$ \'etant une cons\'equence de celles ci-dessus.

 %Lorsque $q_1$ est petit, nous adaptons l'argument pr\'ec\'edent.
Lorsque $(a_1,q_1)=1$, nous avons    \begin{align*} 
N( A_2)&:=\!\!\!\!\!\! \!\!\!\sum_{\substack{ a_2\leq A_2 \\  a_2\equiv -a_1xy^{-1}\bmod z\\ (a_2 ,gx)=(a_1x+a_2y,g )=1}} \!\!\!\!\!\! 1 = \sum_{\substack{k_2\mid gx\\ k_1\mid  g \\  (k_1,k_2)=1}}\mu(k_1)\mu (k_2)\!\!\!\!\!\!\!\!\!\sum_{\substack{ a_2\leq A_2/ k_2\\  k_2a_2\equiv -a_1xy^{-1}\bmod (k_1z)}} \!\!\!\!\!\! 1
\cr&= { \phi_2(g) \phi_1(x) } \frac{  A_2}{z}+O\big(2^{\omega(x)}3^{\omega(g)}\big) 
%= \frac{ \phi_2(g) \phi(x) }{gx }\frac{  A_2-1}{z}+O\big(2^{\omega(x)}3^{\omega(g)}\big) 
.
\end{align*}
Nous rappelons la quantit\'e
$V $  d\'efinie en~\eqref{def Vh}.
Une sommation d'Abel fournit  
\begin{equation}\label{est5}
    \begin{split} V(g,x,y,z;h) &= 
\sum_{\substack{1\leq |a_1|\leq TQ_1/h\\ (a_1,q_1)=1}} \sum_{\substack{ 1\leq |a_2|\leq q_2/2\\  a_2\equiv -a_1xy^{-1}\bmod z\\ (a_2 ,gx)=(a_1x+a_2y,g )=1}}E_h\Big(\frac{a_1}{q_1}\Big)E_h\Big(\frac{a_2}{q_2}\Big)E_h\Big(-\frac{a_1}{q_1}-\frac{a_2}{q_2}\Big)
\cr&=  \sum_{\substack{1\leq |a_1|\leq  TQ_1/h\\ (a_1,q_1)=1}}\Big\{ \frac{  \phi_2(g) \phi_1(x)}{  z}I_7(a_1,q_1,q_2;h)
  +O(R_7(a_1))%\Big( (\log h)^2 2^{\omega(x)}3^{\omega(g)} q_1h^2 \Big)
  \Big\},
\end{split}\end{equation}
avec
\begin{align*}
    I_7(a_1,q_1,q_2;h)&:= \sum_{ \varepsilon_2 \in \{\pm 1\} }\int_{0}^{q_2/2}E_h\Big(\frac{ a_1}{q_1}\Big)E_h\Big(\frac{\varepsilon_2\alpha_2}{q_2}\Big)E_h\Big(-\frac{ a_1}{q_1}-\frac{\varepsilon_2\alpha_2}{q_2}\Big)\d \alpha_2 
\cr&=q_2\int_{-1/2}^{1/2}E_h\Big(\frac{ a_1}{q_1}\Big)E_h( \alpha_2)E_h\Big(-\frac{ a_1}{q_1}-  \alpha_2 \Big)\d \alpha_2=q_2\Big|E_h\Big(\frac{a_1}{q_1}\Big)\Big|^2
\end{align*}
et $R_7(a_1)$ le terme d'erreur v\'erifiant \begin{align*}
  R_7(a_1)  \ll& \frac{2^{\omega(x)}3^{\omega(g)} }{q_2}\Big|E_h\Big(\frac{a_1}{q_1}\Big)\Big|\int_{-q_2/2}^{q_2/2}\Big| 
    E'_h\Big(\frac{\alpha_2}{q_2}\Big)E_h\Big(-\frac{a_1}{q_1}-\frac{\alpha_2}{q_2}\Big)-E_h\Big(\frac{\alpha_2}{q_2}\Big)E'_h\Big(-\frac{a_1}{q_1}-\frac{\alpha_2}{q_2}\Big)\Big| \d \alpha_2
    \cr&+{2^{\omega(x)}3^{\omega(g)} } \Big|E_h\Big(\frac{a_1}{q_1}\Big)\Big|\Big( \Big|E_h\Big(\frac{a_1}{q_1}\Big)\Big|+\Big|E_h\Big(\frac12+\frac{a_1}{q_1}\Big)\Big|\Big) .
\end{align*}
Nous utilisons \eqref{majprodetderivee} et le Lemme~\ref{calculmajI} de sorte  que le premier terme du majorant $R_{71}(a_1)$ v\'erifie
\begin{align*}
  R_{71}(a_1)  \ll&     \frac{{2^{\omega(x)}3^{\omega(g)} } h}{||a_1/q_1||+1/h}   \int_{-1/2}^{1/2} \frac{\d \alpha_2}{(||\alpha_2||+1/h)(||\alpha_2+a_1/q_1||+1/h)}\ll
    \frac{{2^{\omega(x)}3^{\omega(g)} }   h}{(||a_1/q_1||+1/h)^2} .
\end{align*}
 
Il vient
$$\sum_{\substack{1\leq |a_1|\leq q_1/2\\ (a_1,q_1)=1}}
    R_{7}(a_1)\ll   2^{\omega(x)}3^{\omega(g)}h^2{q_1 } (\log q_1) 
    \ll   2^{\omega(x)}3^{\omega(g)} (\log q_1) \frac{hq_1q_2}{zT } . 
    $$ 
    La contribution de ce terme d'erreur est clairement 
    $O\big(h(\log h)^8/T\big)$.

Il s'agit donc d'estimer la somme
$$ 
 \sum_{\substack{P(gxyz)\leq h^3\\ (x,y,z)\in D(\delta) \\
 g\leq T,  h/T<gYZ\leq hT \\ gX> hT}}\frac{\mu(g)\mu(gxyz)^2    \phi_2(g)  }{\phi(g)^2\phi_1(g)\phi( yz)^2\phi(x)}\sum_{\substack{1\leq |a_1|\leq  TQ_1/h\\ (a_1,q_1)=1}} \Big|E_h\Big(\frac{a_1}{q_1}\Big)\Big|^2 .
$$
Une sommation sur $x$ et $X$ fournit un terme 
$$ \Big( \frac{q}{\phi(q)}-\log h\Big)\!\!\!\!\!\!\!\!\!\!\!\!
 \sum_{\substack{g\leq T\\ y\sim Y, z\sim Z\\ h/T<gYZ\leq hT,\, Y\geq  Z }}\!\!\!\!\!\!\!\!\!\!\!\!\frac{\mu(g)\mu(g yz)^2    \phi_2(g)  }{ \phi(g)^2yz\phi( yz) }\sum_{\substack{1\leq |a_1|\leq  TQ_1/h\\ (a_1,gyz)=1}} \Big|E_h\Big(\frac{a_1}{q_1}\Big)\Big|^2
 +O\big( h(\log h)(\log T)^2\big).
$$
Puis nous intervertissons les sommes en $y$ et $z$ et celle de $a_1$ et sommons en $y$ et $z$. Un simple calcul fournit un terme principal
\begin{align*}  \frac{1}{\zeta(2)}&\Big( \frac{q}{\phi(q)}-\log h\Big)\!\!\!\!\!\!\!\!\!\!\!\!
 \sum_{\substack{g\leq T\\   h/T<gYZ\leq hT,\, Y\geq  Z }}\!\!\!\!\!\!\!\!\!
 \frac{\mu(g)     \phi_2(g)   }{g \phi(g)  \phi_{-1}(g)  }\sum_{\substack{1\leq |a_1|\leq  TQ_1/h\\ (a_1,g)=1}}  \frac{ \phi_1(a_1 ) }{  \phi_{-1}(a_1)  }
\int_{Y}^{2Y}  \int_{Z}^{2Z}  \Big|E_h\Big( \frac{a_1}{gt_2t_3} \Big)\Big|^2\frac{\d t_3\d t_2}{t_2^2t_3^2}\cr
& =\frac{1}{\zeta(2)}\Big( \frac{q}{\phi(q)}-\log h\Big)\!\!\!\!\!\!\!\!\!\!\!\!\!\!\!
 \sum_{\substack{g\leq T\\   h/T<gYZ\leq hT,\, Y\geq  Z }}\frac{\mu(g)     \phi_2(g)   }{ \phi(g)  \phi_{-1}( g)  }
 \cr&\qquad\qquad\qquad\qquad\qquad
 \sum_{\substack{1\leq |a_1|\leq  TQ_1/h\\ (a_1,g)=1}}  \frac{ \phi_1(a_1 ) }{  \phi_{-1}(a_1)  }
\int_{1/2}^{1}  \int_{1/2gYZ}^{1/gYZ}  |E_h (  {a_1u_2u_3} ) |^2 {\d u_3\d u_2} .\end{align*}
Des manipulations semblables \`a celles du sixi\`eme cas permettent de se restreindre \`a 
 $h/T<gYZ\leq hT,$ et $h^{1/2}\leq Y\leq h$, puis \`a $|a_1|\leq T^2.$

 Nous obtenons alors un terme principal \'egal \`a 
\begin{align*}& =\frac{1}{\zeta(2)}\Big( \frac{q}{\phi(q)}-\log h\Big)\!\!\!\!
 \sum_{\substack{g\leq T\\  h^{1/2}\leq Y\leq h }}\frac{\mu(g)     \phi_2(g)  }{ \phi(g)  \phi_{-1}(g)  }\sum_{\substack{1\leq |a_1|\leq  T^2\\ (a_1,g)=1}}  
 \frac{ \phi_1(a_1 ) }{  \phi_{-1}(a_1)  }
\int_{1/2}^{1}  \int_{1/2hT}^{T/h}  |E_h (  {a_1u_2u_3} ) |^2 {\d u_3\d u_2} 
\cr
& =\frac{1}{\zeta(2)}\Big( \frac{q}{\phi(q)}-\log h\Big)\!\!
 \sum_{\substack{1\leq |a_1|\leq  T^2\\  h^{1/2}\leq Y\leq h }}
 \frac{ \phi_1(a_1 ) }{  \phi_{-1}(a_1)  }\sum_{\substack{g\leq T \\ (g,a_1)=1}}  \frac{\mu(g)     \phi_2(g)  }{ \phi(g)  \phi_{-1}(g)  }
\int_{1/2}^{1}  \int_{1/2hT}^{T/h}  |E_h (  {a_1u_2u_3} ) |^2 {\d u_3\d u_2}. \end{align*}
Nous utilisons alors
$$\sum_{\substack{ 1\leq g\leq T \\ (a_1,g)=1}}\frac{\mu(g)  \phi_{2}(g)}{\phi(g)  \phi_{-1}( g)}=-\sum_{\substack{   g>T \\ (a_1,g)=1}}\frac{\mu(g)  \phi_{2}(g)}{\phi(g) \tilde\phi( g)} \ll \psi_{1/2}(a_1)^2{\cal L}(T)^{-2}.  $$
ce qui fournit une contribution 
$$\ll h (\log h)^2(\log T){\cal L}(T)^{-2}.$$
En rassemblant ces r\'esultats, nous obtenons bien
$$S_7(h,T;\delta)\ll h (\log h)^4  {\cal L} (T)^{-1}+ h (\log h) (\log T)^2.$$
\end{proof}
 
\subsection{Conclusion de la d\'emonstration du Th\'eor\`eme \ref{mainth}.}
Nous partons de la formule \eqref{estR3h}. Nous devons donc estimer $V_3(h)$. Le Lemme 
\ref{lemme estS1} coupl\'e avec le Lemme \ref{LemmeestC1}     permet de traiter le cas o\`u    $gyz>hT$
et $gx>hT.$  Le Lemme \ref{lemme estS2} permet de n\'egliger la contribution du  cas o\`u    $gyz\leq h/T$
et $gx>hT.$  Le Lemme \ref{lemme estS7} le compl\`ete en majorant  la contribution du  cas o\`u    $h/T<gyz\leq hT$
et $gx >hT.$ 

Nous pouvons donc consid\'erer le cas restant $gx\leq hT.$ Le Lemme \ref{lemme estS3} estime la contribution du cas $gyz>hT$, $x>zT^4$ sachant que gr\^ace au Lemme \ref{lemme estS4} le cas $x\leq zT^4$ est n\'egligeable. Le Lemme  \ref{lemme estS6} qui traite du cas  $h/T<gyz\leq hT$, $gx\leq hT$ permet de conclure.

La somme des termes d'erreur apparaissant dans ces lemmes est major\'ee par
$$
\ll h (\log h) (\log_2 h)^2(\log T)^5+\frac{h (\log h)^{12}}{  {\cal L} (T)}+\frac{h(\log h)^{14} }{T} .$$
En choisissant $T=\exp\{  150(\log_2h)^2\}$, nous obtenons alors le   terme d'erreur annonc\'e au Th\'eor\`eme~\ref{mainth}.

\section{D\'emonstration du Th\'eor\`eme \ref{consth}} 
L'hypoth\`ese \ref{hypo} fournit lorsque $1\leq k\leq 3$ et $1\leq x\leq X$ l'estimation
\begin{equation}
    \label{hypo0}
\sum_{1\leq n\leq x} \prod_{j=1}^k \Lambda_0(n+d_j)= \mathfrak S_0(\mathcal D) x+O\big(E_3(X;h)\big). \end{equation}
%Nous rappelons la formule $\mathfrak S_0(\{ d\})=0.$ 
    Nous avons  
$$M_3(X,h):=\frac1X\sum_{1\leq n\leq X} \Big(\sum_{1\leq d\leq h}\Lambda_0(n+d)\Big)^3.$$
En d\'eveloppant le cube, nous obtenons
\begin{align*}
    M_3(X,h)&=\sum_{\substack{ 1\leq d_1,d_2,d_3\leq h\\ d_j \rm{distincts}}}\frac1X \sum_{1\leq n\leq X} \Lambda_0(n+d_1)\Lambda_0(n+d_2)\Lambda_0(n+d_3)\cr&\quad+ 
    \sum_{\substack{ 1\leq d_1,d_2 \leq h\\ d_1\neq d_2}} \frac3X \sum_{1\leq n\leq X} \Lambda_0(n+d_1)^2\Lambda_0(n+d_2) +
  \sum_{\substack{ 1\leq d \leq h }}\frac1X  \sum_{1\leq n\leq X}  \Lambda_0(n+d)^3.
\end{align*}
Nous notons $M_{3,j}$ la $j$-i\`eme double somme apparaissant dans le d\'eveloppement de $ M_3(X,h)$. 
D'apr\`es l'hypoth\`ese \eqref{hypo} et le Th\'eor\`eme \ref{mainth},
\begin{align*}
    M_{3,1} 
&=   R_3(h)+O\big(h^3E_3(X,h)/X\big)
\cr&=  %\textcolor{red}
{ \frac 94}  h(\log h)^2\Big(1+O\Big( \frac{(\log_2 h)^{12}}{\log h}\Big)\Big)+O\big(h^3E_3(X,h)/X\big).
\end{align*} %\textcolor{red}
{La valeur $\tfrac94$ correspond \`a celle conjectur\'ee par Kuperberg \cite{K21}[section 5] apr\`es des calculs num\'eriques. }

Nous suivons \cite{MS04} pour estimer $M_{3,2}$, puis $M_{3,3}$.
Nous observons que 
\begin{align*}
    \Lambda_0(n+d_1)^2&=
     \Lambda(n+d_1)\Lambda_0(n+d_1) -\Lambda_0(n+d_1)
     \cr&=\Lambda(n+d_1)(\log (n+d_1) -1)-\Lambda_0(n+d_1)+O\big( 
     \Lambda(n+d_1)1_{n+d_1\notin {\cal P}}\log X\big)
     \cr&=\Lambda_0(n+d_1)(\log (n+d_1) -2)+\log (n+d_1) -1+O\big(  
     \Lambda(n+d_1)1_{n+d_1\notin {\cal P}}\log X\big).
\end{align*}
La contribution \`a $M_{3,2} $ du dernier terme est $\ll h^2(\log X)^2/\sqrt{X} $. Comme $\mathfrak S_0(\{ d_2\})=0,$  une sommation d'Abel permet de montrer que la  contribution \`a $M_{3,2} $ du terme $\log (n+d_1) -1$ est
$\ll  h(\log X)E_3(X,h)/X.$ La  contribution \`a $M_{3,2} $ du premier terme
est $$= \frac{R_2(h)}X\int_0^X (\log x-2)\d x+O\big(   h^2(\log X)E_3(X,h)/X\big) .$$
D'apr\`es (16) de \cite{MS04}, lorsque $B$  est la constante d\'efinie dans l'\'enonc\'e, pour tout $\varepsilon>0 $, nous avons
\begin{equation}
\label{estR2}    
R_2(h)=-h(\log h-B-1)+O(h^{1/2+\varepsilon}).
\end{equation}
Il vient 
\begin{align*}
 M_{3,2}&%= -\frac hX(\log h-B-1)\int_0^X (\log x-2)\d x+O\big(  h^2(\log X) E_3(X,h)+h^{1/2+\varepsilon}X\log X\big)\cr&
 = -\frac hX(\log h-B-1)\int_0^X (\log x-2)\d x+O\big(  h^2(\log X) E_3(X,h)/X +h^{1/2+\varepsilon} \log X \big) .\end{align*}

De m\^eme, nous observons que \begin{align*}
    \Lambda_0(n+d)^3%&= \Lambda(n+d)\Lambda_0(n+d)^2 -\Lambda_0(n+d)^2\cr
     &=
     \Lambda(n+d)(\Lambda_0(n+d)^2-\Lambda_0(n+d)) +\Lambda_0(n+d) 
     \cr&=\Lambda(n+d)(\log (n+d) -1)(\log (n+d)-2)+\Lambda_0(n+d)\cr&\quad+O\big( 
     \Lambda(n+d)1_{n+d\notin {\cal P}}(\log X)^2\big)
     \cr&=\Lambda_0(n+d)\big((\log (n+d) -1)(\log (n+d)-2)+1\big)\cr&\quad +(\log (n+d) -1)(\log (n+d)-2)+O\big( 
     \Lambda(n+d)1_{n+d\notin {\cal P}}(\log X)^2\big).
\end{align*}
La contribution \`a $M_{3,3} $ du dernier terme est $\ll h^2 (\log X)^2/\sqrt{X}$. Comme $\mathfrak S(\{ d \})=0,$ la contribution \`a $M_{3,3} $ du premier terme est $\ll h(\log X)^2E_3(X,h)/X$. 
Ainsi
\begin{align*}M_{3,3} 
&=\frac hX\int_0^X (\log x -1)(\log x -2 )\d x +O\big(  h(\log X)^2  E_3(X,h)/X +h^2(\log X)^2/X\big).
\end{align*}
Cela implique bien 
\begin{align*}
     M_3(X,h)= &\frac hX\int_0^X \big( %\textcolor{red}
     { \tfrac 94 }(\log h)^2-3(\log h-B)(\log x) +(\log x)^2\big)\d x  
\cr&+O\big(  h(\log h)(\log_2 h)^{12}+ h (h+ \log X)^2  E_3(X,h) /X+h^{1/2+\varepsilon} \log X \big).
%\cr= &h \big( \textcolor{red}{3%\tfrac 94}(\log h)^2-3(\log h-B)(\log X) +(\log X)^2\big) 
%\cr&+O\big(  h(\log h)(\log_2 h)^{14}+ h (h+ \log X)^2  E_3(X,h) /X+h^{1/2+\varepsilon} \log X \big).
\end{align*}

\section{D\'emonstration du Th\'eor\`eme \ref{conjMoment}} 

Nous suivons la m\'ethode de Montgomery et Soundararajan \cite{MS04} en nous concentrant sur les moments d'ordre impair $K=2K'+1$ avec $K'\geq 1.$ \`A chaque \'etape, nous reprenons leurs notations sauf que leur $H$ est not\'e $h$, et leur $h$ sera not\'e $i.$ Nous n'indiquons pas la d\'ependance en $K$ des estimations.

Suivant \cite{MS04}*{eq. (61)}, nous avons
\begin{equation}
    \label{devMK}
M_K(X,h)=\frac1X
\sum_{k=1}^K \sum_{\substack{M_1,\ldots, M_k\\ M_j\geq 1\\ \sum_j M_j=K}}\binom{K}{M_1,\ldots, M_k}\frac{1}{k!}
\sum_{\substack{d_1,\ldots,d_k\\ 1\leq d_j\leq h\\ d_j \rm{distincts}}}\sum_{1\leq n\leq X} \prod_{j=1}^k \Lambda_0(n+d_j)^{M_j}
,\end{equation}
avec
$$\binom{K}{M_1,\ldots, M_k}:=\frac{K!}{M_1!\cdots  M_k! }.$$

Lorsque $m\geq 1$, nous notons $\Lambda_m(n)=\Lambda(n)^{m}\Lambda_0(n)$. Ainsi lorsque $M\geq 1,$ nous avons
$$\Lambda_0(n)^M=\Lambda_0(n)(\Lambda(n)-1)^{M-1}
=\sum_{m=0}^{M-1} (-1)^{M-m-1}\binom{M-1}{m}\Lambda_m(n).
$$

En reportant dans \eqref{devMK}, nous obtenons
\begin{equation}\label{Mk=sumLk}\begin{split}
    M_K(X,h)=&
\sum_{k=1}^K\frac{1}{k!} \sum_{\substack{M_1,\ldots, M_k\\ M_j\geq 1\\ \sum_j M_j=K}}\binom{K}{M_1,\ldots, M_k}
\cr&\times \sum_{\substack{m_1,\ldots, m_k\\ 0\leq m_j\leq M_j-1}}
\prod_{j=1}^k(-1)^{M_j-m_j-1}\binom{M_j-1}{m_j}L_k(\bfm),
\end{split}
\end{equation}
o\`u 
$$L_k(\bfm):=\frac1X\sum_{\substack{d_1,\ldots,d_k\\ 1\leq d_j\leq h\\ d_j \rm{distincts}}}\sum_{1\leq n\leq X} \prod_{j=1}^k \Lambda_{m_j}(n+d_j).$$
Pour chaque $\bfm,$ nous introduisons quatre ensembles
$${\cal K}=\{ 1,\ldots, k\}, \quad  {\cal H}=\{ j\in {\cal K}:m_j\geq 1\}, \quad  {\cal I}=\{ i\in {\cal K}:m_i=0\}, \quad  {\cal J}\subset {\cal K} $$ et trois cardinaux associ\'es
$i={\cal H},$ $k-i={\cal I},$ $j={\cal J}.$
Ainsi
\begin{align*}
    \prod_{j\in {\cal I}} \Lambda_{0}(n+d_j) \prod_{j\in {\cal H}} \Lambda(n+d_j)&=
     \prod_{i\in {\cal I}} \Lambda_{0}(n+d_i)
    \prod_{j\in {\cal H}} \big(\Lambda_{0}(n+d_j)+1\big)=
    \sum_{\substack{{\cal J}\\ {\cal I}\subset {\cal J}\subset {\cal K}}}\prod_{j\in {\cal J}} \Lambda_{0}(n+d_j).
\end{align*}
L'hypoth\`ese de Hardy et Littlewood fournit lorsque $1\leq x\leq X$
$$ \sum_{n\leq x}
    \prod_{j\in {\cal I}} \Lambda_{0}(n+d_j) \prod_{j\in {\cal H}} \big(\Lambda_{0}(n+d_j)+1\big) 
    =x\sum_{\substack{{\cal J}\\ {\cal I}\subset {\cal J}\subset {\cal K}}}\mathfrak S_0({\cal D}_{\cal J})
    +O( E_K(X,h))
    $$
avec ${\cal D}_{\cal J}:=\{ d_j \quad (j\in  {\cal J})\}.$
Lorsque $\Lambda(n+d_j)$ est non nul avec $j\in {\cal H}$ (c'est-\`a-dire $m_j\geq 1$), soit $n+d_j$ est premier et alors le facteur 
$\Lambda(n+d_j)^{m_j-1}\Lambda_0(n+d_j)=
\log(n+d_j)^{m_j-1}(\log(n+d_j)-1)$, soit $n+d_j$ est une puisssance sup\'erieure \`a $2$ de nombre premier et alors la contribution de ce terme est n\'egligeable puisque $E_K(X,h)\geq \sqrt{X} $. Nous avons 
$$\sum_{j=1}^km_j\leq  \sum_{j=1}^k(M_j-1)\leq K-k.$$
En faisant une int\'egration par parties et en approchant $\log (n+d_j) $
par $\log n$, Montgomery et Soundararajan obtiennent 
\begin{align*}\frac1X\sum_{1\leq n\leq X} \prod_{j=1}^k \Lambda_{m_j}(n+d_j)
&=\sum_{\substack{{\cal J}\\ {\cal I}\subset {\cal J}\subset {\cal K}}}\mathfrak S_0({\cal D}_{\cal J})\big( I(X,\bfm)+O\big((\log X)^{K-k}(h +  E_K(X,h))/X\big),\end{align*}
o\`u 
$$I(X,\bfm):=\frac1X\int_0^X\prod_{j\in {\cal H}}\big( (\log x)^{m_j-1}(\log x-1)\big)\d x. $$
Suivant \cite{MS04}, nous obtenons gr\^ace \`a $h\leq X^{1/k}$
\begin{align*}&L_k(\bfm) =I(X,\bfm)
\sum_{\substack{{\cal J}\\ {\cal I}\subset {\cal J}\subset {\cal K}}} R_j(h)h^{k-j}(1+O(1/h))+O\big( h^k(\log X)^{K-k}E_K(X,h)/X\big)\cr& =I(X,\bfm)h^{i-1}\big( h R_{k-i}(h)(1+O(1/h))+iR_{k-i+1}(h)\big)+O\big( h^k(\log X)^{K-k}E_K(X,h)/X\big).  
\end{align*}
Le facteur $i$ devant $R_{k-i+1}(h)$ est le nombre de mani\`ere de choisir ${\cal J}$ contenant ${\cal I}$ avec un \'element de ${\cal H}$ en plus.

{Nous supposons maintenant $h\geq \log X$.}
Nos hypoth\`eses fournissent pour tout $k\geq 1$
$$R_k(h)\ll h^{\lfloor k/2\rfloor}(\log h)^{\lceil k/2\rceil}.$$
Puisque $\lfloor (k-i)/2\rfloor+ \lceil (k+i)/2\rceil=k$, nous avons donc 
\begin{align*} 
L_k(\bfm)&\ll   (\log X)^{K-k}h^{\lfloor (k-i)/2\rfloor}(\log h)^{\lceil (k+i)/2\rceil} +h^{K}E_K(X,h)/X\cr 
&\ll   (\log X)^{K }\Big(\frac{h}{\log X}\Big)^{\lfloor (k+i)/2\rfloor}\Big(\frac{\log h}{\log X}\Big)^{\lceil (k-i)/2\rceil} +h^{K}E_K(X,h)/X
\cr 
&\ll   (\log X)^{K }\Big(\frac{h}{\log X}\Big)^{\lfloor (k+i)/2\rfloor} +h^{K}E_K(X,h)/X.
\end{align*}

Comme dans \cite{MS04}, nous remarquons que 
$k+i\leq k-i+\sum_{j}( M_j-1)\leq K.$
Il est ais\'e de n\'egliger la contribution des $i$ tels que $k+i<K-1 $. Le cas $k+i= K $ correspond au cas $M_j=2$ pour tout $j\in {\cal H}$ tandis que le cas $k+i= K-1 $ correspond au cas $M_j=2$ pour tout $j\in {\cal H}$ sauf pour une valeur o\`u $M_j=3.$

%{\color{blue} Dans le cas $K=3$ o\`u le cas $k=1$ et $i=1$ fournit un terme  $\sim Xh(\log X)^2 $ lorsque $X$ et $h$ tendent vers l'infini. }

Nous nous pla\c cons dans le cas o\`u $k+i=K $, $M_j-1=m_j=1$ pour tout $j\in {\cal H}.$ Ainsi nous avons $2i=2(K-k)\leq K$ autrement dit $\tfrac 12 K\leq k\leq K.$ Dans ce cas, nous avons aussi
$$I(X,\bfm) =\frac1X\int_0^X (\log x-1)^{K-k}\d x=:I_{K-k}(X)= (\log X)^{K-k}\big(1+O(1/\log X)\big).$$
Notant $k'=k-K'-1$, nous obtenons
\begin{align*} L_k(\bfm) &  =I_{K-k}(X)h^{K-k-1}\big( h R_{2k-K}(h)(1+O(1/h))+(K-k)R_{2k-K+1}(h)\big)\cr&\quad +O\big( h^k(\log X)^{K-k}E_K(X,h)\big)\cr
&   =(-1)^{k'+1}Xh^{K'}(\log X)^{K'-k'}(\log h)^{k'+1}\big( r_{2k'+1}   +(K'-k')\mu_{2k'+2}+o(1) \big)\cr&\quad +O\big( h^k(\log X)^{K-k}E_K(X,h)\big).\end{align*}

%{\color{red}changement}
Nous nous pla\c cons dans le cas o\`u $k+i=K-1=2K' $, $M_j-1=m_j=1$ pour tout $j\in {\cal H} $ sauf peut \^etre un exception o\`u $M_j=3.$
Nous   avons $2i+1=2(K-k-1)+1\leq K$ et $k+2\leq K$ autrement dit $\tfrac 12 (K-1)\leq k\leq K-2.$ Nous pouvons donc choisir $k=K'+k'$ avec $0\leq k'\leq K'-1.$ Notons que cela implique $K'\geq 1.$

Dans ce cas-l\`a, nous avons encore 
$ I(X,\bfm)  \sim  (\log X)^{K-k}$   
et \begin{align*} L_k(\bfm)   &=I(X,\bfm)h^{i}  R_{k-i}(h)(1+O((\log h)/h)) +O\big( h^k(\log X)^{K-k}E_K(X,h)/X\big)
\cr&=(-1)^{k'}\mu_{2k'} (\log X)^{K'+1-k'}h^{K'}(\log h)^{k'}(1+o(1))
+O\big( h^k(\log X)^{K-k}E_K(X,h)/X\big).
\end{align*}

En reportant dans \eqref{Mk=sumLk}, nous obtenons 
\begin{equation} \begin{split}
     &\frac{M_K   (X,h)}{ h^{K'}} \cr  = &  - (\log h)
\sum_{0\leq k'\leq K'} (\log X)^{K'-k'}(-\log h)^{k' }\frac{(2K'+1)!}{(K'+1+k')!2^{K'-k'}}  \binom{K'+1+k'}{K'-k'}\cr&\qquad
\times \big( r_{2k'+1}   + (K'-k')\mu_{2k'+2}+o(1) \big)\cr
& +  (\log X)
\sum_{0\leq k'\leq K'-1} (\log X)^{K'-k'}(-\log h)^{k' }\frac{(2K'+1)!(\mu_{2k' }+o(1))}{3(K'+k')!2^{K'-k'}}  \binom{K'+ k'}{K'-k'} (K'-k')\cr&%+Xh(\log X)^2 
+  O\big(  h \big\{ h^{K'}+(\log X)^{K'}\big\}E_K(X,h)/X\big).
\end{split}
\end{equation}
Ici le coefficient 
$  \binom{k}{K -k}=\binom{K'+1+k'}{K' -k'} $ 
permet de compter le nombre de sous-ensembles ${\cal I}$ de 
$\{1,\ldots,k\}$ de cardinal $k-i=2k-K=k-(K-k) $ tandis que 
le coefficient 
$  \binom{k}{K-1 -k}=\binom{K' +k'}{K' -k'} $ 
permet de compter le nombre de sous-ensembles ${\cal I}$ de 
$\{1,\ldots,k\}$ de cardinal $k-i=2k+1-K=k-(K-1 -k) $.
Dans le second cas, il faut choisir l'\'el\'ement $j\in{\cal H}$ tel que $M_j=3$. Il y en a $i=K-k-1=K'-k'$ mani\`eres possibles.

Nous observons facilement les formules
$$
\frac{(2K'+1)!}{(K'+1+k')!2^{K'-k'}}  \binom{K'+1+k'}{K'-k'}  \big( r_{2k'+1}   + (K'-k')\mu_{2k'+2}  \big)=\mu_{2K'+2}(K'+\tfrac12 k')\binom{K' }{ k'} 
$$
et
$$ \frac{\mu_{2k' }}{(K'+k')!} \binom{K}{M_1,\ldots, M_k} \binom{K'+ k'}{K'+1-k'}=\frac{(2K'+1)!\mu_{2k' }}{3(K'+k')!2^{K'-k'}}  \binom{K'+ k'}{K'+1-k'}=\tfrac 13\mu_{2K'+2} \binom{K' }{ k'} .$$
De plus 
$$\sum_{0\leq k'\leq K'}(K'+\tfrac12 k')\binom{K' }{ k'} 
(\log X)^{K'-k'}(-\log h)^{k' }=K'\big(\log (X/h)\big)^{K'-1}\log (X/h^{3/2})
$$
et 
$$\sum_{0\leq k'\leq K'-1}(K'-k')\binom{K' }{ k'} 
(\log X)^{K'-k'}(-\log h)^{k' }
=K'\big(\log (X/h)\big)^{K'-1}(\log X).$$
Nous en d\'eduisons 
\begin{equation}\label{MK} \begin{split}
     &\frac{M_K (X,h)}{ h^{K'}} 
     % =  -\mu_{2K'+2}K' (\log h) \big(\log (X/h)\big)^{K'-1}\log (X/h^{3/2})   \cr
%& \qquad\qquad+(\tfrac 13+o(1))\mu_{2K'+2} K'  (\log X)^2
% \big(\log (X/h)\big)^{K'-1 } 
%\cr&\qquad\qquad+  O\big(  h \big\{ h^{K'}+(\log X)^{K'}\big\}E_K(X,h)/X\big)
\cr
&\qquad = (\tfrac 13+o(1))\mu_{2K'+2}K'  
 \big(\log (X/h)\big)^{K'-1 } \big( %\textcolor{red}
 {\tfrac 94
 }(\log h)^2-3(\log h)(\log X) +(\log X)^2\big)
\cr&\qquad\quad+  O\big(  h \big\{ h^{K'}+(\log X)^{K'}\big\}E_K(X,h)/X\big)
\cr
&\qquad = (\tfrac 13+o(1))\mu_{2K'+2}K'  
 \big(\log (X/h)\big)^{K'-1 } \big(  \log (X/h^{3/2})\big)^2
\cr&\qquad\quad+  O\big(  h \big\{ h^{K'}+(\log X)^{K'}\big\}E_K(X,h)/X\big).\end{split}
\end{equation}

\section{Remerciements}
L'auteur a plaisir \`a remercier chaleureusement Vivian Kuperberg pour les discussions sur le sujet qu'il a eues avec elle %\textcolor{red}
{et pour lui avoir signal\'e une erreur de calcul d\'etect\'ee \`a l'aide de ChatGPT,} Daniel Fiorilli et Florent Jouve pour leurs remarques et encouragements, G\'erald Tenenbaum pour sa relecture. L'auteur remercie l'Institut Mittag Leffler pour les excellentes conditions de travail lors de la fin de l'\'elaboration de ce travail et aussi les organisateurs pour leur invitation. Pendant cette p\'eriode, l'auteur a \'et\'e soutenu par une chaire senior de l'Institut Universitaire de France.

\noindent Institut de Math\'ematiques de Jussieu-Paris Rive Gauche\\
Universit\'e   Paris Cit\'e, Sorbonne Universit\'e, CNRS UMR 7586\\
Institut Universitaire de France\\
Case Postale 7012\\
F-75251 Paris CEDEX 13, France\\
{regis.delabreteche@imj-prg.fr}

\end{document}